\documentclass{article}

\usepackage{docmute}
\usepackage{hyperref}  
\usepackage{amssymb, amsbsy, amsmath}  
\usepackage{mathtools}  
\usepackage{xfrac}  
\usepackage{xspace}  
\usepackage{float}  
\usepackage{graphicx}  
\usepackage{subcaption}  
\usepackage{natbib}  
\usepackage[capitalise]{cleveref}  
\usepackage{appendix}  
\usepackage{enumitem}  
\usepackage{multicol}  
\usepackage[noend]{algpseudocode}  
\usepackage{algorithm}  
\usepackage{acro}  
\usepackage{xr}  
\usepackage{bm} 
\usepackage{tabularx}  
\usepackage{etoolbox}  
\usepackage{booktabs}
\usepackage{multirow}

\hypersetup{colorlinks=true, allcolors=blue}

\makeatletter
 
  \@addtoreset{equation}{section}
 \makeatother
\usepackage{bm}
\usepackage{amsfonts}
\usepackage{amssymb}
\usepackage{mathrsfs}
\usepackage{amsmath,amsthm}
\usepackage{authblk}

\usepackage{comment}
\usepackage{graphicx}
\usepackage{color}
\usepackage{indentfirst}
\usepackage{here}
\usepackage{natbib}

\theoremstyle{plain}  
\newtheorem{theorem}{Theorem}[section]
\newtheorem{lemma}{Lemma}[section]
\newtheorem{prop}{Proposition}[section]

\newtheorem{rem}{Remark}[section]

\newtheorem{example}{Example}

\newcommand{\trueparam}{\theta^\dagger}
\newcommand{\truebeta}{\beta^\dagger}
\newcommand{\truesigma}{\sigma^\dagger}

\newcommand{\truedist}{\nu_{\trueparam}}

\newcommand{\mrm}{\mathrm} 
  
\newcommand{\limit}{$n \to \infty$, $\Delta_n \to 0$ and $n \Delta_n \to \infty$}

\newcommand*{\sample}[2]{{#1}_{t_{#2}}}
\newcommand*{\obs}[3]{{#1}_{t_{#2}}^{(\mrm{#3})}}

\newcommand{\probconv}{\xrightarrow{\mathbb{P}_{\trueparam}}}
\newcommand{\distconv}{\xrightarrow{\mathcal{L}_{\trueparam}}}

\providecommand{\keywords}[1]{\textit{Keywords:} #1}

\everymath {\displaystyle}

\newcommand{\ruby}[2]{
\leavevmode
\setbox0=\hbox{#1}
\setbox1=\hbox{\tiny #2}
\ifdim\wd0>\wd1 \dimen0=\wd0 \else \dimen0=\wd1 \fi
\hbox{
\kanjiskip=0pt plus 2fil
\xkanjiskip=0pt plus 2fil
\vbox{
\hbox to \dimen0{
\small \hfil#2\hfil}
\nointerlineskip
\hbox to \dimen0{\mathstrut\hfil#1\hfil}}}}
\allowdisplaybreaks[1]

\renewcommand{\theenumi}{$[\mathrm{H}\arabic{enumi}]$}
\renewcommand{\labelenumi}{\theenumi}

\makeatletter
\newcommand*{\addFileDependency}[1]{
\typeout{(#1)}
\@addtofilelist{#1}
\IfFileExists{#1}{}{\typeout{No file #1.}}
}\makeatother

\title{Parameter Inference for Hypo-Elliptic Diffusions under a Weak Design Condition} 

\author[1]{Yuga Iguchi}
\author[1]{Alexandros Beskos}

\affil[1]{Department of Statistical Science, University College London, London, UK}

\begin{document}
\maketitle

\begin{abstract}
We address the problem of parameter estimation for degenerate diffusion processes defined via the solution of Stochastic Differential Equations (SDEs) with diffusion matrix that is not full-rank. For this class of hypo-elliptic SDEs recent works have proposed contrast estimators that are asymptotically normal, provided that the step-size in-between observations $\Delta=\Delta_n$ and their total number $n$ satisfy $n \to \infty$, $n \Delta_n \to \infty$, $\Delta_n \to 0$, and additionally $\Delta_n = o (n^{-1/2})$. 
This latter restriction places a requirement for a so-called  `rapidly increasing experimental design'. In this paper, we relax this limitation and develop a general contrast estimator satisfying asymptotic normality under the weaker condition $\Delta_n = o(n^{-1/p})$ for general $p \ge 2$. 
Such a result has been obtained for elliptic SDEs, 
but its derivation in a hypo-elliptic setting is highly non-trivial and requires a separate approach. We provide numerical results to illustrate the advantages of the developed theory.  
\\ 

\noindent
\end{abstract}



\section{Introduction} 
\label{sec:intro}
We consider the problem of parameter inference for multivariate Stochastic Differential Equations (SDEs) determined via a degenerate  diffusion matrix. Let $(\Omega, \mathcal{F}, \{\mathcal{F}_t \}_{t \geq 0}, \mathbb{P})$ be a filtered probability space and let $W = (W_{t}^1, \ldots, W_{t}^d)$, $t \geq 0$, be the $d$-dimensional standard Brownian motion, $d\ge 1$, defined on such a  space. We introduce $N$-dimensional SDE models, $N\ge 1$: 
%
\begin{align} \label{eq:sde_1}
   d X_t = \mu (X_t, \theta) dt + \sum_{j = 1}^d A_j (X_t, \theta) d W_{t}^j, \qquad  X_0 = x \in \mathbb{R}^N,
\end{align}
with $\mu, A_j (\cdot, \theta) : \mathbb{R}^N \to \mathbb{R}^N$, $1 \le j \le d$, for parameter $\theta$. This work focuses on the case where the diffusion matrix $A = [A_1, \ldots, A_d]$ is not of full-rank, i.e.~the matrix $a = AA^\top$ is not positive definite. The law of the solution to (\ref{eq:sde_1}) is assumed to be absolutely continuous with respect to (w.r.t.) the Lebesgue measure. Such models are referred to as \emph{hypo-elliptic diffusions} and constitute an important class of 
Markovian processes used in a wide range of applications. E.g., this class includes the (underdamped) Langevin equation, the Synaptic-Conductance model \citep{dit:19}, 
the $\mathrm{NLCAR}(p)$ model \citep{ts:00, str:07}, the Jansen-Rit neural mass model \citep{buc:20}, the 
SIR epidemiological model \citep{sir:22} and the quasi-Markovian generalised Langevin equations ($q$-GLE) \citep{vr:22}. We investigate two (sub-)classes of hypo-elliptic SDEs which cover a wide range of models, including the ones mentioned in the above examples. The first model class is specified via the following SDE: 
\begin{align}
\begin{aligned} \label{eq:hypo-I}  
d X_t & = 
 \begin{bmatrix}
 d X_{S,t} \\ 
 d X_{R,t} 
 \end{bmatrix}
 = 
 \begin{bmatrix}
  \mu_{S} ( X_t,  \beta_S)  \\ 
  \mu_{R} ( X_t,  \beta_R) 
 \end{bmatrix} dt
 + 
 \sum_{j = 1}^d 
 \begin{bmatrix}
 \mathbf{0}_{N_S} \\
 A_{R, j} (X_t, \sigma)
 \end{bmatrix} d W_{t}^j, \\
 X_0 & = x \equiv
 \begin{bmatrix}
 x_S \\ 
 x_R 
 \end{bmatrix}
\in \mathbb{R}^N. 
\end{aligned} \tag{Hypo-I}  
\end{align} 
The involved drift functions and diffusion coefficients are specified as follows: 
\begin{gather*}
 \mu_{S} : \mathbb{R}^N  \times \Theta_{\beta_S} \to \mathbb{R}^{N_S}, \quad   
\mu_{R} : \mathbb{R}^N  \times \Theta_{\beta_R} \to \mathbb{R}^{N_R}, \\ 
 A_{R,j} : \mathbb{R}^N  \times \Theta_{\sigma} \to \mathbb{R}^{N_R},  \quad  1 \le j \le d, 
\end{gather*} 
for positive integers $N_S, N_R$ such that $N = N_S + N_R$. 
Here, the unknown parameter vector is
$$
\theta = (\beta_{S}, \beta_{R}, \sigma) \in \Theta 
= \Theta_{\beta_S} \times \Theta_{\beta_R} \times \Theta_{\sigma}
 \subseteq \mathbb{R}^{N_{\beta_S}} \times \mathbb{R}^{N_{\beta_R}} \times\mathbb{R}^{N_{\sigma}},
$$ 
for positive integers $N_{\beta_S}, \, N_{\beta_R}, \, N_\sigma$ satisfying $N_{\beta_S} + N_{\beta_R} + N_\sigma = N_{\theta}$. The parameter space $\Theta$ is assumed to be a compact subset  of $\mathbb{R}^{N_\theta}$. 
Later on we place a condition on (\ref{eq:hypo-I}),
related to the so-called \emph{H\"ormander's condition},
so that such an SDE gives rise to a hypo-elliptic process, i.e.~its finite-time transition distributions 
admit a Lebesgue density. 
In brief, the condition will guarantee  that randomness from the rough component $X_{R,t}$ propagates into the smooth component $X_{S,t}$ via the drift $\mu_S$. The second  hypo-elliptic model class we treat in this work is specified via the following SDE:  
\begin{align}
\begin{aligned} \label{eq:hypo-II}
\hspace{-0.2cm} d X_t &  
=
\begin{bmatrix}
 d X_{S_1,t} \\ 
 d X_{S_2,t} \\  
 d X_{R,t} 
 \end{bmatrix}
 =  
\begin{bmatrix}
    \mu_{S_1} ( X_{S, t} , \beta_{S_1}) \\
    \mu_{S_2} ( X_{t} , \beta_{S_2}) \\ 
    \mu_{R} ( X_{t}, \beta_{R}) 
\end{bmatrix} dt 
+ \sum_{j = 1}^d 
\begin{bmatrix}
   \mathbf{0}_{N_{S_1}} \\
   \mathbf{0}_{N_{S_2}} \\ 
   A_{R, j} (X_t, \sigma)
\end{bmatrix} d W_{t}^j,  \\ 
X_0 & =  x
\equiv
\begin{bmatrix}
   x_{S_1} \\
   x_{S_2} \\
   x_R
\end{bmatrix} 
\in \mathbb{R}^N, 
\end{aligned}  \tag{Hypo-II}  
\end{align}
where we have set $X_{S, t}
= \bigl[ X_{S_1, t}^\top, X_{S_2, t}^\top \bigr]^\top$.
Also, we now have the parameter vector:
\begin{align*}
\theta = (\beta_{S_1}, \beta_{S_2}, \beta_R, \sigma) \in \Theta &=
\Theta_{\beta_{S_1}} \times \Theta_{\beta_{S_2}} \times \Theta_{\beta_{R}} \times \Theta_\sigma \\  &\qquad \qquad \subseteq 
\mathbb{R}^{N_{\beta_{S_1}}} \times
\mathbb{R}^{N_{\beta_{S_2}}} \times 
\mathbb{R}^{N_{\beta_R}} \times
\mathbb{R}^{N_{\sigma}},
\end{align*}
with integers $N_{\beta_{S_1}}, N_{\beta_{S_2}}$ such that $N_{\beta_{S}} = N_{\beta_{S_1}} + N_{\beta_{S_2}}$, and the drift functions are now specified as: 
\begin{align*} 
   \mu_{S_1} : \mathbb{R}^{N_S} \times \Theta_{N_{\beta_{S_1}}} 
   \to \mathbb{R}^{N_{S_1}}, \quad 
   \mu_{S_2} : \mathbb{R}^{N} \times \Theta_{N_{\beta_{S_2}}} 
   \to \mathbb{R}^{N_{S_2}},
\end{align*}
with integers $N_{S_1}, N_{S_2}\ge 1$ so that $ 
N_{S_1} + N_{S_2} = N_S$. Again, $\Theta$ is assumed to be compact. Note that the drift $\mu_{S_1}$ depends on the smooth component  
$
X_{S, t}
$ and not on $X_{R,t}$, thus randomness from $X_{R,t}$ does not directly propagate onto $X_{S_1, t}$. 
We refer to an SDE of the form (\ref{eq:hypo-II}) as a \emph{highly degenerate diffusion}.
Throughout the paper the class (\ref{eq:hypo-II}) is treated separately from (\ref{eq:hypo-I}). Later on we introduce a H\"ormander-type condition guaranteeing that (\ref{eq:hypo-II}) gives rise to hypo-elliptic SDEs. 
%

We consider parameter estimation for the  classes (\ref{eq:hypo-I}),  (\ref{eq:hypo-II}), given  discrete-time observations of the full vector $X_t$
at instances $t_i = i \Delta_n$, $1\le i \le n$,
for step-size $\Delta_n > 0$. We develop contrast estimators for models in (\ref{eq:hypo-I}), (\ref{eq:hypo-II}), 
and study their asymptotic properties under a 
\emph{high-frequency complete observation} scenario. In particular, we assume the setting $n \to \infty$, $\Delta_n \to 0$, $n \Delta_n \to \infty$. 
Over the last two-three decades, numerous works have studied parametric inference for diffusion processes in a high-frequency setting. 
These works were initially focused on \emph{elliptic} diffusions, i.e.~SDEs in~(\ref{eq:sde_1}) with a positive definite matrix $a = A A^\top$. 
The work in \cite{kess:97} proposed a contrast estimator for scalar elliptic diffusions and proved asymptotic normality (via a CLT) under the design condition $\Delta_n = o (n^{-1/p})$, $p \ge 2$. 
\cite{uchi:12} extended such a result to multivariate elliptic SDEs and proposed an adaptive-type contrast estimator achieving a CLT under the same design condition as in \cite{kess:97}.
In contrast,  hypo-elliptic SDEs have been relatively under-explored until recently, though an interesting empirical study, without analytical results, was provided by \cite{poke:09}. 

To obtain asymptotic results in the hypo-elliptic setting similar to the ones available for elliptic SDEs, one must deal with a number of challenges. A standard Euler-Maruyama discretisation leads to a Dirac measure due to the involved degenerate diffusion matrix. 
This can be resolved  by considering a higher-order It\^o-Taylor expansion for $X_{S, t}$ that propagates additional Gaussian variates onto the smooth components.
For class (\ref{eq:hypo-I})
such a direction leads to a discretisation scheme with Gaussian increments of size $\mathcal{O}(\Delta_n^{3/2})$ 
for the smooth components. 
However, recent works \citep{dit:19, glot:21, igu:23} have highlighted that introduction of high-order Gaussian variates in the smooth components must be accompanied by an appropriate high-order mean approximation of the rough components. If such a balance is not achieved, estimation of $\beta_R$ becomes  biased. 
Within class (\ref{eq:hypo-I}), \cite{dit:19} developed a non-degenerate discretisation scheme in the setting of $N_S = 1$ and of the diffusion matrix $A$ being diagonal. 
\cite{dit:19} proposed contrast estimators separately for $\beta_S$ and $(\beta_R, \sigma)$ 
and proved a CLT under $\Delta_n = o (n^{-1/2})$, 
i.e.~for a `rapidly increasing experimental design' \citep{rao:88}.   
\cite{glot:20, glot:21} addressed the issue of disjoint estimation, within (\ref{eq:hypo-I}), by working with an approximate Gaussian density for the complete vector $X_t$ and providing a joint contrast estimator for the full parameter vector $(\beta_R, \beta_S, \sigma)$.
They showed that the estimator is asymptotically normal under the same design condition $\Delta_n = o (n^{-1/2})$. 
For bivariate  models in class (\ref{eq:hypo-I}), \cite{melnykova2020parametric} used a local linearisation \citep{bis:96} of the drift function to obtain a non-degenerate  Gaussian discretisation scheme and construct a contrast estimator attaining a CLT under $\Delta_n = o (n^{-1/2})$. 
For  class (\ref{eq:hypo-II}), \cite{igu:23}  worked in a general multivariate setting and established a joint contrast estimator achieving a CLT under $\Delta_n = o (n^{-1/2})$. 
The design condition $\Delta_n = o (n^{-1/2})$ assumed in the above works can be restrictive in practice. Indicatively, if a user tries to design a dataset with a large  interval $T_n = n \Delta_n$ and a large number of samples $n$ to obtain accurate estimation results, then 
$\Delta_n$ should be set to a quite small value so that 
condition $n \Delta_n^2 = T_n \Delta_n \to 0$ is satisfied, e.g.~for $T_n = 1000$, $\Delta_n$ could be much less than $10^{-3}$.
Datasets with such  small step-sizes are not always available in applications. 

An apparent open question in the hypo-elliptic setting is the weakening of the design condition $\Delta_n = o (n^{-1/2})$. 
Indicatively, \cite{melnykova2020parametric, glot:20} required $\Delta_n = o (n^{-1/2})$ 
to control terms of size $\mathcal{O}(\Delta_n^{-q})$, $q \ge 1/2$, when proving consistency for their contrast estimators. Thus, to arrive to a CLT under a weaker design condition, one needs to develop a different approach compared to previous works to prove consistency without requiring that $\Delta_n = o(n^{-1/2})$. Furthermore, the construction of a general contrast function for degenerate SDEs is not straightforward due to the degenerate structure of the diffusion matrix. 
For class  (\ref{eq:hypo-I}), \cite{igu:22} developed a contrast estimator achieving the CLT under the weaker condition $\Delta_n = o (n^{-1/3})$ via a novel closed-form transition density expansion for SDEs within (\ref{eq:hypo-I}). However, a general asymptotic theory under the condition $\Delta_n = o (n^{-1/p})$, $p\ge 2$,  
has yet to be established for degenerate SDEs. 

Our work closes the above gap in the research of hypo-elliptic SDEs. We propose a general contrast estimator for a wide class of hypo-elliptic diffusion models and show its asymptotic normality under the weaker design condition $\Delta_n = o (n^{-1/p})$, $p\ge 2$. Specifically, our contributions include the following:  
\begin{description} 
\item[(a)] We develop two contrast estimators for the two model classes (\ref{eq:hypo-I}) and (\ref{eq:hypo-II}), for the purposes of joint estimation of the unknown parameter vector $(\beta_S, \beta_R, \sigma)$. 
The contrast functions are based on approximate log-likelihood terms which we construct by making use of Gaussian approximations with high-order mean and variance expansions,  while at the same time dealing with the degenerate structure of the involved SDEs. 
\item[(b)] We show that within the high-frequency, complete observation regime, a CLT is obtained provided that the design condition $\Delta_n = o (n^{-1/p})$, for \mbox{$p \ge 2$}, is satisfied. 
To the best of our knowledge, this is the first work to define a contrast estimator achieving a CLT under the weak design condition that $n\Delta_n^p\rightarrow 0$, for arbitrarily large $p\ge 2$, in the hypo-elliptic setting. For reference, Table \ref{table:literature} summarises existing works with corresponding design conditions, including our contribution in this setting. 
\item[(c)] We provide numerical experiments to demonstrate that the proposed estimator is asymptotically unbiased in the high-frequency complete observation regime under the weaker design condition $\Delta_n = o (n^{-1/p})$, for $p \ge 3$. The numerical results highlight that the estimators requiring $\Delta_n = o (n^{-1/2})$ proposed in the literature can indeed suffer from bias when $n \Delta_n^2$ is not sufficiently small. 
\item[(d)] Our methodology is relevant beyond the setting of 
high-frequency and/or 
complete observations. The developed Gaussian approximation of the SDE transition density can be used as part of a broader data augmentation algorithm (e.g., MCMC or EM) 
given 
low-frequency and/or 
partial observations. Analytical results for such a setting are beyond the scope of this paper.  
We provide a numerical result showcasing the advantage of the use of the developed approximate log-likelihood under the weaker condition $\Delta_n = o (n^{-1/p})$, $p \ge 3$, in a setting of high-frequency partial observations. 
\end{description} 
To add to point (d) above, 
often in applications only smooth components are observed. 
\cite{dit:19, igu:23} have shown empirically that filtering procedures incorporating 
the developed approximate likelihood of the full state vector (within a data augmentation approach) can lead to asymptotically unbiased parameter estimation in a partial observation regime.

The rest of the paper is organised as follows. Section \ref{sec:pre} prepares some conditions related to the hypo-ellipticity of  models (\ref{eq:hypo-I}) and (\ref{eq:hypo-II}). Section~\ref{sec:main} develops the new contrast estimators and states their asymptotic properties, thus providing the main results of this work. Simulation studies are shown in Section \ref{sec:sim}
(the code is available at \url{https://github.com/YugaIgu/Parameter-estimation-hypo-SDEs}).
The proofs for the main results are collected in Section~\ref{sec:pf}. Section \ref{sec:conclusion} concludes the paper.  
\begin{table}
\caption{Contrast estimators for Hypo-elliptic SDEs (high-frequency, complete observation regime)} 
\label{table:literature}
\centering 
\begin{tabular}{c c c} 
\toprule 
\\[-5pt] 
Work  & Model & 
\begin{tabular}{c}
     Design Condition \\ 
     on $\Delta_n$ for CLT 
\end{tabular} 
\\ 
\midrule  
\cite{dit:19} & 
\begin{tabular} {c}
    (\ref{eq:hypo-I}) with $N_S = 1$. \\ Diffusion matrix is diagonal.
\end{tabular} & $\Delta_n = o (n^{-1/2})$  \\
\hline
\cite{melnykova2020parametric} & (\ref{eq:hypo-I}) with 
$N_S = N_R = 1$. & $\Delta_n =  o (n^{-1/2})$ 
\\ 
\hline
\cite{glot:20, glot:21} & (\ref{eq:hypo-I}) & $\Delta_n = o (n^{-1/2})$ 
\\ 
\hline 
\cite{igu:22} & (\ref{eq:hypo-I}) & $\Delta_n  = o (n^{-1/3})$ 
\\
\hline 
\cite{igu:23} & (\ref{eq:hypo-II}) & $\Delta_n = o (n^{-1/2})$ 
\\
\hline 
This paper & (\ref{eq:hypo-I}) \& (\ref{eq:hypo-II}) & 
$\Delta_n = o (n^{-1/p}), \; p \ge 2$ 
\\ 
\bottomrule   
\end{tabular}
\label{table:nonlin} 
\end{table}

\vspace{0.2cm}\noindent \textbf{Notation.}
\noindent 
For the model class (\ref{eq:hypo-II}), and for 
$x = (x_S, x_R) \in \mathbb{R}^{N_S} \times \mathbb{R}^{N_R} = \mathbb{R}^N$, we write: 
\begin{align*}
\beta_S = (\beta_{S_1}, \beta_{S_2}) \in \Theta_{\beta_S} = \Theta_{\beta_{S_1}} \times \Theta_{\beta_{S_2}}, 
\quad 
\mu_{S} (x, \beta_S) 
= 
\begin{bmatrix}
    \mu_{S_1} (x_S, \beta_{S_1}) \\
    \mu_{S_2} (x, \beta_{S_2})   
\end{bmatrix}. 
\end{align*}
%
We can now use the following notation, for both (\ref{eq:hypo-I}),  (\ref{eq:hypo-II}) for $x  \in \mathbb{R}^N$, $\theta = (\beta_S, \beta_R, \sigma) \in \Theta$: 
\begin{align*}
\mu (x, \theta)  = 
\begin{bmatrix}
    \mu_{S} (x, \beta_S) \\ 
    \mu_{R} (x, \beta_R) 
\end{bmatrix}, 
\ \ 
A_j (x, \theta)  = 
\begin{bmatrix}
    \mathbf{0}_{N_S}^\top \\ 
    A_{R,j} (x, \sigma)
\end{bmatrix}, 
\quad 1 \le j \le d. 
\end{align*} 
%
%
%
%
For a test function $\varphi (\cdot , \theta) : \mathbb{R}^N \to \mathbb{R}$, $\theta \in \Theta$, bounded up to second order derivatives, we introduce differential operators $\mathcal{L}$ and $\mathcal{L}_j, \; 1 \le  j  \le d$, so that for $(x, \theta) \in \mathbb{R}^N \times \Theta$: 
\begin{align}
  \mathcal{L} \varphi (x ,  \theta)  
   &:= \sum_{i=1}^N \mu^i (x, \theta) 
  \frac{\partial \varphi} {\partial x_i}(x, \theta)  + \tfrac{1}{2}  \sum_{i_1, i_2 = 1}^N \sum_{k=1}^d  A_k^{i_1} (x, \theta) A_k^{i_2} (x, \theta)  \frac{\partial^2 \varphi }{\partial x_{i_1} \partial x_{i_2} } (x,\theta);   \label{eq:L_0} \\[0.2cm] 
   \mathcal{L}_j \varphi (x , \theta) 
  &:= \sum_{i=1}^N A_j^i (x, \theta) \frac{\partial \varphi}{\partial x_i}(x , \theta), \quad  1 \le j \le d.  \label{eq:L_j} 
\end{align}  
We denote by $\mathbb{P}_{\theta}$ the probability law of 
$\{ X_{t} \}_{t \geq 0}$ under  $\theta$. We write  
$\probconv$ and $\distconv$ 
to denote convergence in probability and distribution, respectively, under the true parameter $\trueparam$. The latter is assumed to be unique and to lie in the interior of $\Theta$. We denote by $\mathcal{S}$ {the space of functions $f : [0, \infty) \times \mathbb{R}^N \times \Theta \to \mathbb{R}$ so that there are  constants $C, q > 0$ such that 
$| f (\Delta, x, \theta) |  \le C ( 1 + |x|^q) \Delta $,  $(\Delta, x, \theta) \in [0, \infty) \times \mathbb{R}^N \times \Theta$}. 
\textcolor{black}{We denote by $C_p^{K} (\mathbb{R}^n \times \Theta; \mathbb{R}^m)$, $n, m, K \in \mathbb{N}$, the set of functions $f: \mathbb{R}^n \times \Theta \to \mathbb{R}^m$ such that $f$ is continuously differentiable up to order $K$ w.r.t. $x \in \mathbb{R}^n$ for all $\theta \in \Theta$, and $f$ and its derivatives up to order $K$ are of polynomial growth in $x$ uniformly in $\theta \in \Theta$.}
$C_b^\infty (\mathbb{R}^n, \mathbb{R}^m)$ is the set of smooth functions such that the function and its derivatives of any order are bounded. 
For a sufficiently smooth $f: \mathbb{R}^n \to \mathbb{R}$, 
for $\alpha \in  \{1, \ldots,  n \}^l$, $l \ge 1$ and $ 1 \le i \le n$,  we write  
$\textstyle 
\partial_\alpha^u f(u) 
:= \tfrac{\partial^l}{\partial u_{\alpha_1} \cdots \partial u_{\alpha_l}} f (u)$,
$\partial_{u, i} f (u) 
:= \tfrac{\partial}{\partial u_i} f (u).  
$
We write 
$
\partial_u = 
\bigl[
\partial_{u, i}, \ldots, \partial_{u, n} \bigr]^\top, 
\, \partial_u^2 = \partial_u \partial_u^\top, 
$
for the standard differential operators acting upon maps $\mathbb{R}^n \to \mathbb{R}, \, 
n \ge 1$. For a matrix $A$, we write its $(i,j)$-th element as $[A]_{ij}$. 
\section{Model Assumptions and Contrast Estimators}  
\label{sec:pre}
We start by stating some basic assumptions for the classes (\ref{eq:hypo-I}), (\ref{eq:hypo-II}). 
In particular, Section \ref{sec:basic_con} provides H\"ormander-type conditions implying that the SDEs of interest are hypo-elliptic, thus their transition distribution admits a density w.r.t.~the Lebesgue measure. The stated conditions also highlight that classes  (\ref{eq:hypo-I}),  (\ref{eq:hypo-II}) are distinct.
We proceed to develop our estimators in Section \ref{sec:contrast_hypo}. \textcolor{black}{We often denote by $\obs{X}{}{I}$, $\obs{X}{}{II}$ the solutions to SDEs (\ref{eq:hypo-I}), (\ref{eq:hypo-II}), respectively.} 
%
%
\subsection{Conditions for Hypo-Ellipticity}  
\label{sec:basic_con}
%
We first introduce some notation. 
We write the drift function of the Stratonovich-type SDE corresponding to the It\^o-type one, given in (\ref{eq:hypo-I}) or (\ref{eq:hypo-II}), as  
$ \textstyle
A_0 (x, \theta) := \mu (x, \theta) - \tfrac{1}{2} 
\sum_{k = 1}^d \mathcal{L}_k A_k (x, \theta)
$, 
$ 
(x, \theta) \in \mathbb{R}^N  \times \Theta. 
$ 
We often treat vector fields $V:\mathbb{R}^N \to \mathbb{R}^N$ as differential operators via the relation $\textstyle{V \leftrightarrow  \sum_{i = 1}^N V^i \partial_{x, i}}$. For two vector fields $V, W : \mathbb{R}^N \to \mathbb{R}^N$, their Lie bracket is defined as 
$[V, W] = VW - W V$. That is, for the SDE vector fields $A_k, A_l$,  $0 \le k, l\le d$, we have that:  
\begin{align*} 
[A_k, A_l ] (x, \theta) =  \sum_{i = 1}^N A_k^i (x, \theta) \partial_{x, i} A_l (x, \theta)  
- \sum_{i = 1}^N A_l^i (x, \theta) \partial_{x, i} A_k (x, \theta), \qquad   
(x, \theta) \in \mathbb{R}^N \times \Theta. 
\end{align*} 
%
%
%
For $1 \le i \leq j \le N$, we define the projection operator $\mathrm{proj}_{i,j} : \mathbb{R}^N \to \mathbb{R}^{j - i + 1}$ as 
%
$
x = \bigl[x_1,\ldots, x_N\bigr]^{\top} \mapsto \mathrm{proj}_{i, j} (x)
= \bigl[ x_{i}, \ldots, x_j \bigr]^\top
$. 
We impose the following conditions on classes (\ref{eq:hypo-I}) and (\ref{eq:hypo-II}). \\
%
\begin{enumerate}[noitemsep, topsep=1pt, parsep=1pt, partopsep=1pt, leftmargin=30pt] 
\renewcommand{\theenumi}{H\arabic{enumi}}
\renewcommand{\labelenumi}{\theenumi}
\item 
\label{assump:hor}
$(i)$. For class (\ref{eq:hypo-I}), it holds that, for any $(x, \theta) \in \mathbb{R}^N \times \Theta$: 
\begin{align} 
&\mrm{span} \Big\{ A_{R, k} (x, \sigma),
 \,1 \le k \le  d \Big\} =
\mathbb{R}^{N_R}; \nonumber \\ 
& \mrm{span} \Big\{ \big\{\,A_{k} (x,\sigma),\,
[A_{0}, A_k](x,\theta) \, \big\}, \, 1 \le k \le d \Big\} = \mathbb{R}^N. 
\label{eq:span-I} 
\end{align} 
$(ii)$. For class (\ref{eq:hypo-II}),
it holds that, for any $(x, \theta) \in \mathbb{R}^N \times \Theta$:
\begin{align}
&\mrm{span} \Big\{ A_{R,k} (x,\sigma), \, 1 \le 
 k \le d \Big\} = \mathbb{R}^{N_R};
\nonumber
\\[0.2cm]
&
\mrm{span} \Big\{ \mathrm{proj}_{N_{S_{1}}+1,N} \big\{  A_{k} (x,\sigma)  \bigr\}, \,\mathrm{proj}_{N_{S_{1}}+1,N} \big\{\,
[A_{0}, A_k](x,\theta) \big\}, \, 1 \le k\le d \Big\} = \mathbb{R}^{N_{S_2} + N_R }; 
\nonumber
\\[0.2cm]
&
\mrm{span} \Big\{ \big\{\, A_{k} (x,\sigma),\,[A_{0}, A_k] (x,\theta)
\,, 
\bigl[A_{0}, [A_0, A_k]\bigr](x,\theta)\,\big\},\,1\le k\le d\Big\} 
= \mathbb{R}^N. 
\label{eq:span-II}
\end{align}
\textcolor{black}{\item \label{assump:coeff} 
\hspace{-0.3cm} 
$(K)$
$(i)$. For class (\ref{eq:hypo-I}): 
\begin{align*}
    \mu_S,  A_j \in C_p^{K+2} (\mathbb{R}^N \times \Theta \, ; \mathbb{R}^{N_S}), 
    \ \  
    \mu_R \in C_p^{K} (\mathbb{R}^N \times \Theta \, ; \mathbb{R}^{N_R}), \quad  1 \le j \le d.  
\end{align*}
$(ii)$. For class (\ref{eq:hypo-II}):  
%
\begin{gather*}
\mu_{S_1} \in C_p^{K+4} (\mathbb{R}^N \times \Theta \, ; \mathbb{R}^{N_{S_1}}), \ \ 
\mu_{S_2}, A_j \in C_p^{K+2} (\mathbb{R}^N \times \Theta \, ; \mathbb{R}^{N_{S_2}}); \\ 
\mu_R \in C_p^{K} (\mathbb{R}^N \times \Theta \, ; \mathbb{R}^{N_R}),  
\quad 1 \le j \le d. 
\end{gather*}}
\end{enumerate} 
\begin{rem}
We use \ref{assump:hor} to introduce some  structure upon the degenerate SDE models (\ref{eq:hypo-I}) and (\ref{eq:hypo-II}) and so that contrast functions developed later on are well-defined.
\ref{assump:hor} is stronger than {H\"ormander's condition}. The latter states that there is $M \in \mathbb{N}$ so that for any $(x, \theta) \in \mathbb{R}^N \times \Theta$: 
\begin{align*}
\mrm{span} 
\Bigl\{ 
  V (x, \theta)  :  V(x,\theta) \in \bigcup_{1 \le m \le  M} \mathbf{V}_m
 \Bigr\} = \mathbb{R}^N, 
\end{align*}
where we have set
%
$\mathbf{V}_0 = \bigl\{ A_1, \ldots, A_d \bigr\}$,   
$\mathbf{V}_k = \big\{ [A_l, A]  :  A  \in \mathbf{V}_{k-1}, \, 0  \le l \le d \bigr\}$, $k\ge 1$. \\[0.2cm]   
If $\mu (\cdot, \theta), A_j (\cdot, \theta) \in C_b^\infty (\mathbb{R}^N; \mathbb{R}^N)$ then H\"ormander's condition implies that there exists a smooth Lebesgue density for the law of $X_t$, $t>0$, for any initial condition $x \in \mathbb{R}^N$ (see e.g.~\cite{nua:06}). Thus, (\ref{eq:hypo-I}), (\ref{eq:hypo-II})  
belong to the class of hypo-elliptic SDEs.  We note that \ref{assump:hor} is satisfied for most hypo-elliptic SDEs used in applications, as e.g.~is the case for the models cited in Section \ref{sec:intro}. Condition \ref{assump:coeff} allows the SDE coefficients to lie in the larger class 
$C_p^K (\mathbb{R}^N; \mathbb{R}^N)$ for some $K \in \mathbb{N}$ 
rather than $C_b^\infty (\mathbb{R}^N; \mathbb{R}^N)$. \textcolor{black}{$K$ is determined so that the contrast functions proposed later are well-defined and error terms obtained from It\^o-Taylor expansion in constructing the contrast functions are appropriately controlled. We also note that the drift functions in the smooth components, i.e. $\mu_{S}$ in (\ref{eq:hypo-I}) and $\mu_{S_1}, \mu_{S_2}$ in (\ref{eq:hypo-II}), require more regularity compared to 
$\mu_R$. This is because our construction of the contrast functions uses higher order stochastic Taylor expansions for those functions. More details are given in Sections \ref{sec:contrast_I}, \ref{sec:contrast_II}.}
\end{rem} 
\begin{example} \label{ex:hypo_I_II}
We provide some examples of SDEs satisfying the above conditions in applications.
\begin{itemize}[noitemsep, topsep=1pt, parsep=1pt, partopsep=1pt, leftmargin=14pt]  
\item[i.] Underdamped (standard) Langevin equation for 
one-dimensional particle with a unit mass:
\begin{align}
\begin{aligned} \label{eq:ULE}
d q_t & = p_t dt; \\ 
d p_t & = (-  U' (q_t) + \gamma p_t) dt + \sigma d W_t, 
\end{aligned} 
\end{align}
where $\theta = (\gamma, \sigma)$ is the parameter vector and $U : \mathbb{R} \to \mathbb{R}$ is some smooth potential with polynomial growth. The values of $q_t$ and $p_t$ represent the position and momentum of the particle, respectively. This model belongs in class (\ref{eq:hypo-I}) and indeed satisfies condition \ref{assump:hor}(i). 
\item[ii.] Quasi-Markovian Generalised Langevin Equation ($q$-GLE) for the case of an one-dimensional particle with an one-dimensional auxiliary variable: 
\begin{align} 
\begin{aligned} 
d q_t & = p_t dt; \\
d p_t & = 
( -  U' (q_t) +  \lambda s_t ) dt; \\ 
d s_t & = (- \lambda p_t - \alpha s_t) dt +  \sigma  \, d W_t,   
\end{aligned} \label{eq:QGLE}
\end{align} 
where $\theta = (\lambda, \alpha, \sigma)$ is the parameter vector and $U$ is as in (\ref{eq:ULE}). Note that component $s_t$ is now introduced as an auxiliary variable to capture the non-Markovianity of the memory kernel. The $q$-GLE class has been recently actively studied as an effective model in physics (see e.g.~\cite{lei:15}) and parameter estimation of the model also has been investigated in  \cite{kal:15}. The drift function of the smooth component $q_t$ is now independent of the rough component $s_t$. Model (\ref{eq:QGLE}) belongs in the SDE class (\ref{eq:hypo-II}) and satisfies  condition \ref{assump:hor}(ii). Details can be found, e.g., in \cite{igu:23}. 
\end{itemize}
 
\end{example}
\subsection{Contrast Estimators for Degenerate SDEs} \label{sec:contrast_hypo} 
Under \ref{assump:hor}-\ref{assump:coeff}, we define contrast functions for the classes of  hypo-elliptic SDEs (\ref{eq:hypo-I}) and (\ref{eq:hypo-II}), so that the corresponding parameter estimators attain a CLT under the design condition $\Delta_n = o (n^{-1/p})$, $p \ge 2$. The development of  our contrast functions is related to the approaches of \cite{kess:97, uchi:12} where, in an elliptic setting, contrast functions delivering CLTs with $\Delta_n = o (n^{-1/p})$ are obtained.   
However, carrying forward such earlier approaches to the hypo-elliptic diffusion is far from straightforward, mainly due to the degeneracy of the diffusion matrix. After a brief review of the construction of contrast functions that deliver a CLT with $\Delta_n = o (n^{-1/p})$ in the elliptic case, we proceed to the treatment of the hypo-elliptic class of models. \textcolor{black}{Hereafter, we make use of superscripts $(\mrm{I})$ and $(\mrm{II})$ as necessary to specify the class of hypo-elliptic SDEs we work with.}  
\subsubsection{Review of Contrast Estimators for  Elliptic SDEs}
We review the construction of contrast estimators for elliptic SDEs in \cite{kess:97, uchi:12}, where the diffusion matrix $a (x, \sigma) = A (x, \sigma) A(x, \sigma)^\top$ for the SDE in (\ref{eq:sde_1}) is now assumed to be positive definite for any $(x, \sigma) \in \mathbb{R}^N \times \Theta_\sigma$.   

\begin{rem}
We introduce the following notation. For a vector $y \in \mathbb{R}^N$ and a symmetric matrix $\Sigma \in \mathbb{R}^{N \times N}$, $N \in \mathbb{N}$, we define $\mathbf{q} (y; \Sigma) = y^\top \Sigma y$.    
\end{rem}

\noindent
\textbf{Step 1.} Via an It\^o-Taylor expansion, one obtains high-order approximations for the mean and variance of $\sample{X}{i}$ given $\sample{X}{i-1}$, that is: 
\begin{align*}
\mathbb{E}_{\theta} 
\bigl[ \sample{X}{i}  | \mathcal{F}_{t_{i-1}}\bigr] 
& = r_{K_p} (\Delta, \sample{X}{i-1}, \theta) + \mathcal{O} (\Delta^{K_p + 1}); \\
\mathrm{Var}_{\theta} 
\bigl[ \sample{X}{i}  | \mathcal{F}_{t_{i-1}}\bigr] 
& = \Delta \cdot \Xi_{K_p} (\Delta, \sample{X}{i-1}, \theta) 
+ \mathcal{O} (\Delta^{K_p + 1}), 
\end{align*}
where $K_p = [p/2]$, with $r_{K_p} (\Delta, \sample{X}{i-1}, \theta) \in \mathbb{R}^N$ and $ \Xi_{K_p} (\Delta, \sample{X}{i-1}, \theta) \in \mathbb{R}^{N \times N}$ determined as follows:
\begin{gather*} 
r_K (\Delta, \sample{X}{i-1}, \theta) 
= \sample{X}{i-1} + \sum_{1 \le k \le K_p} \Delta^k \cdot \mathcal{L}^{k-1} \mu (\sample{X}{i-1}, \theta); \\ 
\Xi_{K_p} (\Delta, \sample{X}{i-1}, \theta) 
=  \sum_{0 \le k \le K_p} \Delta^{k} \cdot 
\Sigma_{k} ( \sample{X}{i-1}, \sigma ).
\end{gather*}
In the above expression for $\Xi_{K_p}$, we have $\Sigma_0 = a = A A^\top$, and $\Sigma_k$, $k \ge 1$, are matrices available in closed-form and which include high-order derivatives of the SDE coefficients. 
\\ 

\noindent
\textbf{Step 2.} 
\cite{kess:97, uchi:12} make use of a Gaussian density with mean $r_{K_p} (\Delta_n, \sample{X}{i-1}, \theta)$ and variance $\Delta_n \cdot \Xi_{K_p} (\Delta_n, \sample{X}{i-1}, \theta)$ as a proxy for the intractable transition density of $X_{t_i}$ given $X_{t_{i-1}}$.   
Furthermore, to ensure that the approximation is well-defined (notice that $\Xi_{K_p}$ is not guaranteed to be positive-definite) and avoid cumbersome technicalities,  they apply a formal Taylor expansion on $\Xi_{K_p}^{-1} (\Delta_n, \sample{X}{i-1}, \theta)$ and $\log \det \Xi_{K_p} (\Delta_n, \sample{X}{i-1}, \theta)$ 
around $\Delta_n = 0$ so that the positive definiteness of the matrix $a(x, \sigma)$ is exploited.  
Thus, they define the estimator as 
%
$
\hat{\theta}_{p, n}^{\, \mrm{Elliptic}} = \mathrm{arg \, min}_{\theta \in \Theta} \ell_{p, n}^{\, \mrm{Elliptic}} (\theta), \,  p \ge 2,
$ 
%
for the contrast function: 
\begin{align}  \label{eq:contrast_ell}
\ell_{p, n}^{\, \mrm{Elliptic}} (\theta)  
& = \sum_{i = 1}^n \sum_{k = 0}^{K_p} 
\Delta_n^{k} \, \cdot  \biggl\{ 
\tfrac{1}{\Delta_n}  \mathbf{q} \bigl(\sample{X}{i} -  r_{K_p} (\Delta_n, \sample{X}{i-1}, \theta); \mathbf{G}_k (\sample{X}{i-1}, \theta) \bigr)
\nonumber \\ 
& \qquad \qquad \qquad \qquad 
+ \mathbf{H}_k (\sample{X}{i-1}, \theta) \biggr\}, 
\end{align} 
%
%
%
where
$\mathbf{G}_k : \mathbb{R}^N \times \Theta \to \mathbb{R}^{N \times N}$ and 
$\mathbf{H}_k : \mathbb{R}^N \times \Theta \to \mathbb{R}$ are analytically available and correspond to the coefficients of the $\Delta^k$-term in the formal Taylor expansion of 
$\Xi_{K_p}^{-1} 
(\Delta, \sample{X}{i-1}, \theta)$ 
and of $\log \det \Xi_{K_p} (\Delta, \sample{X}{i-1}, \theta)$ at $\Delta = 0$, respectively.
\begin{rem}
Due to the fact that $K_p = [p/2]$, $p \ge 2$,
the contrast function/estimator has the same form for $p = 2k, 2k+1, \, k \in \mathbb{N}$. However, as \cite{kess:97, uchi:12} noted in their works, when $p=2$, a simpler estimator based upon the Euler-Maruyama discretisation is asymptotically normal under the condition $\Delta_n = o (n^{-1/2})$, with a contrast function given as:
\begin{align*}
& \ell^{\mrm{EM}}_n (\theta)  = \sum_{i = 1}^n 
\biggl\{ \tfrac{1}{\Delta_n} \cdot \mathbf{q} \bigl(\sample{X}{i} -  X_{t_{i-1}} - \Delta_n \, \mu (X_{t_{i-1}}, \beta); a^{-1} (\sample{X}{i-1}, \sigma)  \bigr) \\
& \qquad \qquad \qquad  
+ \log \det a (\sample{X}{i-1}, \sigma)
\biggr\}, 
\end{align*}  
%
%
%
\end{rem}
%
%
\subsubsection{Contrast Estimators for SDE Class (\ref{eq:hypo-I})} 
\label{sec:contrast_I}
We will adapt the strategy of \cite{kess:97, uchi:12} to construct contrast functions for degenerate SDEs, starting from the hypo-elliptic class (\ref{eq:hypo-I}). 
However, the development is not straightforward because now the diffusion matrix $a(x, \sigma)$ is not positive definite. 
Another important difference between the hypo-elliptic and the elliptic setting is that an It\^o-Taylor expansion for moments of the SDE can involve $\Delta$ with varying orders across smooth and rough components. 
%
%
\begin{example} \label{ex:hypo-I}
%
We consider the under-damped Langevin equation defined in (\ref{eq:ULE}). 
For the model, the It\^o-Taylor expansion gives:  
\begin{align*}
\mrm{Var}_\theta [\sample{X}{i+1} | \mathcal{F}_{t_i}] = \sigma^2 
\begin{bmatrix}
\tfrac{\Delta^3}{3} & \tfrac{\Delta^2}{2} \\[0.1cm] 
\tfrac{\Delta^2}{2} & \Delta 
\end{bmatrix} 
+ 
\begin{bmatrix}
\mathcal{O} (\Delta^4) & \mathcal{O} (\Delta^3) \\ 
\mathcal{O} (\Delta^3) & \mathcal{O} (\Delta^2)
\end{bmatrix}, \quad \Delta = t_{i+1} - t_i.  
\end{align*}
Thus, the order of $\Delta$ is larger for the smooth component. Note that the leading term $\sigma^2 \Delta^3/3$ in the variance of $X_{\Delta}^1$ derives from the Gaussian variate $\textstyle \sigma \int_0^\Delta W_s ds$ arising in the It\^o-Taylor expansion of  $X_\Delta^1$. 
\end{example}
One must now find a positive definite matrix instead of $a(x, \theta)$ to obtain a general contrast function in the form of (\ref{eq:contrast_ell}) in a hypo-elliptic setting while dealing with such structure of varying scales amongst the components of degenerate SDEs. We thus begin by considering a standardisation of $X_\Delta^{\, (\mrm{I})}$ (conditionally on $\textstyle X_0^{\, (\mrm{I})} = x$) via subtracting high-order mean approximations from smooth/rough components and dividing with appropriate $\Delta$-terms. In particular, we introduce the $\mathbb{R}^N$-valued random variables as:
\begin{align*}
{Y}^{\,(\mathrm{I})}_{\, p, \Delta}
:=  m^{\, (\mathrm{I})}_{p} (\Delta, x, 
X_\Delta^{\, (\mathrm{I})}, \theta)
= 
\begin{bmatrix}
  \tfrac{1}{\sqrt{\Delta^3}} 
\Bigl( X_{S, \Delta}^{\, (\mathrm{I})} 
- {r}_{S, K_p + 1}^{(\mathrm{I})} (\Delta, x; \theta)  \Bigr) \\   
\tfrac{1}{\sqrt{\Delta}} 
\Bigl( X_{R, \Delta}^{\, (\mathrm{I})}  
- {r}_{R, K_p}^{(\mathrm{I})}  (\Delta, x; \theta) \Bigr) 
\end{bmatrix}, 
\end{align*}
where we have set $K_p = [p/2]$ and for $q \in \mathbb{N}$, 
\begin{align}  \label{eq:mean_expansion_I} 
\begin{bmatrix}
    {r}_{S, q}^{(\mathrm{I})}  (\Delta, x, \theta)  \\[0.1cm]  
    {r}_{R, q}^{(\mathrm{I})} (\Delta, x, \theta)  
\end{bmatrix}
= 
\begin{bmatrix}
    x_{S} \\[0.1cm]
    x_R
\end{bmatrix}
 + \sum_{k=1}^q \frac{\Delta^k}{k!} 
\begin{bmatrix}
    \mathcal{L}^{k-1}  \mu_{S} (x, \theta) \\
    \mathcal{L}^{k-1}  \mu_{R} (x, \theta)
\end{bmatrix}, 
\quad x = 
\begin{bmatrix}
   x_S \\ 
   x_R 
\end{bmatrix} \in \mathbb{R}^N
\end{align}  
with this latter quantity obtained from an It\^o-Taylor expansion of $\mathbb{E}_\theta [X_{\Delta}^{\, (\mathrm{I})}]$. 
\textcolor{black}{We note that $Y_{p, \Delta}$ is well-defined under \ref{assump:coeff}$(2K_p)$}. 
Also, note that different orders (by one) of mean approximation are used for smooth and rough components in the above standardisation. We will provide some more details on this in Remark \ref{rem:mean_def} later in the paper. We denote the Lebesgue density of the distribution of $X_\Delta^{\, (\mrm{I})}$ given $X_0^{\, (\mrm{I})} = x$ and $\theta \in \Theta$ as 
%
$\textstyle 
y \mapsto {\mathbb{P}_\theta ( X^{\, (\mrm{I})}_\Delta \in dy) }/{dy} = p^{X^{\, (\mrm{I})}_\Delta} (x, y; \theta).
$
%
Transformation of random variables gives that, for $(x, y, \theta) \in \mathbb{R}^N \times \mathbb{R}^N \times \Theta$:
\begin{align*} 
p^{X^{(\mathrm{I})}_{\Delta}} (x, y ; \theta)
= \frac{1}{\sqrt{\Delta^{ 3 N_{S} + N_R}}} p^{{Y}^{\,(\mathrm{I})}_{\, p, \Delta}} 
\left( \xi ; \theta \right) \big|_{\xi =  m_{p}^{\, (\mathrm{I})} (\Delta, x, y, \theta)},  \ \    
\end{align*} 
where 
\textcolor{black}{$p^{{Y}^{\,(\mathrm{I})}_{\, p, \Delta}} (\xi; \theta) \equiv \mathbb{P}_\theta (Y_{p, \Delta}^{(\mrm{I})} \in d \xi) / d \xi $ denotes the Lebesgue density of the law of ${Y}^{ (\mathrm{I})}_{\, p, \Delta}$ given the parameter $\theta \in \Theta$.}  
%
Following the standarisation, an It\^o-Taylor expansion now gives, under \ref{assump:coeff}$\,(2J)$, $J \in \mathbb{N}$:
\begin{align} \label{eq:quadratic_approx}
\mathbb{E}_{\theta} 
\left[
{Y}^{\, (\mathrm{I})}_{\, p, \Delta} 
\bigl({Y}^{\,(\mathrm{I})}_{\, p, \Delta} \bigr)^\top 
\right]  
=
\boldsymbol{\Sigma}^{\, (\mathrm{I})} (x, \theta) 
+ \sum_{j = 1}^{J} \Delta^j \cdot
\boldsymbol{\Sigma}_{j}^{\, (\mathrm{I})} (x, \theta)   
+ R^{(\mathrm{I})} (\Delta^{J  + 1}, x, \theta),
\end{align} 
where $R^{(\mathrm{I})} \in \mathcal{S}$,
for some analytically available $N \times N$ matrices  $\boldsymbol{\Sigma}_{j}^{\, (\mathrm{I})} ( x, \theta)$. In particular, matrix $\boldsymbol{\Sigma}^{(\mrm{I})} (x, \theta)$ has the following block expression: 
\begin{align*} 
\boldsymbol{\Sigma}^{\, (\mrm{I})} ( x, \theta) 
\equiv 
\begin{bmatrix}
\boldsymbol{\Sigma}_{SS}^{\, (\mrm{I})} ( x, \theta)  
& \boldsymbol{\Sigma}_{SR}^{\, (\mrm{I})} ( x, \theta)  \\[0.1cm] 
\boldsymbol{\Sigma}_{RS}^{\, (\mrm{I})} ( x, \theta)  
& \boldsymbol{\Sigma}_{RR}^{\, (\mrm{I})} ( x, \theta) 
\end{bmatrix}, 
\end{align*}
where we have set: 
\begin{align*}
\boldsymbol{\Sigma}_{RR}^{\, (\mrm{I})} (x, \theta) 
&= \sum_{k = 1}^d  A_{R,k} (x , \sigma)  A_{R, k} (x , \sigma)^\top 
\equiv  a_R (x , \sigma);  \\ 
\boldsymbol{\Sigma}_{SR}^{\, (\mrm{I})} (x, \theta)
& = \tfrac{1}{2} \sum_{k = 1}^d 
\mathcal{L}_k \mu_{S} (x, \theta) A_{R, k} (x, \sigma)^\top, \quad  
\boldsymbol{\Sigma}_{RS}^{\, (\mrm{I})} (x, \theta) = \boldsymbol{\Sigma}_{SR}^{\, (\mrm{I})} (x, \theta)^\top; \\ 
\boldsymbol{\Sigma}_{SS}^{\, (\mrm{I})} (x, \theta)  
& = \tfrac{1}{3} \sum_{k = 1}^d 
\mathcal{L}_k \mu_{S} (x, \theta) 
\mathcal{L}_k \mu_{S} (x, \theta)^\top \equiv \tfrac{1}{3} a_S (x, \theta).  
\end{align*}
Matrix $\boldsymbol{\Sigma}^{\, (\mrm{I})} (x, \theta)$ plays a role similar to $a (x, \theta)$ in the elliptic setting due to the following result whose proof is in Section \ref{append:pf_positive_I} of the Supplementary Material. 
\begin{lemma} \label{lemma:positive_I}
Under \ref{assump:hor}$(i)$, the matrices $a_R (x , \sigma)$,   $a_S (x, \theta)$ and $\boldsymbol{\Sigma}^{\, (\mrm{I})} (x, \theta)$ are positive definite for all $(x, \theta) \in \mathbb{R}^N \times \Theta$.
\end{lemma}
We form a Gaussian approximation for $ p^{{Y}^{\,(\mathrm{I})}_{\, p, \Delta}} \left( \xi ; \theta \right)$ to obtain a contrast function that will be well-defined due to the positive definiteness of matrix $\boldsymbol{\Sigma}^{(\mrm{I})} (x, \theta)$. We introduce some notation. For $h > 0$, $\theta \in \Theta$, $0 \le i \le n$, $1 \le j \le K$ and $K \in \mathbb{N}$,  
\begin{gather*}
\mathbf{m}_{p, i}^{\, (\mathrm{I})} (h, \theta) 
:= {m}_p^{(\mathrm{I})} 
(h, \sample{X}{i-1}^{(\mathrm{I})}, 
\sample{X}{i}^{(\mathrm{I})}, \theta); \\  
\boldsymbol{\Sigma}_{i}^{(\mathrm{I})} (\theta) 
:= \boldsymbol{\Sigma}^{(\mathrm{I})} (\sample{X}{i}, \theta), 
\quad  
\boldsymbol{\Sigma}_{i, j}^{(\mathrm{I})} (\theta) := \boldsymbol{\Sigma}_{j}^{(\mathrm{I})} (\sample{X}{i}, \theta); \\ 
\boldsymbol{\Xi}^{(\mathrm{I})}_{K, i} (h, \theta )
:= \boldsymbol{\Sigma}_i^{(\mathrm{I})} (\theta) 
+ \sum_{1 \le j \le K} h^j \cdot 
\boldsymbol{\Sigma}_{i, j}^{(\mathrm{I})} (\theta), \quad  
\boldsymbol{\Lambda}_{i}^{(\mathrm{I})} (\theta)
:= \bigl( \boldsymbol{\Sigma}_{i}^{(\mathrm{I})} (\theta) \bigr)^{-1}.  
\end{gather*}
\textcolor{black}{Note that the above functions are well-defined under \ref{assump:coeff}$(2K_p)$; see (\ref{eq:mean_expansion_I})-(\ref{eq:quadratic_approx}).} 
We write the formal Taylor expansion of 
$ \bigl( \boldsymbol{\Xi}_{K, i}^{(\mathrm{I})} (h, \theta) \bigr)^{-1}$ and $\log \det \boldsymbol{\Xi}_{K, i}^{(\mathrm{I})} (h, \theta)$ up to the level $K \in \mathbb{N}$ as 
%
$ \textstyle 
\sum_{0 \le k \le K} h^k \cdot \mathbf{G}_{i, k}^{(\mathrm{I})} (\theta), \;  
\sum_{0 \le k \le K} h^k \cdot 
\mathbf{H}_{i, k}^{(\mathrm{I})} (\theta), 
$
%
respectively, where 
\begin{align*}
\mathbf{G}_{i, k}^{(\mrm{I})} (\theta) 
& = 
\frac{1}{k!} \partial_h^k \, \bigl( \boldsymbol{\Xi}_{K,i}^{(\mrm{I})} (h, \theta) \bigr)^{-1} \, \big|_{h = 0}; \\  
\mathbf{H}_{i, k}^{(\mrm{I})} (\theta) 
& = 
\frac{1}{k!} \partial_h^k 
\bigl( \log \det \boldsymbol{\Xi}_{K, i}^{(\mrm{I})} (h, \theta) \bigr) \big|_{h = 0}, 
\qquad 0 \le k \le K. 
\end{align*}
For instance, one has: 
\begin{gather*}
\mathbf{G}_{i, 0}^{(\mathrm{I})} (\theta) 
 = \boldsymbol{\Lambda}_{i}^{\mrm{(I)}} (\theta), \quad  
\mathbf{G}_{i, 1}^{(\mathrm{I})} (\theta)  
 = - \boldsymbol{\Lambda}_i^{(\mathrm{I})} (\theta) \, 
\boldsymbol{\Sigma}_{i, 1}^{(\mathrm{I})} (\theta) \, \boldsymbol{\Lambda}_i^{(\mathrm{I})} (\theta); \\ 
\mathbf{G}_{i, 2}^{(\mathrm{I})} (\theta)
=  - \bigl( \mathbf{G}_{i, 1}^{ (\mathrm{I})} (\theta) \boldsymbol{\Sigma}_{i, 1}^{  (\mrm{I})} (\theta) 
+ \boldsymbol{\Lambda}_i^{ (\mathrm{I})} (\theta) 
\boldsymbol{\Sigma}_{i, 2}^{ (\mathrm{I})} (\theta) 
\bigr) 
\boldsymbol{\Lambda}_i^{ (\mathrm{I})} (\theta)
\end{gather*}
and 
\begin{gather*}
\mathbf{H}_{i, 0}^{(\mathrm{I})} (\theta) 
= \log \det \boldsymbol{\Sigma}^{\,(\mrm{I})}_{i} (\theta),  \quad 
\mathbf{H}_{i, 1}^{(\mathrm{I})} (\theta) 
 = \mathrm{Tr} 
\Bigl[ 
\boldsymbol{\Lambda}^{\, (\mathrm{I})}_i (\theta) \, 
\boldsymbol{\Sigma}_{i, 1}^{(\mathrm{I})} (\theta) 
\Bigr]; \\ 
\mathbf{H}_{i, 2}^{(\mathrm{I})} (\theta)  
= 
\mathrm{Tr} 
\Bigl[ 
\tfrac{1}{2} \mathbf{G}_{i, 1}^{(\mrm{I})} (\theta) \boldsymbol{\Sigma}^{(\mrm{I})}_{i, 1} (\theta)
+ 
\boldsymbol{\Lambda}^{ (\mrm{I})}_i (\theta) \boldsymbol{\Sigma}_{i, 2}^{(\mrm{I})} (\theta) 
\Bigr].
\end{gather*}
%
%
Note that the terms $\mathbf{G}_{i, k}^{(\mrm{I})} (\theta)$ and 
$\mathbf{H}_{i, k}^{(\mrm{I})} (\theta)$ involve the inverse of the matrix $\boldsymbol{\Sigma}_i^{(\mathrm{I})} (\theta)$, and are well-defined following Condition \ref{assump:hor}$(i)$ 
and Lemma \ref{lemma:positive_I}. We now construct the contrast function $\ell_{p, n, \Delta}^{\, (\mathrm{I})} (\theta)$,  $\theta = (\beta_S, \beta_R, \sigma) \in \Theta$, for the hypo-elliptic class (\ref{eq:hypo-I}) as follows. \textcolor{black}{Under Condition \ref{assump:coeff}$(2K_p)$}: 
\begin{itemize}[leftmargin=20pt]
\item[(i)] For $p = 2$, 
\begin{align} 
\label{eq:contrast_I_p2}
\ell_{2, n}^{\, (\mrm{I})} (\theta) = 
\sum_{i = 1}^n 
\Bigl\{ \mathbf{m}_{2, i}^{\, (\mathrm{I})} (\Delta_n, \theta)^\top 
\boldsymbol{\Lambda}^{\, (\mathrm{I})}_{i-1} (\theta) 
\, \mathbf{m}_{2, i}^{\, (\mathrm{I})} (\Delta_n, \theta) 
+ \log \det 
\boldsymbol{\Sigma}^{\, (\mathrm{I})}_{i-1} (\theta)  \Bigr\}. 
\end{align}
\item[(ii)] For $p \ge 3$, 
\begin{align} \label{eq:contrast_I} 
\hspace{-0.8cm}
\ell_{p, n}^{\, (\mathrm{I})} (\theta) 
=
\sum_{i = 1}^n \sum_{j = 0}^{K_p} 
\Delta_n^{j} \cdot 
\bigl\{ 
\mathbf{m}_{p, i}^{\, (\mathrm{I})} (\Delta_n, \theta)^\top  
\, 
\mathbf{G}_{i-1, j}^{\, (\mathrm{I})} (\theta) 
\, 
\mathbf{m}_{p, i}^{\, (\mathrm{I})} (\Delta_n, \theta) 
+ \mathbf{H}_{i-1, j}^{\, (\mathrm{I})} (\theta)
\bigr\}. 
\end{align}
\end{itemize}
Thus, the contrast estimator $
\hat{\theta}_{p, n}^{\, (\mathrm{I})}
= ( 
\hat{\beta}_{S, p, n}, \, 
\hat{\beta}_{R, p, n}, \, 
\hat{\sigma}_{p, n}
)$,  $p \ge 2$, is defined as 
%
%
$$\hat{\theta}_{p, n}^{\, (\mathrm{I})} =
\operatorname*{arg\,min}_{\theta \in \Theta} \ell_{p, n}^{\, (\mathrm{I})} (\theta).$$ 
%
%
\subsubsection{Contrast Estimators for SDE Class (\ref{eq:hypo-II})} \label{sec:contrast_II} 
To construct the contrast function and corresponding estimator for the highly degenerate diffusion class defined via (\ref{eq:hypo-II}), 
we proceed as above, that is we consider an appropriate standardisation that takes under consideration the scales of the three components $(X_{S_1, \Delta}, X_{S_2, \Delta}, X_{R, \Delta}), \, \Delta > 0 $. Note that the first smooth component $X_{S_1, \Delta}$ has a variance of size $\mathcal{O} (\Delta^5)$ due to the largest (as $\Delta\rightarrow 0$) variate in the It\^o-Taylor expansion being  $\textstyle \int_0^\Delta \int_0^s W_u du ds$, where we have made use of the fact that the drift  $\mu_{S_1}$ is a function only of the smooth components. 
Thus, \textcolor{black}{under Condition \ref{assump:coeff}$\,(2K_p)$}, we introduce: 
\begin{align*}
\mathbf{m}_{p, i}^{(\mathrm{II})} (\Delta, \theta)
:= 
\begin{bmatrix} 
\tfrac{1}{\sqrt{\Delta_n^{5}}} 
\bigl( X_{S_1, t_i}^{(\mathrm{II})} - {r}_{S_1, K_p + 2}^{(\mathrm{II})} (\Delta, \sample{X}{i-1}^{(\mathrm{II})}, \theta) \bigr)  \\ 
\tfrac{1}{\sqrt{\Delta_n^{3}}} 
\bigl( X_{S_2, t_i}^{(\mathrm{II})} - {r}_{S_2, K_p + 1}^{(\mathrm{II})} (\Delta, \sample{X}{i-1}^{(\mathrm{II})}, \theta) \bigr)  \\ 
\tfrac{1}{\sqrt{\Delta_n}} 
\bigl( X_{R, t_i}^{(\mathrm{II})} - {r}_{R, K_p}^{(\mathrm{II})} (\Delta, \sample{X}{i-1}^{(\mathrm{II})}, \theta) \bigr)
\end{bmatrix}, 
\end{align*}
where we have set, for $q \in \mathbb{N}$, $(\Delta, x, \theta) \in (0, \infty) \times \mathbb{R}^N \times \Theta$:
\begin{align*}  
\begin{bmatrix}
{r}_{S_1, q}^{(\mathrm{II})}  (\Delta, x, \theta)  \\[0.1cm]  
{r}_{S_2, q}^{(\mathrm{II})}  (\Delta, x, \theta)  \\[0.1cm]    
{r}_{R, q}^{(\mathrm{II})} (\Delta, x, \theta)  
\end{bmatrix}
= 
\begin{bmatrix}
x_{S_1} \\[0.1cm]
x_{S_2} \\[0.1cm] 
x_R
\end{bmatrix}
+ \sum_{k=1}^q \frac{\Delta^k}{k!} 
\begin{bmatrix}
\mathcal{L}^{k-1}  \mu_{S_1} (x, \theta) \\
\mathcal{L}^{k-1}  \mu_{S_2} (x, \theta) \\
\mathcal{L}^{k-1}  \mu_{R} (x, \theta)
\end{bmatrix}. 
\end{align*}   
We will explain the choice of the truncation levels used above in the mean approximation $r^{(\mrm{II})}$ in Remark \ref{rem:mean_def}. 
Via an It\^o-Taylor expansion, we obtain \textcolor{black}{under Condition \ref{assump:coeff}$\,(2K_p)$} that:
\begin{align} \label{eq:expansion_cov_2}
& \mathbb{E}_{\trueparam} 
\bigl[  
\mathbf{m}_{p, i+1}^{(\mathrm{II})} (\Delta, \theta)^\top 
\mathbf{m}_{p, i+1}^{(\mathrm{II})} (\Delta, \theta) 
| \mathcal{F}_{t_{i}} 
\bigr] 
= \boldsymbol{\Sigma}^{(\mrm{II})}_i (\theta) + \sum_{j = 1}^{K_p} \Delta^j \cdot  \boldsymbol{\Sigma}^{(\mrm{II})}_{i, j} (\theta)  
+ {R} (\Delta^{K_p + 1}, \obs{X}{i}{II}, \theta), 
\end{align}
for some analytically available matrices $\boldsymbol{\Sigma}^{(\mrm{II})}_i (\theta)$,  $\boldsymbol{\Sigma}^{(\mrm{II})}_{j, i} (\theta)  \in \mathbb{R}^{N \times N}$, $ 1 \le i \le n$, $j \in \mathbb{N}$, 
and a residual $R(\cdot)$, such that $[{R}(\cdot)]_{k_1 k_2} \in \mathcal{S}$, $1 \le k_1, k_2  \le N$. 
%
%
In particular, the matrix $\boldsymbol{\Sigma}^{(\mrm{II})}_i (\theta)
\equiv \boldsymbol{\Sigma}^{(\mrm{II})} (\obs{X}{i}{II}, \theta)$ admits the following block expression:
\begin{align} \label{eq:mat_Sigma_II}
\boldsymbol{\Sigma}^{(\mrm{II})} (x, \theta) 
= 
\begin{bmatrix}
\boldsymbol{\Sigma}^{(\mrm{II})}_{S_1S_1} (x, \theta) & 
\boldsymbol{\Sigma}^{(\mrm{II})}_{S_1 S_2} (x, \theta) & 
\boldsymbol{\Sigma}^{(\mrm{II})}_{S_1 R}  (x, \theta)  \\[0.1cm]
\boldsymbol{\Sigma}^{(\mrm{II})}_{S_2 S_1} (x, \theta) & 
\boldsymbol{\Sigma}^{(\mrm{II})}_{S_2 S_2} (x, \theta) & 
\boldsymbol{\Sigma}^{(\mrm{II})}_{S_2 R}  (x, \theta) \\[0.1cm]  
\boldsymbol{\Sigma}^{(\mrm{II})}_{R S_1} (x, \theta) & 
\boldsymbol{\Sigma}^{(\mrm{II})}_{R S_2} (x, \theta) & 
\boldsymbol{\Sigma}^{(\mrm{II})}_{R R}  (x, \theta) %
\end{bmatrix} 
\end{align} 
for $(x, \theta) \in \mathbb{R}^N \times \Theta$, where we have set: 
\begin{align*}
\boldsymbol{\Sigma}_{RR}^{\, (\mrm{II})} (x, \theta) 
&= \sum_{k = 1}^d  A_{R,k} (x , \sigma)  A_{R, k} (x , \sigma)^\top 
\equiv  a_R (x , \sigma);  \\ 
\boldsymbol{\Sigma}_{S_2 R}^{\, (\mrm{II})} (x, \theta)  
& = \tfrac{1}{2} \sum_{k = 1}^d 
\mathcal{L}_k \mu_{S_2} (x, \theta) A_{R, k} (x, \sigma)^\top, \quad  \boldsymbol{\Sigma}_{RS_2}^{\, (\mrm{II})} (x, \theta)   
= \boldsymbol{\Sigma}_{S_2R}^{\, (\mrm{II})} (x, \theta)^\top; \\ 
\boldsymbol{\Sigma}_{S_1 R}^{\, (\mrm{II})} (x, \theta)  
& = \tfrac{1}{6} \sum_{k = 1}^d 
\mathcal{L}_k \mathcal{L} \mu_{S_1} (x, \theta) A_{R, k} (x, \sigma)^\top, \quad 
\boldsymbol{\Sigma}_{R S_1}^{\, (\mrm{II})} (x, \theta)   
= \boldsymbol{\Sigma}_{S_1 R}^{\, (\mrm{II})} (x, \theta)^\top; \\ 
\boldsymbol{\Sigma}_{S_2 S_2}^{\, (\mrm{II})} (x, \theta)  
& = \tfrac{1}{3} \sum_{k = 1}^d 
\mathcal{L}_k  \mu_{S_2} (x, \theta) \mathcal{L}_k  \mu_{S_2} (x, \theta)^\top
\equiv \tfrac{1}{3} a_{S_2} (x, \theta),
\\[-0.3cm] 
\boldsymbol{\Sigma}_{S_1 S_2}^{\, (\mrm{II})} (x, \theta)  
&= \tfrac{1}{8} \sum_{k = 1}^d 
\mathcal{L}_k \mathcal{L} \mu_{S_1} (x, \theta) \mathcal{L}_k  \mu_{S_2} (x, \theta)^\top, \quad \boldsymbol{\Sigma}_{S_2 S_1}^{\, (\mrm{II})} (x, \theta)   
= \boldsymbol{\Sigma}_{S_1 S_2}^{\, (\mrm{II})} (x, \theta)^\top, \\[-0.3cm]
\boldsymbol{\Sigma}_{S_1 S_1}^{\, (\mrm{II})} (x, \theta)  
&= \tfrac{1}{20} \sum_{k = 1}^d 
\mathcal{L}_k \mathcal{L} \mu_{S_1} (x, \theta) \mathcal{L}_k \mathcal{L} \mu_{S_1} (x, \theta)^\top \equiv \tfrac{1}{20} a_{S_1} (x, \theta).   
\end{align*}
As is the case with the matrix $\boldsymbol{\Sigma}^{(\mrm{I})} (x, \theta)$,  $(x, \theta) \in \mathbb{R}^N \times \Theta$, 
for the matrix $\boldsymbol{\Sigma}^{(\mrm{II})} (x, \theta)$, we have the following result whose proof is provided in Appendix B of \cite{igu:23}. 
\begin{lemma} \label{lemma:positive_II}
Under condition {\ref{assump:hor}}$(ii)$, the matrices $a_R (x, \theta)$, $a_{S_2} (x, \theta)$, $a_{S_1} (x, \theta)$ and $\boldsymbol{\Sigma}^{(\mrm{II})}(x, \theta)$ are positive definite for all $(x, \theta) \in \mathbb{R}^N \times \Theta$. 
\end{lemma}
For $h > 0$, $0 \le i \le n$ 
and $0 \le k \le K_p = [p/2]$,  we write: 
\vspace{-0.1cm}
\begin{gather*} 
\boldsymbol{\Lambda}_i^{(\mrm{II})} (\theta) = 
\bigl( \boldsymbol{\Sigma}_i^{(\mrm{II})} (\theta)  \bigr)^{-1},
\  
\boldsymbol{\Xi}_{K_p, i}^{\, (\mrm{II})} (h, \theta) 
= \boldsymbol{\Sigma}^{(\mrm{II})}_i (\theta) 
+ \sum_{j = 1}^{K_p} h^j \cdot  \boldsymbol{\Sigma}^{(\mrm{II})}_{i, j} (\theta); \\[-0.2cm]  
\mathbf{G}_{i, k}^{(\mrm{II})} (\theta) 
= \tfrac{1}{k!} \, \partial_h^k 
\bigl( \boldsymbol{\Xi}_{K_p, i}^{(\mrm{II})} (h, \theta) \bigr)^{-1} \big|_{h = 0}, 
\    
\mathbf{H}_{i, k}^{(\mrm{II})}  
(t, \theta) 
= \tfrac{1}{k!} \partial_h^k 
\bigl( 
\log 
\det 
\boldsymbol{\Xi}_{K_p, i}^{(\mrm{II})}
(h, \theta)     
\bigr) |_{h = 0}.   
\end{gather*}
We now obtain our contrast function $\ell_{p, n}^{\, (\mrm{II})} (\theta)$, $p \ge 2$, for the hypo-elliptic class (\ref{eq:hypo-II}). \textcolor{black}{Under Condition \ref{assump:coeff}$\,(2K_p)$}:   
\begin{itemize}[topsep = 1pt, parsep = 1pt, partopsep = 1pt, leftmargin=20pt]
\item[(i)] For $p = 2$: 
\begin{align*}
\ell_{2, n}^{\, (\mrm{II})} (\theta) = 
\sum_{i = 1}^n 
\Bigl\{ \mathbf{m}_{2, i}^{\, (\mathrm{II})} (\Delta_n, \theta)^\top 
\boldsymbol{\Lambda}^{\, (\mathrm{II})}_{i-1} (\theta) 
\, \mathbf{m}_{2, i}^{\, (\mathrm{II})} (\Delta_n, \theta) 
+ \log \det \boldsymbol{\Sigma}^{\, (\mathrm{II})}_{i-1} (\theta) \Bigr\}. 
\end{align*}
\item[(ii)] For $p \ge 3$: 
\begin{align*} 
\ell_{p, n}^{\, (\mathrm{II})} (\theta) 
=
\sum_{i = 1}^n \sum_{j = 0}^{K_p} 
\Delta_n^{j} \cdot 
\bigl\{ 
\mathbf{m}_{p, i}^{\, (\mathrm{II})} (\Delta_n, \theta)^\top 
\, 
\mathbf{G}_{i-1, j}^{\, (\mathrm{II})} (\theta) 
\, 
\mathbf{m}_{p, i}^{\, (\mathrm{II})} (\Delta_n, \theta) 
+ \mathbf{H}_{i-1, j}^{\, (\mathrm{II})} (\theta) 
\bigr\}.
\end{align*}
%
\end{itemize}
Thus, the contrast estimator  is defined as:
\begin{align*} 
\hat{\theta}_{p, n}^{\, (\mathrm{II})} 
=
\bigl(
\hat{\beta}_{S_1, p, n}^{\, (\mathrm{II})}, \,
\hat{\beta}_{S_2, p, n}^{\, (\mathrm{II})}, \, 
\hat{\beta}_{R, p, n}^{\, (\mathrm{II})}, \, 
\hat{\sigma}_{p, n}^{\, (\mathrm{II})}  \bigr)
= 
\operatorname*{arg \, min}_{\theta \in \Theta} \ell_{p, n}^{\, (\mathrm{II})} (\theta). 
\end{align*}   
\begin{rem} \label{rem:mean_def}
In the definition of contrast estimator, we make use of mean approximations with different number of terms for the various components, so that it holds that, for any $\theta \in \Theta$, $p \ge 2$, $1 \le i \le n,$ and $1 \le k \le N$:
\begin{align} \label{eq:m_approx}
\mathbb{E}_\theta 
\Bigl[
\mathbf{m}_{p, i}^{(w), k} (\Delta_n,  \theta) | \mathcal{F}_{t_{i-1}} \Bigr] 
= R_k \bigl( \sqrt{\Delta_n^{2 K_p + 1}}, \sample{X}{i-1}^{(w)}, \theta \bigr),  \quad  w \in \{ \mrm{I}, \mrm{II} \}, 
\end{align}
with some $R_k \in \mathcal{S}$. This is one of the key developments to obtain the CLT under the weaker design condition $\Delta_n = o (n^{-1/p})$, $p \ge 2$. 
Notice that we divide the smooth components by smaller quantities in terms of powers of $\Delta_n$, e.g., $\sqrt{\Delta_n^5}$ and $\sqrt{\Delta_n^3}$ for the components $X_{S_1}^{(\mrm{II})}$ and $X_{S_2}^{(\mrm{II})}$, respectively, thus we require more accurate mean approximations for those components to obtain (\ref{eq:m_approx}).  
\end{rem}
\section{Asymptotic Properties of the Contrast Estimators}
\label{sec:main} 
To state our main results, we introduce a set of additional conditions for the two classes (\ref{eq:hypo-I}) and (\ref{eq:hypo-II}). Recall that we write the true parameter as 
$ \theta^{\dagger} = 
\bigl(
\beta_{S}^{\dagger},  
\beta_{R}^{\dagger},  
\sigma^{\dagger} 
\bigr) $, under the interpretation that
$ \beta_{S}^\dagger = \bigl(
\beta_{S_1}^{\dagger},  
\beta_{S_2}^{\dagger}  
\bigr)$ 
for class (\ref{eq:hypo-II}), 
and that $\theta^{\dagger}$ is assumed to be unique and to lie in the interior of $\Theta$.
%

{
\begin{enumerate}[noitemsep, topsep=2pt,parsep=3pt, partopsep=3pt, leftmargin=25pt] 
\setcounter{enumi}{2}
\renewcommand{\theenumi}{H\arabic{enumi}}
\renewcommand{\labelenumi}{\theenumi} 
\item \label{assump:additional_con}
For any $x \in \mathbb{R}^N$ and any multi-index $\alpha \in \{1, \ldots, N\}^{l}$, $l \ge 0$,
the functions: 
\begin{align*}
\theta \mapsto  
\partial^x_\alpha
\mu^i (x,  \theta), 
\quad 
\theta \mapsto  
\partial^x_\alpha A_j^i (x,  \theta), 
\qquad  1 \le j \le  {d}, \quad 1 \le i \le N, 
\end{align*}
are three times differentiable. For any multi-index $\beta \in \{1, \ldots,  N_\theta \}^{l}$,  
$l \in\{1, 2, 3\}$, the functions:  
\begin{align*}
x \mapsto  
\partial^\theta_\beta
\partial^x_\alpha 
\mu^i (x ,  \theta),  \qquad 
x \mapsto   
\partial^\theta_\beta
\partial^x_\alpha A_j^i (x, \theta), \qquad 1 \le j \le d, 
\quad 1 \le i \le N,  
\end{align*}
\textcolor{black}{are of polynomial growth}, uniformly in $\theta \in \Theta$. 
%
%
\item \label{assump:moments}
The diffusion process $\{X_t\}_{t \geq 0}$ defined via (\ref{eq:hypo-I}) or (\ref{eq:hypo-II}) is ergodic under $\theta = \trueparam$, with invariant distribution denoted by $\truedist$. Furthermore, all moments of $\truedist$ are finite. 
\item \label{assump:finite_moment}
It holds that for all $r \ge 1$, 
$\textstyle \sup_{t > 0} \mathbb{E}_{\trueparam} [|X_t|^r] < \infty$. 
\item \label{assump:ident}
If it holds: 
\begin{align*}
\mu_{S} (x , \beta_S) = \mu_{S} (x , \truebeta_S), \ \ 
\mu_{R} (x , \beta_R) = \mu_{R} (x , \truebeta_R), \ \  
A_R (x , \sigma) = A_R (x , \truesigma), 
\end{align*}
for $x$ in set of probability 1 under  $\nu_{\trueparam}$, then 
$\beta_S = \truebeta_S, \,  \beta_R = \truebeta_R, \, \sigma = \truesigma$.
\end{enumerate}  
} 
We first show that the proposed contrast estimators $\hat{\theta}_{p, n}^{(\mathrm{I})}$, $\hat{\theta}_{p, n}^{(\mathrm{II})}$ are consistent in the  high-frequency, complete observation regime:
\begin{theorem}[Consistency] \label{thm:consistency}
Let $p \ge 2$ be an arbitrary integer and $K_p = [p/2]$. Assume that conditions \ref{assump:hor}, \ref{assump:coeff}\,$(2K_p)$, \ref{assump:additional_con}, \ref{assump:moments}, \ref{assump:finite_moment} and \ref{assump:ident} hold. If \limit, then:
\begin{align*} 
\mbox{$\hat{\theta}_{p,n}^{\, (\mathrm{I})}
\probconv
\trueparam$},  
\qquad    
\hat{\theta}_{p,n}^{\, (\mathrm{II})}
\probconv
\trueparam.  
\end{align*} 
\end{theorem}
\noindent Also, the proposed estimators are  asymptotically normal under the condition $\Delta_n = o (n^{-1/p})$, $p \ge 2$.
\begin{theorem}[CLT]  \label{thm:clt}
Let $p \ge 2$ be an arbitrary integer and $K_p = [p/2]$.  
Assume that conditions \ref{assump:hor}, \ref{assump:coeff}\,$(2K_p)$, \ref{assump:additional_con}, \ref{assump:moments}, \ref{assump:finite_moment} and \ref{assump:ident} hold. If \limit, with $\Delta_n =  o ( n^{-1/p})$, then:
\begin{enumerate}[leftmargin = 15pt, parsep=1pt, partopsep = 1pt]
\item[I.] For the hypo-elliptic class (\ref{eq:hypo-I}):
\end{enumerate}
\begin{align*}
\begin{bmatrix}
\sqrt{\tfrac{n}{\Delta_n}} ( \hat{\beta}_{S, p, n}^{(\mrm{I})} - \truebeta_S) 
\\[0.2cm] 
\sqrt{n \Delta_n} ( \hat{\beta}_{R, p, n}^{(\mrm{I})} - \truebeta_R)  \\[0.2cm]  
\sqrt{n} (\hat{\sigma}_{R, p, n}^{(\mrm{I})} - \truesigma ) 
\end{bmatrix} 
\distconv \mathcal{N} \bigl( \mathbf{0}_{N_\theta}, 
{\Gamma}^{(\mathrm{I})} ( \trueparam )^{-1} \bigr),  
\end{align*} 
where $\Gamma^{(\mathrm{I})} ( \trueparam ) 
= 
\mrm{Diag} 
\bigl[ 
\Gamma^{(\mathrm{I})}_{\beta_{S}}  ( \trueparam ), \, 
\Gamma^{(\mathrm{I})}_{\beta_{R}} ( \trueparam ), \, 
\Gamma^{(\mathrm{I})}_{\sigma}  ( \trueparam ) 
\bigr]$ 
is the asymptotic precision matrix 
%
with the involved block matrices specified as: 
\begin{align*}
& \bigl[ \Gamma^{(\mathrm{I})}_{\beta_{S}} ( \trueparam ) 
\bigr]_{i_1j_1} 
= 12 \int  \partial_{\beta_{S}, i_1}  
\mu_{S}  (y, \truebeta_{S})^\top 
a_{S}^{-1} (y, \trueparam) 
\, 
\partial_{\beta_{S}, j_1}  
\, 
\mu_{S} (y, \truebeta_{S}) 
\, \truedist (d y);  
\\
& \bigl[ \Gamma^{(\mathrm{I})}_{\beta_{R}} ( \trueparam ) 
\bigr]_{i_2 j_2} 
= \int  \partial_{\beta_{R}, i_2}  
\mu_{R} \, (y, \truebeta_{R})^\top 
a_{R}^{-1} (y, \truesigma) 
\, 
\partial_{\beta_{R}, j_2}  
\, 
\mu_{R} \, (y, \truebeta_{R}) 
\, \truedist (d y);  \\ 
& \bigl[ \Gamma^{(\mathrm{I})}_{\sigma} ( \trueparam ) 
\bigr]_{i_3j_3} 
= \tfrac{1}{2} \int  
\mathrm{Tr} \bigl[ \partial_{\sigma, i_3} \boldsymbol{\Sigma}^{(\mrm{I})} (y, \trueparam) \boldsymbol{\Lambda}^{(\mathrm{I})} (y, \trueparam) 
\partial_{\sigma, j_3} 
\boldsymbol{\Sigma}^{(\mrm{I})} (y, \trueparam) 
\boldsymbol{\Lambda}^{(\mathrm{I})} (y, \trueparam) \bigr]
\, \truedist (d y),
\end{align*}   
for $1 \le i_1, j_1 \le N_{\beta_{S}}$, $1 \le i_2, j_2 \le N_{\beta_{R}}$ and $1 \le i, j \le N_{\sigma}$. 
\begin{enumerate}[leftmargin = 15pt, parsep=1pt, partopsep = 1pt]
\item[II.] For the hypo-elliptic class (\ref{eq:hypo-II}):     
\end{enumerate} 
\begin{align*}
\begin{bmatrix}
\sqrt{\tfrac{n}{\Delta_n^3}} ( \hat{\beta}_{S_1, p, n}^{(\mrm{II})} - \truebeta_{S_1}) \\[0.2cm] 
\sqrt{\tfrac{n}{\Delta_n}} ( \hat{\beta}_{S_2, p, n}^{(\mrm{II})} - \truebeta_{S_2}) \\[0.2cm] 
\sqrt{n \Delta_n} ( \hat{\beta}_{R, p, n}^{(\mrm{II})} - \truebeta_S)  \\[0.2cm] 
\sqrt{n} (\hat{\sigma}_{R, p, n}^{(\mrm{II})} - \truesigma ) 
\end{bmatrix} 
\distconv \mathcal{N} \bigl( \mathbf{0}_{N_\theta}, \Gamma^{(\mathrm{II})} ( \trueparam )^{-1} \bigr),  
\end{align*} 
where $\Gamma^{(\mathrm{II})} ( \trueparam ) = 
\mrm{Diag} 
\bigl[ 
\Gamma^{(\mathrm{II})}_{\beta_{S_1}} ( \trueparam ), \, 
\Gamma^{(\mathrm{II})}_{\beta_{S_2}}  ( \trueparam ), \, 
\Gamma^{(\mathrm{II})}_{\beta_{R}} ( \trueparam ), \, 
\Gamma^{(\mathrm{II})}_{\sigma}  ( \trueparam ) 
\bigr]$ is the asymptotic precision matrix 
%
%
with the involved block matrices specified as: 
\begin{align*}
& \bigl[ \Gamma^{(\mathrm{II})}_{\beta_{S_1}} ( \trueparam ) 
\bigr]_{i_1j_1}
= 720 \int  \partial_{\beta_{S_1}, i_1} \, 
\mu_{S_1}  (y, \truebeta_{S_1})^\top 
a_{S_1}^{-1} (y, \trueparam) 
\, 
\partial_{\beta_{S_1}, j_1} 
\, 
\mu_{S_1}  (y, \truebeta_{S_1}) 
\, \truedist (d y);  \\ 
& \bigl[ \Gamma^{(\mathrm{II})}_{\beta_{S_2}} ( \trueparam ) 
\bigr]_{ij} = 12 \int  \partial_{\beta_{S_2}, i_2}  
\mu_{S_2}  (y, \truebeta_{S_2})^\top 
a_{S_2}^{-1} (y, \trueparam) 
\, 
\partial_{\beta_{S_2}, j_2}  
\, 
\mu_{S_2} (y, \truebeta_{S_2}) 
\, \truedist (d y);   \\ 
& \bigl[ \Gamma^{(\mathrm{II})}_{\beta_{R}} ( \trueparam ) 
\bigr]_{i_3j_3} = \int  \partial_{\beta_{R}, i_3}  
\mu_{R} \, (y, \truebeta_{R})^\top 
a_{R}^{-1} (y, \truesigma) 
\, 
\partial_{\beta_{R}, j_3}  
\, 
\mu_{R} \, (y, \truebeta_{R}) 
\, \truedist (d y);  \\ 
& \bigl[ \Gamma^{(\mathrm{II})}_{\sigma} ( \trueparam ) 
\bigr]_{i_4j_4} 
= \tfrac{1}{2} \int    
\mathrm{Tr}  
\bigl[ \partial_{\sigma, i_4} \boldsymbol{\Sigma}^{(\mrm{II})} (y, \trueparam) \boldsymbol{\Lambda}^{ (\mathrm{II})} (y, \trueparam) 
\partial_{\sigma, j_4} \boldsymbol{\Sigma}^{ (\mrm{II})} (y, \trueparam) 
\boldsymbol{\Lambda}^{(\mathrm{II})} (y, \trueparam) \bigr]
\truedist (d y),   
\end{align*}  
for $1 \le i_1,j_1 \le N_{\beta_{S_1}}$, $1 \le i_2, j_2 \le N_{\beta_{S_2}}$, $1 \le i_3, j_3 \le N_{\beta_{R}}$ and $1 \le i_4, j_4 \le N_{\sigma}$.  
\end{theorem} 
The proofs are given in Section \ref{sec:pf}.  
\begin{rem} 
The design condition $\Delta_n = o (n^{-1/p})$, $p \ge 2$, i.e.~$n \Delta_n^p \to 0$, appears in the CLT result (Theorem \ref{thm:clt}). As we explain later in the proof in Section \ref{sec:pf_asymptotic_normality}, the condition is required so that the expectation of the score function, specifically, the gradient of the contrast function $\ell_{p,n} (\trueparam)$, tends to $0$. This is relevant for the mean of the asymptotic Gaussian distribution to converge to $0$. If the given design condition is not satisfied, the distribution of estimators will tend to concentrate on an area that deviates from the true value, thus the estimators will suffer from bias. We observe this issue in the numerical experiments in Section \ref{sec:sim}, where the standard contrast estimator $\hat{\theta}_{2, n}$ exhibits the described bias when $n \Delta_n^2$ is not sufficiently small,  while  $\hat{\theta}_{p, n}$ for $p \ge 3$ avoids such a bias. 
\end{rem}
\section{Numerical Applications} 
\label{sec:sim}
In the numerical examples shown below, for given choices of $\Delta_n$ and $n$, we observe gradual improvements in the behaviour of the estimates when moving from $p=2$ to $p=3$ and then to $p=4$. Indeed, discrepancies are stronger in the experiment that contrasts $p=2$ with $p=4$.
\subsection{Quasi-Markovian Generalised Langevin Equation}
We study numerically the properties of the proposed contrast estimator for the $q$-GLE model defined via (\ref{eq:QGLE}) in Example \ref{ex:hypo_I_II} and which belongs in class \mbox{(\ref{eq:hypo-II})}. We consider the following two choices of potential function $U$,  leading to a linear/non-linear system of degenerate SDEs: 
\begin{description}
    \item[Case I. Quadratic potential:] $\mathbb{R} \ni q \mapsto U (q) = D q^2/2$ with some parameter  $D > 0$.
    \item[Case II. Double-well potential:] $\mathbb{R} \ni q \mapsto U(q) =  (q^2 - D)^2/4$, with some parameter $D > 0$. 
\end{description}
For the above two cases, we generate $M=100$ independent datasets by applying the locally Gaussian (LG) discretisation scheme defined in \cite{igu:23} with a small step-size $\Delta=10^{-4}$. We treat such data as obtained from the true model (\ref{eq:hypo-II}) in the understanding that the discretisation bias is negligible.
We then obtain the synthetic complete observations by sub-sampling from the above datasets with coarser time increments. For the two choices of potential above, we consider the designs of high-frequency observations given in Table \ref{table:design}. 
We specify the true values for $\theta = (D, \lambda, \alpha, \sigma)$ as $(2.0, 2.0, 4.0, 4.0)$ and $(2.0, 1.0, 4.0, 4.0)$ for Case I and Case II, respectively.
We compute the estimators $\hat{\theta}_{p, n} (=\hat{\theta}_{p, n}^{\, (\mathrm{II})}) = (\hat{D}_{p,n}, \hat{\lambda}_{p,n}, \hat{\alpha}_{p,n}, \hat{\sigma}_{p,n})$ for $p = 2,3$. 
To find the minima of the contrast functions we used the adaptive moments (Adam) optimiser with the following  specifications: (step-size) = $0.1$, 
(exponential decay rate for the first moment estimates) = $0.9$, (exponential decay rate for the second moment estimates) = $0.999$, (additive term for numerical stability) = $10^{-8}$ and (number of iterations) = $8,000$. 
Figure \ref{fig:mle} shows boxplots of the individual realisations of relative discrepancies $(\hat{\theta}_{p,n}^j - \theta^{\dagger,j})/\theta^{\dagger,j}, \, 1 \le j \le 4$. In all four designs, when $p=2$ we observe that estimator $\hat{\sigma}_{2, n}$ suffers from a severe bias, as its realisations do not cover the true value $\truesigma$ (see the boxplots at the bottom of Figure~\ref{fig:mle}). In contrast, when $p=3$ the $100$ realisations of $\hat{\sigma}_{3, n}$ are concentrated around $\truesigma$. 
We observe that the estimators of parameters $D, \lambda$ in the drift functions of the smooth components converge quickly to the true values for both $p = 2, 3$. 
This agrees with the fast convergence rate of  $\sqrt{\Delta_n/n}$
obtained in the CLT result (Theorem \ref{thm:clt}). 
\textcolor{black}{In Table \ref{table:time_qgle} we show the mean computational time of one realisation of estimators corresponding to $p=2$ and $p=3$ from 10 independent sets of observations. Notice from the table that in all experiments the calculation of $\hat{\theta}_{3, n}$ does not add a significant cost compared with that of $\hat{\theta}_{2, n}$. We can thus conclude from Figure \ref{fig:mle} and Table \ref{table:time_qgle} that use of $\hat{\theta}_{3, n}$ provides a reliable estimation that does not suffer from the bias observed in the case of $\hat{\theta}_{2, n}$, under a competitive computational cost.} 
%
%
\begin{table}
\centering
\caption{Designs of experiment} 
\label{table:design} 
\begin{tabular}{lcc} 
\toprule  \\[-8pt]
 &  \textbf{Case I.} &  \textbf{Case II.}  \\ \midrule
 \textbf{Design A}  & $(\Delta_n, T_n) = (0.008, 1000)$ 
 &   $(\Delta_n, T_n) = (0.01, 1000)$ \\ 
 \hline 
 \\[-8pt] 
 \textbf{Design B}  & $(\Delta_n, T_n) = (0.005, 500)$  & $(\Delta_n, T_n) = (0.005, 1000)$ 
 \\ \bottomrule
\end{tabular}
%
\end{table}
\begin{table}[H] 
\centering
\caption{\textcolor{black}{Average running times (seconds) of one realisation of estimators for the $q$-GLE model.}} 
\label{table:time_qgle} 
\begin{tabular}{lcc} 
\toprule  \\[-8pt]
 Experiment & $ \hat{\theta}_{2,n}$  &  $\hat{\theta}_{3,n} $ 
 \\ \midrule
 \\[-8pt] 
 \textbf{Design A} - \textbf{Case I} 
 & 38.25  & 44.33
 \\
 \textbf{Design A} - \textbf{Case II} 
 & 24.46  & 28.34   
 \\
 \textbf{Design B} - \textbf{Case I} 
 & 27.48  & 28.71  
 \\
 \textbf{Design B} - \textbf{Case II} 
 & 90.26  & 104.12  
 \\ \bottomrule
\end{tabular}
%
\end{table} 
\begin{figure}
\centering
\begin{subfigure}[b]{0.49\textwidth}
\centering
\includegraphics[width=6.6cm]{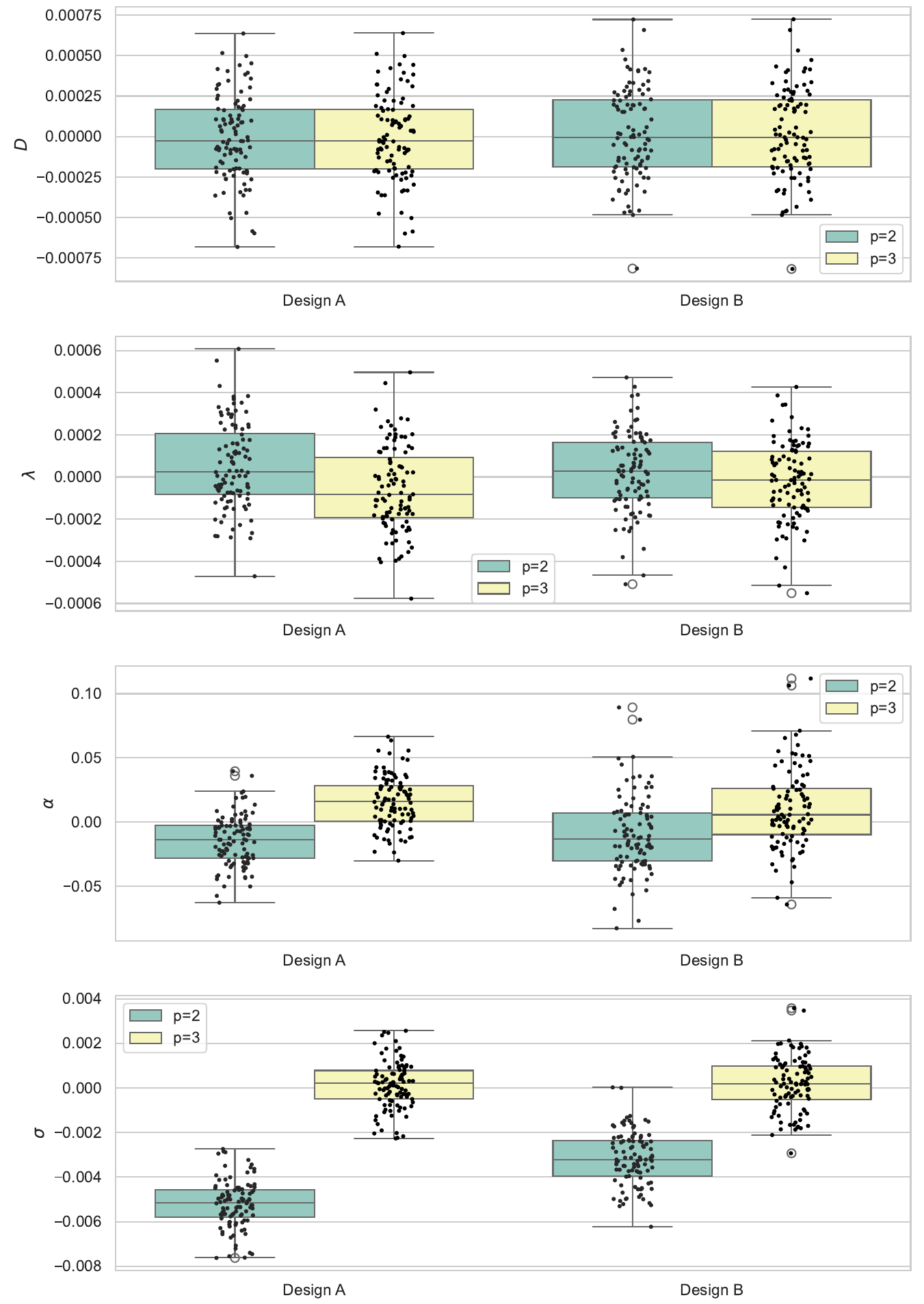}
\caption{Case I.} %
\label{fig:set_1}
\end{subfigure}
\hfill
\begin{subfigure}[b]{0.47\textwidth}
\centering
\includegraphics[width=6.6cm]{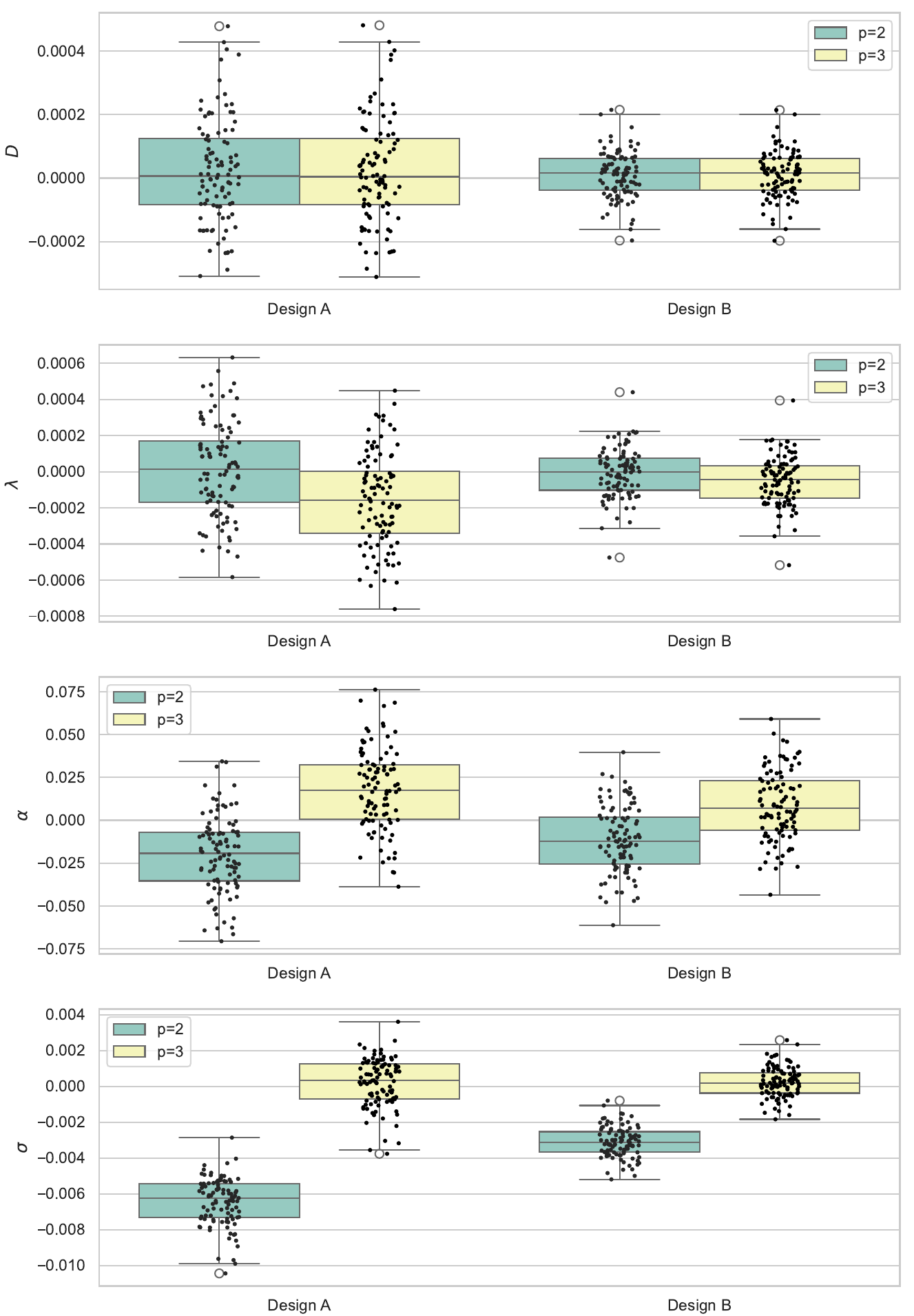}
\caption{Case II.} 
\label{fig:set_2}
\end{subfigure} 
\caption{Boxplots of $M=100$ independent realisations of $\bigl( \hat{\theta}_{p, n}-\trueparam \bigr) / \trueparam$ for the models Case I (left) and Case II (right). The black points show the individual realisations. }
\label{fig:mle}
\end{figure} 
\begin{figure}
    \centering
    \includegraphics[width=15cm]{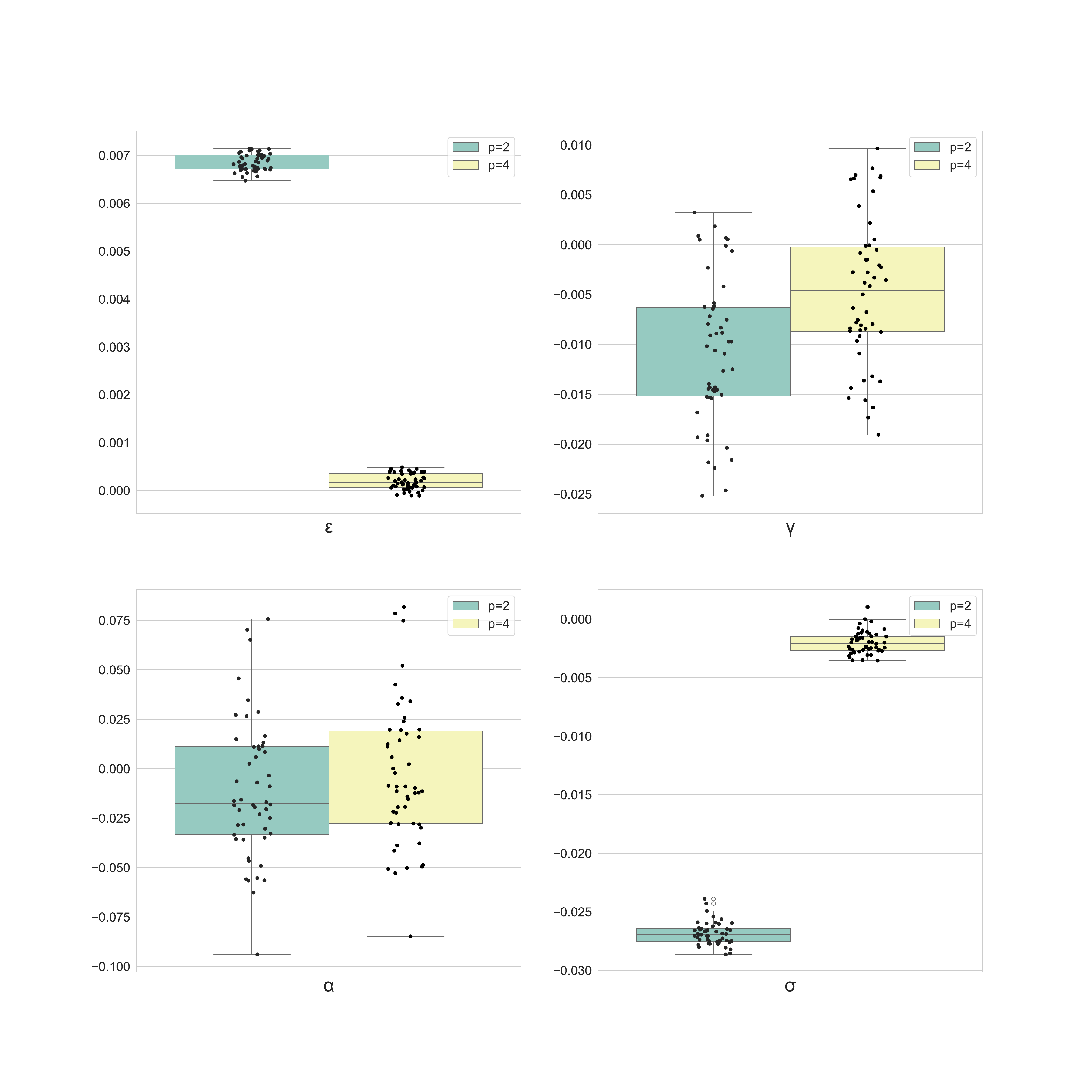}
    \caption{Boxplots of the $M=50$ realisations of $\bigl( \hat{\theta}_{p, n}-\trueparam \bigr) / \trueparam$ for the FHN model, under the design $n = 250,000$, $T_n = 5,000$, $\Delta_n = 0.02$. The black points show the individual values. }
    \label{fig:mle_re_fhn}
\end{figure}  
\subsection{FitzHugh-Nagumo Model}
We consider the FitzHugh-Nagumo (FHN) model belonging in  class (\ref{eq:hypo-I}), and specified via the following bivariate SDE:
\begin{align} 
\begin{aligned} 	
d X_t   = \tfrac{1}{\varepsilon} (X_t - (X_t)^3 - Y_t - s )dt;  \\ 
d Y_t   = (\gamma X_t  - Y_t + \alpha) dt + \sigma d W_t^1, 
\end{aligned} 
\label{eq:fhn_model}
\end{align}  
with $\theta= (\varepsilon, \gamma, \alpha, \sigma)$. We consider parameter estimation of the FHN model in both complete and partial observation regimes, where only the component $X_t$ is observed in the latter case. 
\subsubsection{Complete Observation Regime}
We consider the estimator  $\hat{\theta}_{p, n} (=\hat{\theta}_{p, n}^{\, (\mathrm{I})}) = (\hat{\varepsilon}_{p, n}, \hat{\gamma}_{p, n}, \hat{\alpha}_{p, n}, \hat{\sigma}_{p, n} )$ for $p = 2, 4$. As \cite{melnykova2020parametric}, we fix $s = 0.01$ and set the true values to $\trueparam = (0.10, 1.50, 0.30, 0.60)$. We generate $M=50$ independent datasets by using the LG discretisation for the (\ref{eq:hypo-I}) class defined in \cite{glot:20, igu:22} with a step-size $\Delta=10^{-4}$. Then, we obtain $M=50$ synthetic complete observations for the FHN model by sub-sampling from the above datasets so that  $n = 250,000$, $T_n = 5,000$ and $\Delta_n = 0.02$. 
To obtain the minima of the
 contrast functions we used the Adam optimiser with the same settings as the experiments with the $q$-GLE earlier, except for (step-size) = 0.01. 
 Figure \ref{fig:mle_re_fhn} shows boxplots of the corresponding $50$ realisations of relative discrepancies $(\hat{\theta}_{p, n}^j - \theta^{\dagger,j})/\theta^{\dagger,j}$,  $1 \le j \le 4$. 
Note that the means of $\hat{\theta}_{p, n} - \trueparam$ with $p=4$ are closer to $0$ than those with $p =2$ for all $\theta$-coordinates  while both estimators attain similar standard deviations. 
 Importantly, when $p=2$, we see in Figure \ref{fig:mle_re_fhn}  that 
 $\hat{\varepsilon}_{2,n}$ and $\hat{\sigma}_{2,n}$ are severely biased as the true values are not covered by the 50 realisations. In contrast, when $p=4$, the realisations of $\hat{\theta}_{4, n}$ are spread on areas close to the true parameters. 
%
%
%
%
\subsubsection{Partial Observation Regime}
\label{sec:fhn_partial}
We use the proposed contrast function, and the corresponding approximate log-likelihood, under the partial observation regime where only the smooth component $\{X_{t_i}\}$ is observed.  
In particular, we compute maximum likelihood estimators (MLEs) based upon a closed-form marginal likelihood constructed from the Gaussian approximation corresponding to the contrast function $\ell_{p, n}^{\, (\mrm{I})} (\theta)$ for $p=2, 3$. To be precise, we use the high-order expansions for the mean and the variance, but do not make use of the Taylor expansions for the inverse of the covariance and of its log-determinant, so that we obtain quantities that correspond to proper density functions. 
Here, we mention that due to the structure of the FHN model, the obtained one-step Gaussian approximation corresponding to the developed contrast function for $p=2, 3$ is a linear Gaussian model w.r.t.~the hidden component $Y_{t_i}$ given the observation $X_{t_i}$. Thus, one can make use of the Kalman Filter (KF) and calculate the marginal likelihood of the partial observations. 
See e.g.~\cite{igu:23} where a marginal likelihood is built upon KF for the $q$-GLE model belonging in class (\ref{eq:hypo-II}). We provide the details of the Gaussian approximation and the marginal likelihood for the FHN model in Section  \ref{appendix:KF} of the Supplementary Material.  
Then, the MLE for partial observation is defined as: 
\begin{align} \label{eq:mle_partial}
\hat{\theta}_{p, n} 
\bigl( = (
\hat{\varepsilon}_{p, n}, 
\hat{\gamma}_{p, n} , 
\hat{\alpha}_{p, n} ,
\hat{\sigma}_{p, n} 
) 
\bigr) 
= \operatorname*{arg \, max}_{\theta \in \Theta} 
\log f_{p, n}
(\{ X_{t_{i}} \}_{i = 0, \ldots, n}; \theta),
\end{align} 
for $p = 2, 3$, where $f_{p, n} (\{X_{t_i}\}_{i = 0, \ldots, n}; \theta)$ is the marginal likelihood. As with the experiment under complete observations, we fix $s = 0.01$ and set the true values to 
$\trueparam = (0.10, 1.50, 0.30, 0.60)$. We generate $M = 50$ independent datasets from the LG discretisation with a step-size $\Delta = 10^{-4}$. 
Then, we obtain $50$ synthetic partial observations of $\{X_{t_i}\}_{i = 0,\ldots, n}$ with $n = 200,000$, $T_n = 1,000$ and $\Delta_n = 0.005$ by sub-sampling from the datasets and removing the rough component $\{Y_{t_i}\}_{i = 0, \ldots, n}$. The Nelder-Mead method is applied to optimise the log-marginal likelihood for $p = 2,3$ with the initial guess $\theta_0 = (0.5, 0.5, 0.5, 0.5)$.     
%
%
\begin{figure}
    \centering
    \includegraphics[width=12cm]{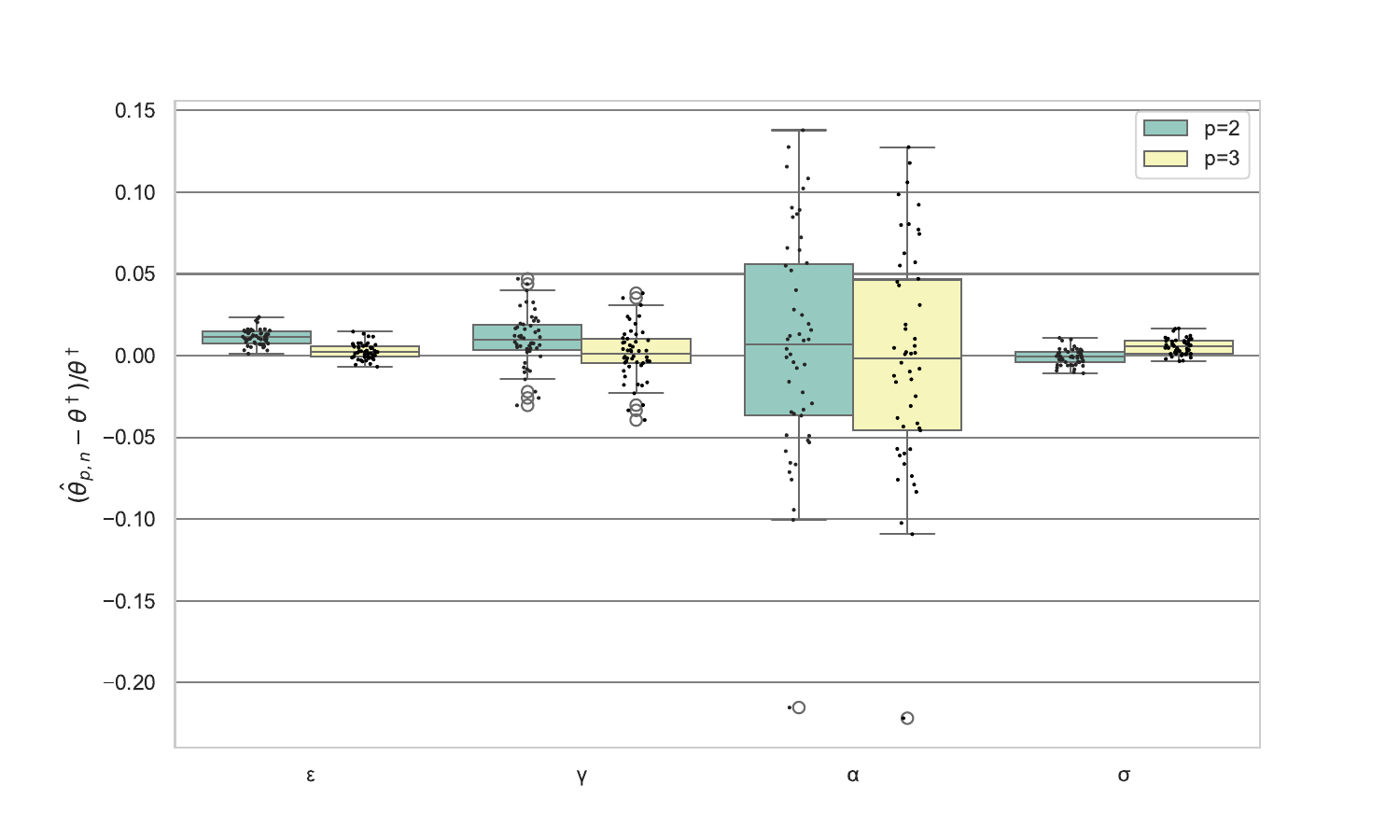}
    \caption{Boxplots of $\bigl( \hat{\theta}_{p, n}-\trueparam \bigr) / \trueparam$ from $M=50$ sets of partial observations for the FHN model, under the design $n = 200,000$, $T_n = 1,000$ and $\Delta_n = 0.005$. The black points show the individual values. }
    \label{fig:mle_re_fhn_partial}
\end{figure}  
\begin{table}
\centering
\caption{\textcolor{black}{Average running times (seconds) of one realisation of estimators for the FHN model, from partial observations.}} 
\label{table:time_fhn} 
\begin{tabular}{cc} 
\toprule  \\[-8pt]
 $ \hat{\theta}_{2,n}$  &  $\hat{\theta}_{3,n} $ 
 \\ \midrule
 \\[-8pt] 
100.976  & 138.428 
 \\ \bottomrule
\end{tabular} 
\end{table}
Figure \ref{fig:mle_re_fhn_partial} shows boxplots of the corresponding $50$ realisations of relative discrepancies $(\hat{\theta}_{p, n}^j - \theta^{\dagger,j})/\theta^{\dagger,j}$,  $1 \le j \le 4$. First, we observe that both estimators share almost the same standard deviations. Second, the bias in $\hat{\varepsilon}_{2, n}$ is clear from the boxplot in the sense that the true value $\varepsilon^\dagger$ is not included in the interval between the minimum and the maximum of the $50$ realisations of $\hat{\varepsilon}_{2, n}$.  
Finally, the square of the mean of $\hat{\theta}_{p, n} - \trueparam$ (eq.~bias) is reduced in the drift parameters ($\varepsilon$, $\gamma$ and $\alpha$) for $p=3$. 
\textcolor{black}{Table \ref{table:time_fhn} presents the average computational times of one realisation of estimators $\hat{\theta}_{2, n}$ and $\hat{\theta}_{3,n}$ from 10 independent sets of partial observations. The increase in the average time of $\hat{\theta}_{3, n}$ from that of $\hat{\theta}_{2, n}$ is about $40 \%$, thus use of $\hat{\theta}_{3, n}$ is quite feasible.} 
\section{Proof of Main Results} \label{sec:pf}
We provide the proof for our main results, Theorems \ref{thm:consistency} and \ref{thm:clt}, which demonstrate the asymptotic normality of the proposed contrast estimators under the design condition $\Delta_n = o (n^{-1/p})$, $p \ge 2$. We show all details of the proof for class (\ref{eq:hypo-II}) as the proof for (\ref{eq:hypo-I}) follows from similar arguments. Throughout this section, for simplicity of notation, we frequently omit the subscript $(\mrm{II})$ appearing in the estimator and the functions introduced in Section \ref{sec:contrast_II}. A number of technical results are collected in the Supplementary Material. We make use of the notation $o_{\mathbb{P}_{\trueparam}} (\cdot)$: for a sequence of random variables $\{F_n\}$ and a numerical sequence $\{\alpha_{n, \Delta_n}\}$ depending on $n$ or $\Delta_n$, we write $F_n = o_{\mathbb{P}_{\trueparam}} (\alpha_{n, \Delta_n})$ 
if 
$ {\textstyle 
{F_n}/{\alpha_{n, \Delta_n}}  
\probconv 0}
$ as \limit. 
\textcolor{black}{Throughout this section, $p \ge 2$ is an arbitrary integer, and we set $K_p = [p/2]$ as before.} 
\subsection{Proof of Theorem \ref{thm:consistency} (Consistency)} 
\label{sec:pf_consistency}
We show the consistency of the contrast estimator via the following procedure:
\begin{description}
\item[Step 1.] We show that if \limit, then $\hat{\beta}_{S_1, p, n} \probconv \beta_{S_1}^\dagger$. In particular:  
\begin{align} \label{eq:step1}
 \hat{\beta}_{S_1, p, n} - \beta_{S_1}^\dagger 
 = o_{\mathbb{P}_{\trueparam}} (\Delta_n^{3/2}). 
\end{align}
\item[Step 2.] Making use of the rate of convergence found in (\ref{eq:step1}), we show that if \limit, then $\hat{\beta}_{S_2, p, n} \probconv \beta_{S_2}^\dagger$. In particular: 
\begin{align} \label{eq:step2}
\hat{\beta}_{S_2, p, n} 
- \beta_{S_2}^\dagger 
= o_{\mathbb{P}_{\trueparam}} (\Delta_n^{1/2}). 
\end{align} 
\item[Step 3.] Using the rates obtained in 
(\ref{eq:step1})-(\ref{eq:step2}), we show that if \limit, then $(\hat{\beta}_{R, p, n}, \hat{\sigma}_{p, n}) \probconv (\truebeta_R, \truesigma)$.
\end{description}  
In the sequel, we express the matrix $\boldsymbol{\Lambda} (x, \theta) = ( \boldsymbol{\Sigma}^{(\mrm{II})}(x, \theta))^{-1}, \, (x, \theta) \in \mathbb{R}^N \times \Theta$, as: 
\begin{align*}
\boldsymbol{\Lambda} (x, \theta)
& = 
\begin{bmatrix}  
\boldsymbol{\Lambda}_{S_1 S_1} (x, \theta) 
& \boldsymbol{\Lambda}_{S_1 S_2} (x, \theta)
& \boldsymbol{\Lambda}_{S_1 R} (x, \theta)  \\ 
\boldsymbol{\Lambda}_{S_2 S_1} (x, \theta) 
& \boldsymbol{\Lambda}_{S_2 S_2} (x, \theta)
& \boldsymbol{\Lambda}_{S_2 R} (x, \theta)  \\  
\boldsymbol{\Lambda}_{R S_1} (x, \theta) 
& \boldsymbol{\Lambda}_{R S_2} (x, \theta)
& \boldsymbol{\Lambda}_{R R} (x, \theta) 
\end{bmatrix}, 
\end{align*} 
for block matrices $\boldsymbol{\Lambda}_{\iota_1 \iota_2} (x, \theta) \in \mathbb{R}^{N_{\iota_1} \times N_{\iota_2}}$,  $\iota_1, \iota_2 \in \{S_1, S_2, R \}$. Note that from Lemma 13 in \cite{igu:23} we have:
\begin{gather*}
\boldsymbol{\Lambda}_{S_1 S_1} (x, \theta) = 720 \, a_{S_1}^{-1} (x, \theta), \qquad (x, \theta) \in \mathbb{R}^N \times \Theta,
\end{gather*} 
where $a_{S_1} (x, \theta)$ is specified in the definition of the matrix $\boldsymbol{\Sigma}^{\,(\mrm{II})} (x, \theta)$ in (\ref{eq:mat_Sigma_II}).
\begin{rem} 
The strategy outlined above is technically different from proofs in \cite{melnykova2020parametric, glot:20}, where they require that $\Delta_n = o (n^{-1/2})$ to prove consistency. 
Our proof of consistency proceeds without relying on $\Delta_n = o (n^{-1/2})$, and such an approach then leads to a CLT under the weaker condition $\Delta_n = o (n^{-1/p})$, $p \ge 2$. More details on this point are given in Remark~\ref{rem:key}. 
\end{rem}

\subsubsection{Step 1.}  \label{sec:step1}
The consistency of the estimator $\hat{\beta}_{S_1, p, n}$ is deduced from the following result.%
\begin{lemma} \label{lemm:step1}
Assume that conditions \ref{assump:hor}, \ref{assump:coeff}\,$(2K_p)$, \ref{assump:additional_con}, \ref{assump:moments} and \ref{assump:finite_moment} hold. If \limit, then: 
\begin{align} \label{eq:conv_1}
\sup_{\theta \in \Theta} \, \Bigl| 
\tfrac{\Delta_n^3}{n} \ell_{p, n} (\theta)  - \mathbf{Y}_1 (\theta) \Bigr| \probconv 0, 
\end{align} 
where, for $\theta \in \Theta$: 
\begin{align*} 
\mathbf{Y}_1 (\theta) 
= \int \mathbf{q} \Bigl( \mu_{S_1} (x_S, \beta_{S_1}) - \mu_{S_1} (x_S, \truebeta_{S_1}) ; \boldsymbol{\Lambda}_{S_1 S_1} (x , \theta) \Bigr) \nu_{\trueparam} (dx). 
\end{align*}  
\end{lemma}
The proof is given in Section \ref{appendix:pf_step1} of the Supplementary Material. Lemma \ref{lemm:step1} indeed implies the consistency of $\hat{\beta}_{S_1, p, n}$ via the following arguments. We first notice that the matrix $\boldsymbol{\Lambda}_{S_1 S_1} (x, \theta) = 720 a_S^{-1} (x, \theta)$ is positive definite for any $(x, \theta) \in \mathbb{R}^N \times \Theta$ under \ref{assump:hor} due to Lemma \ref{lemma:positive_II}. From the identifiability condition \ref{assump:ident} and the positive definiteness of $\boldsymbol{\Lambda}_{S_1 S_1}$, the term $\mathbf{Y}_1 (\theta)$, $\theta \in \Theta$, should be positive if $\beta_{S_1} \neq \truebeta_{S_1}$. Thus, for any $\varepsilon > 0$, there exists a constant $\delta > 0$ so that:
\begin{align*}
\mathbb{P}_{\trueparam} \bigl( | \hat{\beta}_{S_1, p, n}  - \truebeta_{S_1} | > \varepsilon \bigr) 
& 
\leq \mathbb{P}_{\trueparam} 
\Bigl( \mathbf{Y}_1 (\hat{\theta}_{p, n}) > \delta \Bigr). 
\end{align*}
From the definition of the estimator and Lemma \ref{lemm:step1}, we have: 
\begin{align*}
& \mathbb{P}_{\trueparam} 
\Bigl( \mathbf{Y}_1 (\hat{\theta}_{p, n}) > \delta \Bigr) \\ 
& \leq 
\mathbb{P}_{\trueparam} \Bigl(
\tfrac{\Delta_n^3}{n} \ell_{p, n}  
(\truebeta_{S_1}, \hat{\beta}_{S_2, p, n}, \hat{\beta}_{R,p,n}, \hat{\sigma}_{p,n})
- \tfrac{\Delta_n^3}{n} \ell_{p, n}
\bigl( \hat{\theta}_{p, n}  \bigr)  
+ \mathbf{Y}_1 (\hat{\theta}_{p, n}) > \delta
\Bigr) \\
& \le \mathbb{P}_{\trueparam} 
\Bigl( \, 
\sup_{\theta \in \Theta} 
\Bigl| 
\tfrac{\Delta_n^3}{n} \ell_{p, n}  
(\truebeta_{S_1}, {\beta}_{S_2}, {\beta}_{R}, {\sigma})
- \tfrac{\Delta_n^3}{n} \ell_{p, n}
\bigl( {\theta} \bigr)  
+ \mathbf{Y}_1 ({\theta})
\Bigr| > \delta \Bigr)  \to 0, 
\end{align*}
as \limit , thus $\hat{\beta}_{S_1, p, n}$ is a consistent estimator. 

We now prove  (\ref{eq:step1}). A Taylor expansion of 
$
\partial_{\beta_{S_1}} \ell_{p, n} (\hat{\theta}_{p,n})$ 
%
gives:   
\begin{align*} 
\textbf{A}_{p,n} 
\bigl(  
\truebeta_{S_1}, \hat{\beta}_{S_2,p,n}, \hat{\beta}_{R,p,n}, \hat{\sigma}_{p,n} 
\bigr) 
= \textbf{B}_{p,n} 
\bigl( 
\hat{\theta}_{p,n} \bigr) 
\; \times 
\tfrac{1}{\Delta_n^{3/2}} (\hat{\beta}_{S_1, p, n} - \truebeta_{S_1}),  
\end{align*}
where we have set, for $\theta = (\beta_{S_1}, \theta^{S_2, R} ) \in \Theta$, $\theta^{S_2, R} 
\equiv (\beta_{S_2}, \beta_R, \sigma) \in \Theta_{\beta_{S_2}} \times \Theta_{\beta_{R}} \times \Theta_{\sigma}$:
\begin{align*} 
& \mathbf{A}_{p,n}  \bigl( \theta \bigr) =  
- \tfrac{\Delta_n^{3/2}}{n} 
\partial_{\beta_{S_1}} \ell_{p, n} 
\bigl(  \theta \bigr); \\  
& \textbf{B}_{p,n}
\bigl( \theta \bigr)  
= \tfrac{\Delta_n^{3}}{n} 
\int_0^1  
\partial_{\beta_{S_1}}^2 \ell_{p, n} 
\bigl( \truebeta_{S_1} 
+ \lambda ( {\beta}_{S_1} - \truebeta_{S_1}), \theta^{S_2, R}  \bigr) d \lambda.  
\end{align*}
For simplicity of notation, we write $\Theta_{\beta_{S_2}, \beta_R, \sigma} \equiv \Theta_{\beta_{S_2}} \times \Theta_{\beta_R} \times \Theta_\sigma$. Limit (\ref{eq:step1}) holds from the following result whose proof is given in Section \ref{appendix:pf_step1_B} of the Supplementary Material.
\begin{lemma} \label{lemma:step1_B}
Assume that conditions \ref{assump:hor}, \ref{assump:coeff}\,$(2K_p)$, \ref{assump:additional_con}, \ref{assump:moments}, \ref{assump:finite_moment} and \ref{assump:ident} hold. If \limit, then:
\begin{align}
& \sup_{(\beta_{S_2}, \beta_R, \sigma) \in 
\Theta_{\beta_{S_2}, \beta_R, \sigma}}
\Bigl| \mathbf{A}_{p,n}
\bigl( 
\truebeta_{S_1}, \beta_{S_2}, \beta_R, \sigma
\bigr) \Bigr| \probconv 0;  \label{eq:conv_A}  \\
& 
\sup_{(\beta_{S_2}, \beta_R, \sigma) \in 
\Theta_{\beta_{S_2}, \beta_R, \sigma}}
\Bigl| 
\mathbf{B}_{p,n}
\bigl( 
\hat{\beta}_{S_1, p, n}, \beta_{S_2}, \beta_R, \sigma
\bigr) 
- 2  \mathbf{{B}}
\bigl( 
\truebeta_{S_1}, \beta_{S_2}, \sigma
\bigr) \Bigr| \probconv 0,  \label{eq:conv_B}
\end{align}
where for $(\beta_{S_2}, \beta_R, \sigma) \in \Theta_{\beta_{S_2}} \times \Theta_{\beta_R} \times \Theta_\sigma$: 
\begin{align*}
& \mathbf{{B}}
\bigl( 
{\beta}_{S_1}, \beta_{S_2}, \sigma
\bigr) \\ 
& = 720 \int 
\bigl( \partial_{\beta_{S_1}}^\top \mu_{S_1} (x_S, \beta_{S_1}) \bigr)^\top
a_{S_1}^{-1} (x, (\beta_{S_1}, \beta_{S_2}, \sigma)) \partial_{\beta_{S_1}}^\top \mu_{S_1} (x_S, \beta_{S_1}) \, \truedist (dx). 
\end{align*}
\end{lemma}
\begin{rem} \label{rem:key}
We point here to a key fact to obtain (\ref{eq:conv_A}) leading to the rate $\hat{\beta}_{S_1, p, n} - \truebeta_{S_1} = o_{\mathbb{P}_{\trueparam}} (\Delta_n^{3/2})$. The term $\mathbf{A}_{p,n}$ is given in the form of: 
\begin{align} \label{eq:demo_A}
\mathbf{A}_{p,n} (\truebeta_{S_1}, \beta_{S_2}, \beta_R, \sigma)
& = \tfrac{1}{n \sqrt{\Delta_n}} \sum_{i = 1}^n F (1, \sample{X}{i-1}, (\truebeta_{S_1}, \beta_{S_2}, \beta_R, \sigma)
)  \nonumber \\ 
& \qquad + R (\Delta_n, (\truebeta_{S_1}, \beta_{S_2}, \beta_R, \sigma)),  
\end{align} 
where $F = [F^k]_{1 \le k \le N_{\beta_{S_1}}}$ with $F^k \in \mathcal{S}$, and the second term is a residual such that: 
$$
\sup_{(\beta_{S_2}, \beta_R, \sigma) \in \Theta_{\beta_{S_2}} \times \Theta_{\beta_R} \times \Theta_\sigma} 
\bigl| 
R (\Delta_n, (\truebeta_{S_1}, \beta_{S_2}, \beta_R, \sigma)) 
\bigr| \probconv 0. 
$$ 
The first term of the right side of (\ref{eq:demo_A}) includes $\Delta_n^{-1/2}$, however $F^k (1, x, \theta)$ is identically $0$ for any $(x, \theta) \in \mathbb{R}^N \times \Theta$  
following some matrix algebra (as indicated in Lemma \ref{lemm:mat_1} of the Supplementary Material) and then (\ref{eq:conv_A}) holds. 
A similar argument related with matrix algebra (see, e.g., Lemmas \ref{lemm:mat_2}, \ref{lemma:matrix_3}) is also used in the proofs of other technical lemmas below to deal with terms of size $\mathcal{O} (\Delta_n^{-1/2})$, and then the proof of consistency proceeds without requiring that $\Delta_n = o (n^{-1/2})$. 
\end{rem}
\subsubsection{Step 2}
Making use of limit (\ref{eq:step1}), we obtain the following result whose proof is postponed to Section  \ref{appendix:pf_step_2} of the Supplementary Material.  
\begin{lemma} \label{lemm:step_2}
Assume that conditions \ref{assump:hor}, \ref{assump:coeff}\,$(2K_p)$, \ref{assump:additional_con}, \ref{assump:moments}, \ref{assump:finite_moment} and \ref{assump:ident} hold. 
If \limit, then:
\begin{align*} 
\sup_{(\beta_{S_2}, \beta_R, \sigma) \in \Theta_{\beta_{S_2}} \times \Theta_{\beta_R} \times \Theta_\sigma} 
\Bigl| 
\tfrac{\Delta_n}{n} \ell_{p, n} 
 ( \hat{\beta}_{S_1, p, n}, \beta_{S_2}, \beta_R, \sigma
) - \mathbf{Y}_2 (\beta_{S_2}, \beta_R, \sigma)
\Bigr| 
\probconv 0, 
\end{align*}
where: 
\begin{align*}
& \mathbf{Y}_2 (\beta_{S_2}, \beta_R, \sigma)  \\
& \qquad = 12 \int 
\mathbf{q} \Bigl( \mu_{S_2} (x, \beta_{S_2}) - \mu_{S_2} (x, \truebeta_{S_2}); a_{S_2}^{-1} 
\bigl(x, (\truebeta_{S_1}, \beta_{S_2}, \sigma) \bigr)  \Bigr)
\truedist (dx). 
\end{align*}
\end{lemma}
\noindent Lemma \ref{lemm:step_2} leads to the consistency of $\hat{\beta}_{S_2, p, n}$ following arguments similar to the ones used in Section \ref{sec:step1} to show the consistency of $\hat{\beta}_{S_1, p, n}$. 
\\ 

\noindent To obtain the rate of convergence in (\ref{eq:step2}), we consider the Taylor expansion of $\partial_{\beta_S} \ell_{p,n} (\hat{\theta}_{p,n})$ at $(\truebeta_S, \hat{\beta}_{R, p, n}, \hat{\sigma}_{p, n})$: 
\begin{align} \label{eq:taylor_step2}
\widetilde{\mathbf{A}}_{p,n} 
\bigl(  
\truebeta_{S_1}, 
\truebeta_{S_2}, \hat{\beta}_{R,p,n}, \hat{\sigma}_{p,n} 
\bigr) 
= 
\widetilde{\mathbf{B}}_{p,n}
\bigl( 
\hat{\theta}_{p,n}
\bigr)   
\begin{bmatrix}
\tfrac{1}{\sqrt{\Delta_n^3}} (\hat{\beta}_{S_1, p, n} - \truebeta_{S_1}) \\  
\tfrac{1}{\sqrt{\Delta_n}} (\hat{\beta}_{S_2, p, n} - \truebeta_{S_2})
\end{bmatrix}, 
\end{align}
where we have set for $\theta = (\beta_{S}, \beta_R, \sigma )\in \Theta$ with $\beta_S = ( \beta_{S_1}, \beta_{S_2}) \in \Theta_{\beta_S}$: 
\begin{align*}
& \widetilde{\mathbf{A}}_{p,n} \bigl(  \theta \bigr)  
= 
\begin{bmatrix}
- \tfrac{\sqrt{\Delta_n^3}}{n} \partial_{\beta_{S_1}} \ell_{p, n} ( \theta ) \\
- \tfrac{\sqrt{\Delta_n}}{n} \partial_{\beta_{S_2}} \ell_{p, n} ( \theta )
\end{bmatrix}; \\ 
& \widetilde{\mathbf{B}}_{p,n} \bigl(  \theta \bigr)  
= \widetilde{\mathbf{M}}_{n}
\Bigl( 
\int_0^1   
\,\partial_{\beta_S}^2 \ell_{p, n} 
\bigl( \truebeta_S + \lambda ( \beta_S - \truebeta_S ), \beta_R, \sigma \bigr) 
d \lambda \,
\Bigr) 
\widetilde{\mathbf{M}}_{n}^\top,
\end{align*}
for the matrix $\widetilde{\mathbf{M}}_{n} = \mrm{Diag} (\widetilde{v}_n)$, where:  
\begin{align*}
\widetilde{v}_n = 
\Bigl[
\underbrace{\sqrt{\tfrac{\Delta_n^3}{n}}, \ldots, \sqrt{\tfrac{\Delta_n^3}{n}}}_{N_{\beta_{S_1}}}, \;  
\underbrace{\sqrt{\tfrac{\Delta_n}{n}}, \ldots, \sqrt{\tfrac{\Delta_n}{n}}}_{N_{\beta_{S_2}}} 
\Bigr]^\top.
\end{align*}
The convergence (\ref{eq:step2}) is immediately deduced from  (\ref{eq:taylor_step2}) given the following result whose proof is provided in Section \ref{appendix:pf_step_4} of the Supplementary Material.  
\begin{lemma} \label{lemma:step_2}
Assume that conditions \ref{assump:hor}, \ref{assump:coeff}\,$(2K_p)$, \ref{assump:additional_con}, \ref{assump:moments}, \ref{assump:finite_moment} and \ref{assump:ident} hold. If \limit, then:
\begin{align}
& 
\sup_{(\beta_R, \sigma) \in \Theta_{\beta_R} \times \Theta_\sigma} 
\Bigl| \widetilde{\mathbf{A}}_{p,n} \bigl( \truebeta_{S}, \beta_R, \sigma \bigr)  \Bigr| \probconv 0;  \label{eq:step2_A} \\[-0.3cm]  
& \sup_{(\beta_R, \sigma) \in \Theta_{\beta_R} \times \Theta_\sigma} 
\Bigl| \widetilde{\mathbf{B}}_{p,n} \bigl( \truebeta_{S}, \beta_R, \sigma \bigr) - 2 \, \widetilde{\mathbf{B}}  
\bigl( \truebeta_{S}, \beta_R, \sigma \bigr)  \Bigr| \probconv 0;  \label{eq:step2_B}, 
\end{align}
where we have set for $\theta \in \Theta$, 
%
$
\widetilde{\mathbf{B}} \bigl( \theta \bigr) 
= \mrm{Diag} \Bigl[
\widetilde{\mathbf{B}}^{S_1} (\theta), \; 
\widetilde{\mathbf{B}}^{S_2} (\theta)
\Bigr]
$ 
%
with $\widetilde{\mathbf{B}}^{S_1} (\theta) \in \mathbb{R}^{N_{\beta_{S_1}} \times N_{\beta_{S_1}}}$ and 
$\textstyle \widetilde{\mathbf{B}}^{ S_2} (\theta) \in \mathbb{R}^{N_{\beta_{S_2}} \times N_{\beta_{S_2}}} $ defined as: 
\begin{align}
\widetilde{\mathbf{B}}^{S_1} (\theta) 
& = 720 \int
\bigl( \partial_{\beta_{S_1}}^\top \mu_{S_1} (x_S, \beta_{S_1}) \bigr)^\top a_{S_1}^{-1} (x, \theta) 
\, 
\partial_{\beta_{S_1}}^\top  \mu_{S_1} (x_S, \beta_{S_1})   \truedist (dx); \label{eq:limit_B_1} \\
\widetilde{\mathbf{B}}^{S_2} (\theta)  
& = 12 \int 
\bigl( 
\partial_{\beta_{S_2}}^\top \mu_{S_2} (x, \beta_{S_2}) \bigr)^\top 
a_{S_2}^{-1} (x, \theta) \, 
\partial_{\beta_{S_2}}^\top  \mu_{S_2} (x, \beta_{S_2})   \truedist (dx).  \label{eq:limit_B_2}
\end{align}
\end{lemma} 
\subsubsection{Step 3}
Finally, we show the consistency of estimators $\bigl(\hat{\beta}_{R,p,n}, \hat{\sigma}_{p,n} \bigr)$. Working with the rates of convergence obtained in (\ref{eq:step1}) and (\ref{eq:step2}), we prove the following result leading to the consistency of $\hat{\sigma}_{p, n}$. 
\begin{lemma} \label{lemma:step3_1}
Assume that conditions \ref{assump:hor}, \ref{assump:coeff}\,$(2K_p)$, \ref{assump:additional_con}, \ref{assump:moments}, \ref{assump:finite_moment} and \ref{assump:ident} hold. 
If \limit, then:
\begin{align*}
\sup_{(\beta_R, \sigma) \in \Theta_{\beta_R} \times \Theta_\sigma} \Bigl| 
\tfrac{1}{n} \ell_{p, n}
\bigl( \hat{\beta}_{S,p,n},
\beta_R, \sigma \bigr)  -  \mathbf{Y}_3 (\sigma) 
\Bigr|  \probconv 0,
\end{align*}
where we have set for $\sigma \in \Theta_\sigma$: 
\begin{align*}
\mathbf{Y}_3 (\sigma) 
=
\int \Bigl\{
\mrm{Tr} \bigl( \boldsymbol{\Lambda} (x, (\truebeta_S, \sigma) ) 
\boldsymbol{\Sigma} \bigl(x, (\truebeta_S, \truesigma)  \bigr)  \bigr)
+ \log \det \boldsymbol{\Sigma} 
\bigl(x, (\truebeta_S, \sigma) \bigr) \Bigr\} \truedist (dx).   
\end{align*} 
\end{lemma}
\noindent We provide the proof in Section  \ref{appendix:pf_step3_1} of the Supplementary Material. We will show that Lemma \ref{lemma:step3_1} indeed leads to the consistency of $\hat{\sigma}_{p, n}$. We have:
\begin{align*}
\mathbf{Y}_3 (\sigma) - \mathbf{Y}_3 (\truesigma)
= 2 \int \varphi (y; x, \truesigma) \log \frac{\varphi (y; x, \sigma)}{\varphi(y; x, \truesigma)} \, \truedist(dx), \quad 
\sigma \in \Theta, 
\end{align*}
%
%
%
with $y \mapsto \varphi (y; x, \sigma), \, (x, \sigma) \in \mathbf{R}^N \times \Theta$ being the density of the distribution $\mathcal{N} (\mathbf{0}_{N}, \boldsymbol{\Sigma} (x, (\truebeta_S, \sigma))$. Thus, under condition \ref{assump:ident}, $\mathbf{Y}_3 (\sigma) - \mathbf{Y}_3 (\truesigma)$ should be positive if $\sigma \neq \truesigma$. Hence, for every $\varepsilon > 0$, there exists a constant $\delta > 0$ so that: 
\begin{align*}
& \mathbb{P}_{\trueparam} \bigl( | \hat{\sigma}_{p, n } - \truesigma| > \varepsilon \bigr)
 \le 
\mathbb{P}_{\trueparam} \bigl( \mathbf{Y}_3 (\hat{\sigma}_{p, n}) - \mathbf{Y}_3 (\truesigma)  > \delta \bigr) \\
& \le 
\mathbb{P}_{\trueparam} 
\Bigl( 
\tfrac{1}{n} \ell_{p,n} (\hat{\beta}_{p,n}, \truesigma)
- \tfrac{1}{n} \ell_{p,n} (\hat{\theta}_{p,n}) 
+ \mathbf{Y}_3 (\hat{\sigma}_{p, n}) - \mathbf{Y}_3 (\truesigma)  > \delta 
\Bigr) \to 0, 
\end{align*}
as \limit. 

To show the consistency of $\hat{\beta}_{R, p, n}$, we consider, for 
$\theta = (\beta_{S_1}, \beta_{S_2}, 
\beta_{R}, \sigma) \in \Theta$:
\begin{align} \label{eq:K}
\mathbf{K} ( \theta )  
:= \tfrac{1}{n \Delta_n}  \ell_{p, n}  
\bigl( \theta \bigr) 
- \tfrac{1}{n \Delta_n}  \ell_{p, n}  
\bigl( \beta_{S_1},  \beta_{S_2}, 
\truebeta_R, \sigma \bigr).   
\end{align}
The consistency of  $\hat{\beta}_{R, p, n}$ follows from the result below whose proof is given in Section  \ref{appendix:pf_step3_2} of the Supplementary Material. 
\begin{lemma} \label{lemma:step3_2}
Assume that conditions \ref{assump:hor}, \ref{assump:coeff}\,$(2K_p)$, \ref{assump:additional_con}, \ref{assump:moments}, \ref{assump:finite_moment} and \ref{assump:ident} hold. 
If \limit, then:
\begin{align*}
\sup_{\beta_R \in \Theta_{\beta_R}} 
\Bigl| 
\mathbf{K} (\hat{\beta}_{S_1, p, n},  \hat{\beta}_{S_2, p, n},  \beta_R,  \hat{\sigma}_{p,n}) - \mathbf{Y}_4 (\beta_R)
\Bigr| \probconv 0, 
\end{align*}
where we have set for $\beta_R \in \Theta_{\beta_R}$: 
\begin{align*}
\mathbf{Y}_4 (\beta_R) 
= \int 
\mathbf{q} \bigl( \mu_{R} (x, \beta_R) - \mu_{R} (x, \truebeta_R) ;  a^{-1}_{R} (x, \truesigma) \bigr) 
\, \truedist (dx). 
\end{align*} 
%
%
\end{lemma}
Thus, the proof of consistency for the proposed estimator is now complete.

\subsection{Proof of Theorem \ref{thm:clt} (CLT)} \label{sec:pf_asymptotic_normality}
We define the $N_{\theta} \times N_{\theta}$ matrix $\mathbf{M}_n$ as $\mathbf{M}_n = \mrm{Diag} (v_n)$, where: 
\begin{align*}
v_n = \Bigl[
\underbrace{\sqrt{\tfrac{n}{\Delta_n^3}}, \ldots, \sqrt{\tfrac{n}{\Delta_n^3}}}_{N_{\beta_{S_1}}}, \; 
\underbrace{\sqrt{\tfrac{n}{\Delta_n}}, \ldots, \sqrt{\tfrac{n}{\Delta_n}}}_{N_{\beta_{S_2}}}, \;
\underbrace{{\sqrt{n\Delta_n}}, \ldots, {\sqrt{n\Delta_n}}}_{N_{\beta_R}}, \;
\underbrace{{\sqrt{n}}, \ldots, {\sqrt{n}}}_{N_\sigma}
\Bigr]. 
\end{align*}
Noting that $\partial_\theta \ell_{p, n} (\hat{\theta}_{p, n}) = \mathbf{0}_{N_{\theta}}$, a Taylor expansion of $\partial_\theta \ell_{p, n} (\hat{\theta}_{p, n})$ at the true parameter $\trueparam$ gives: 
\begin{align} \label{eq:taylor_clt}
\mathbf{I}_{p, n} ( \trueparam ) 
 = \int_0^1 
\mathbf{J}_{p, n} 
\bigl(  
\trueparam  + \lambda ( \hat{\theta}_{p, n} - \trueparam)  \bigr) d \lambda \times 
\mathbf{M}_{n} (\hat{\theta}_{p, n} - \trueparam), 
\end{align} 
where we have set for $\theta \in \Theta$:
\begin{align*}
\mathbf{I}_{p, n} (\theta) \equiv  
- \mathbf{M}_n^{-1} \,  \partial_\theta \ell_{p, n} (\theta), 
\qquad  
\mathbf{J}_{p, n} (\theta) \equiv 
\mathbf{M}_{n}^{-1} \, \partial^2_\theta \ell_{p, n} (\theta) \, \mathbf{M}_{n}^{-1}.  
\end{align*} 
To prove the asymptotic normality, we exploit the following two results.
%
\begin{lemma} \label{lemm:slln}
Assume that conditions \ref{assump:hor}, \ref{assump:coeff}\,$(2K_p)$, \ref{assump:additional_con}, \ref{assump:moments}, \ref{assump:finite_moment} and \ref{assump:ident} hold. 
If \limit, then: 
\begin{align*}
\sup_{\lambda \in [0,1]} \Bigl|  
\mathbf{J}_{p, n} \bigl( \trueparam  + \lambda ( \hat{\theta}_{p, n} - \trueparam)  \bigr) 
- 2 \Gamma (\trueparam) \Bigr| 
\probconv 0. 
\end{align*} 
\end{lemma}  
\begin{prop} \label{prop:clt}
Assume that conditions \ref{assump:hor}, \ref{assump:coeff}\,$(2K_p)$, \ref{assump:additional_con}, \ref{assump:moments} and \ref{assump:finite_moment} hold. 
If \limit , with the additional design condition $\Delta_n = o (n^{-1/p})$, then: 
\begin{align*}
\textstyle 
\mathbf{I}_{p, n} (\trueparam) \distconv \mathcal{N} \bigl( \mathbf{0}_{N_\theta}, 4 \Gamma (\trueparam) \bigr). 
\end{align*} 
\end{prop}
\noindent We provide the proof of Lemma \ref{lemm:slln} in Section  \ref{appendix:pf_slln} of the Supplementary Material. We show the proof of Proposition \ref{prop:clt} in the next subsection where we highlight the manner in which we make use of the condition  $\Delta_n = o (n^{-1/ p})$, $p \ge 2$. By applying these two results to equation (\ref{eq:taylor_clt}), the proof of Theorem \ref{thm:clt} is complete.   
\subsubsection{Proof of Proposition \ref{prop:clt}} 
For simplicity of notation, we write 
$
\textstyle 
\partial_{\theta, k} \ell_{p, n} (\theta) 
= \sum_{i = 1}^n {\xi}_{i}^k (\theta), \; \theta \in \Theta, \, 1 \le k \le N_\theta$, where: 
\begin{align} \label{eq:xi}
\xi_{i}^k (\theta)
& =  
[ \mathbf{M}_n^{-1}]_{kk} \times  \nonumber \\ 
& \sum_{j = 0}^{K_p} 
\, \Delta_n^j \cdot 
\partial_{\theta, k} \bigl\{ 
\mathbf{m}_{p, i} (\Delta_n, \theta)^\top 
\, 
\mathbf{G}_{i-1, j} (\theta) 
\, 
\mathbf{m}_{p, i} (\Delta_n, \theta) 
+ \mathbf{H}_{i-1, j} (\theta) 
\bigr\}. 
\end{align} 
Due to Theorems 4.2 \& 4.4 in \cite{hall:14}, Proposition \ref{prop:clt} holds once we prove the following convergences:  
\begin{description}
\item[(i)] 
If \limit \, with $\Delta_n = o(n^{-1/p})$, then:   
\begin{align} \label{eq:score_E}
\sum_{i = 1}^n 
\mathbb{E}_{\trueparam} \bigl[ \xi_i^k (\trueparam) | \mathcal{F}_{t_{i-1}} \bigr] 
\probconv 0, \quad 1 \le k \le N_\theta. 
\end{align}
\item[(ii)] If \limit, then:    
\begin{gather}
\sum_{i = 1}^n 
\mathbb{E}_{\trueparam} \bigl[ \xi_i^{k_1} (\trueparam) \xi_i^{k_2} (\trueparam) | \mathcal{F}_{t_{i-1}} \bigr] 
\probconv 4 \bigl[ \Gamma(\trueparam) \bigr]_{k_1 k_2}, \quad 1 \le k_1, k_2 \le N_\theta; \label{eq:hess_E} \\[-0.2cm]  
\sum_{i = 1}^n 
\mathbb{E}_{\trueparam} \bigl[ 
\bigl( \xi_i^{k_1} (\trueparam) \xi_i^{k_2} (\trueparam) \bigr)^2 
| \mathcal{F}_{t_{i-1}} \bigr] 
\probconv 0, \quad 1 \le k_1, k_2  \le N_\theta.  
\label{eq:fourth_E}
\end{gather}   
\end{description}
%

%
\noindent Notice that the design condition $\Delta_n = o (n^{-1/p})$ is used in the proof that the expectation of the score function converges to $0$, i.e., for convergence (\ref{eq:score_E}). We provide the proof of the two convergences (\ref{eq:hess_E}) and (\ref{eq:fourth_E}) in Section \ref{sec:pf_clt_limits} of the Supplementary Material and focus on the proof of (\ref{eq:score_E}) in this section. 
\\

\noindent
{\it Proof of  convergence (\ref{eq:score_E})}. 
Let $1 \le k \le N_{\theta}$. We write for $\theta \in \Theta$ and $1 \le i \le n$: 
\begin{align*} 
\partial_{\theta, k} \mathbf{m}_{p, i} (\Delta_n, \theta) 
\equiv v_{p, k} (\Delta_n, \sample{X}{i-1}, \theta). 
\end{align*}
Note that the $\mathbb{R}^N$-valued function 
$v_{p, k} (\Delta_n, \sample{X}{i-1}, \theta)$ is independent of $\sample{X}{i}$. It then follows that: 
\begin{align*}
\sum_{i = 1}^n \mathbb{E}_{\trueparam}  \bigl[ \xi_i^k (\trueparam)  | \mathcal{F}_{t_{i-1}} \bigr] 
= \sum_{i = 1}^n \bigl\{ F^{(1), k}_{i-1} (\trueparam) +  F^{(2), k}_{i-1} (\trueparam) \bigr\},  
\end{align*}
where we have set: 
\begin{align*}
& F^{(1), k}_{i-1} (\trueparam) 
= 2 [\mathbf{M}_n^{-1}]_{kk} 
\times \sum_{j = 0}^{K_p} \sum_{l_1, l_2 = 1}^N \Delta_n^j \, 
v_{p, k}^{l_1} (\Delta_n, \sample{X}{i-1}, \trueparam) 
\bigl[ \mathbf{G}_{i-1, j} (\trueparam) \bigr]_{l_1 l_2} 
\mathbb{E}_{\trueparam} 
\bigl[ 
\mathbf{m}_{p, i}^{l_2} (\Delta_n, \trueparam) | \mathcal{F}_{t_{i-1}} 
\bigr]; \\
& F^{(2), k}_{i-1} (\trueparam) 
=  
[\mathbf{M}_n^{-1}]_{kk} 
\times 
\sum_{j = 0}^{K_p} \Delta_n^j
\biggl\{ 
\sum_{l_1, l_2 = 1}^N 
\bigl[ \partial_{\theta, k} \mathbf{G}_{i-1, j} 
(\trueparam) \bigr]_{l_1 l_2} 
\mathbb{E}_{\trueparam} 
\bigl[ 
\mathbf{m}_{p, i}^{l_1} (\Delta_n, \trueparam) 
\mathbf{m}_{p, i}^{l_2} (\Delta_n, \trueparam) | \mathcal{F}_{t_{i-1}} 
\bigr]  \\[-0.3cm]  
& \qquad \qquad \qquad  \qquad \qquad    + \partial_{\theta, k} \mathbf{H}_{i-1, j} (\trueparam) \biggr\}.  
\end{align*}
For the first term, it follows from (\ref{eq:m_approx}) that: 
\begin{align*} 
F_{i-1}^{(1), k} (\trueparam) = 
\begin{cases}
\tfrac{1}{n} R^{(1), k} (\sqrt{n \Delta_n^{2K_p + 1}}, \sample{X}{i-1}, \trueparam), & 
{1 \le k \le N_{\beta}}; \\ 
\tfrac{1}{n} R^{(1), k} (\sqrt{n \Delta_n^{2K_p + 2}}, \sample{X}{i-1}, \trueparam), & 
{N_{\beta}  + 1 \le k \le N_\theta}, 
\end{cases}
\end{align*}
for some $R^{(1), k} \in \mathcal{S}$, under condition \ref{assump:coeff}$\,(2K_p)$. 
Thus, it follows from Lemmas \ref{lemma:aux_1}-\ref{lemma:aux_2} in the Supplementary Material that if \limit \, with $\Delta_n = o (n^{-1/p})$, then: 
\begin{align*}  
\sum_{i = 1}^n F_{i-1}^{(1), k}  (\trueparam) \probconv 0, \qquad 1 \le k \le N_\theta. 
\end{align*}
%
%
%
We consider the convergence of the term $\textstyle \sum_{i = 1}^n F_{i-1}^{(2), k}  (\trueparam)$. We introduce the following subsets in the event space $\Omega$:
\begin{align}
\Omega_{\varepsilon} 
& := \Bigl\{ \omega \in \Omega \, \Big| \, 
\bigl| \textstyle \sum_{i = 1}^n  F^{(2), k}_{i-1} (\trueparam) \bigr| > \varepsilon   \Bigr\},
\quad  \varepsilon > 0; \nonumber \\   
D_{\Delta_n, i} 
& := 
\Bigl\{ \omega \in \Omega \, \Big| \, \det \boldsymbol{\Xi}_{K_p, i} (\Delta_n, \trueparam)  > 0 \Bigr\}, \quad  0 \le i \le n-1. \nonumber 
\end{align}
We also set $\textstyle  D_{\Delta_n} = \bigcap_{i = 1}^n D_{\Delta_n, i-1} $. We then have for any $\varepsilon > 0$: 
\begin{align} \label{eq:prob_bd}
\mathbb{P}_{\trueparam} (\Omega_{\varepsilon}) 
\le \mathbb{P}_{\trueparam} \bigl( \Omega_{\varepsilon} \cap D_{\Delta_n} \bigr) 
+  \mathbb{P}_{\trueparam} \bigl( \Omega_{\varepsilon} \cap D^c_{\Delta_n} \bigr).  
\end{align}
For the first term of the right hand side of (\ref{eq:prob_bd}), we have the following result whose proof is provided in Section \ref{append:pf_F2_conv} of the Supplementary Material. 
\begin{lemma} \label{lemma:F_2_conv}
Assume that conditions \ref{assump:hor}, \ref{assump:coeff}\,$(2K_p)$, \ref{assump:additional_con}, \ref{assump:moments}, \ref{assump:finite_moment} and \ref{assump:ident} hold. 
For any $\varepsilon > 0$, if \limit \, with $\Delta_n = o (n^{-1/p})$, then 
$
\mathbb{P}_{\trueparam} 
\bigl(
\Omega_{\varepsilon} \cap D_{\Delta_n} 
\bigr)
\to 0. 
$ 
\end{lemma} 
\noindent We next consider the second term of the right hand side of (\ref{eq:prob_bd}). We have:
\begin{align*}
\mathbb{P}_{\trueparam} \bigl( \Omega_{\varepsilon} \cap D^c_{\Delta_n} \bigr) 
& \leq \sum_{i = 1}^{n} \mathbb{P}_{\trueparam} \bigl(  D_{\Delta_n, i-1}^c \bigr) \\ 
& = \sum_{i = 1}^{n} \mathbb{P}_{\trueparam} 
\Bigl( 
\bigl\{ 
\omega \in \Omega \, | \, 
\det \boldsymbol{\Xi}_{K_p, i-1} (\Delta_n, \trueparam)  = 0  
\bigr\} \Bigr).   
\end{align*}
Since it follows that: 
\begin{align*} 
 \det \boldsymbol{\Xi}_{K_p, i-1} (\Delta_n, \trueparam) 
 - \det \boldsymbol{\Sigma}_{i-1} (\trueparam) 
 = R (\Delta_n, \sample{X}{i-1}, \trueparam) 
\end{align*}
for ${R} \in \mathcal{S}$, 
and $\det \boldsymbol{\Sigma}_{i-1} (\trueparam) > 0$, $1 \le i \le n$, under condition \ref{assump:hor}-$(ii)$ due to Lemma \ref{lemma:positive_II}, we have that 
$\mathbb{P}_{\trueparam} \bigl( \Omega_{\varepsilon} 
\cap D^c_{\Delta_n} \bigr) \to 0$ as $\Delta_n \to 0$.   
%
%
Thus, we obtain that $\textstyle \sum_{i = 1}^n F_{i-1}^{(2), k}  (\trueparam) \probconv 0$ if \limit \, with $\Delta_n = o (n^{-1/p})$, and the proof of convergence (\ref{eq:score_E}) is complete. 
\section{Conclusions} 
\label{sec:conclusion}
We have proposed general contrast estimators for a wide class of hypo-elliptic diffusions specified in (\ref{eq:hypo-I}) and (\ref{eq:hypo-II}), and showed that the estimators achieve consistency and asymptotic normality in a high-frequency complete observation regime under the weakened design condition $\Delta_n = o (n^{-1/p})$, $p \ge 2$. 
\textcolor{black}{For elliptic diffusions, contrast estimators under such a general weak design condition have been investigated in earlier works, e.g.~\cite{kess:97, uchi:12}, however 
in the context of hypo-elliptic SDEs, established results (e.g.~\cite{dit:19, glot:21, melnykova2020parametric, igu:23}) before our work typically relied on a condition of  `rapidly increasing experimental design', i.e.~the requirement that $\Delta_n = o (n^{-1/2})$. }
The numerical experiments in 
Section~\ref{sec:sim} illustrated cases where for a given step-size $\Delta_n$ in-between observations and given $n$,
estimators requiring the condition $\Delta_n = o(n^{-1/2})$ can induce unreliable biased parametric inference procedures. In contrast, the bias was removed upon consideration of estimators which required the weaker design condition 
$\Delta_n = o (n^{-1/p})$, $p \ge 3$. 
We stress that computations w.r.t.~the contrast function that delivers estimators under the weakened design condition can be automated via use of symbolic programming and automatic differentiation, so that that users need only specify the drift and diffusion coefficients of the given SDE model. 

Our research opens up several potential future directions.
\begin{itemize}[leftmargin=0.5cm]
\item[1.] {\textit{Development and analytical study of estimators under alternative observation designs, such as partially observed coordinates and/or low-frequency observation settings.} Such designs are important from a practical perspective but have yet to be investigated analytically even for elliptic diffusions. The results in this paper are already relevant for such different observation designs as they can be regarded as providing a minimum requirement on the step-size $\Delta_n$ (for given $n$) that needs to be used when considering smaller amounts of data compared to the complete observation regime treated in this work. E.g., in a low-frequency setting where \emph{data augmentation} procedures (within an Expectation-Maximisation or MCMC setting) will introduce latent SDE values, 
the user-specified step-size $\Delta_n$ will need to satisfy the weakened design conditions obtained in our work (i.e.~if the likelihood-based method is not supported in the case where imputed variables were indeed observations, there is absolutely no basis for the method that uses imputed values to provide reliable estimates).  
In general, it is of interest to explore the type of convergence rates and step-size conditions obtained and required, respectively, in CLTs under such different observation designs.}    

\item[2.] 
\textcolor{black}{\textit{Parametric inference for Mckean-Vlasov type hypo-elliptic SDEs and related interacting particle system models.} Recently, parametric inference for SDEs with coefficients depending on the (empirical) law of the process has become an active research area in Statistics, e.g. \cite{amo:23, sha:23}, though these works focus on elliptic models. Law-dependent hypo-elliptic SDEs appear in several applications \citep{bal:12}, and it is of high interest to establish estimators and analytic results that cover such hypo-elliptic models.} 

\item[3.] 
\textcolor{black}{\textit{Non-parametric estimation for hypo-elliptic SDEs.} Established results in the literature have focused on elliptic diffusions. In the scalar case, \cite{hoff:99} studied non-parametric estimation of drift and diffusion functions from high-frequency observations and obtained optimal convergence rates, with extensions to low-frequency observations regimes given in \cite{gobet:04}. \cite{ab:19} studied a Bayesian non-parametric approach using high-frequency data. For multivariate elliptic models, see \cite{nick:20} for a work in a Bayesian setting and \cite{oga:24} which proposed a non-parametric drift estimator based on deep-learning and obtained a convergence rate under the condition of $\Delta_n = o (n^{-1/2})$. Providing results for hypo-elliptic models along the above lines makes up a very interesting research direction.} 
\end{itemize}
\subsection*{Acknowledgements}
Yuga Iguchi is supported by the Additional Funding Programme for Mathematical Sciences, delivered by EPSRC (EP/V521917/1) and the Heilbronn Institute for Mathematical Research. 
%
%
%
\appendix 

\section{Preliminaries} \label{appendix:aux}
\subsection{Two Auxiliary Results}
 We introduce two auxiliary lemmas used in the proof of technical results required by the main statements, Theorems \ref{thm:consistency} \& \ref{thm:clt}.
\begin{lemma} \label{lemma:aux_1}
Let $f : \mathbb{R}^N \times \Theta \to \mathbb{R}$ be differentiable w.r.t.~$x \in \mathbb{R}^N$ and $\theta \in \Theta $, with derivatives of polynomial growth in $x \in \mathbb{R}^N$ uniformly in $\theta \in \Theta$. Under conditions \ref{assump:hor}, \ref{assump:coeff}\,$(2K_p)$, \ref{assump:additional_con}, \ref{assump:moments} and \ref{assump:finite_moment} in the main text, if \limit, then
\begin{align*}
\sup_{\theta \in \Theta} 
\left| 
\frac{1}{n} \sum_{i=1}^n f (\sample{X}{i-1}, \theta) -
\int f (x, \theta) \truedist (dx)
\right| \probconv  0.
\end{align*}
\end{lemma}
\textit{Proof}.
This is a multivariate version of Lemma 8 in \cite{kess:97}, and we omit the proof.  $\Box$
\begin{lemma} \label{lemma:aux_2}
Let $p \ge 2$, $K_p = [p/2]$, $1 \leq j_1, j_2 \leq N$, and $f : \mathbb{R}^N \times \Theta \to \mathbb{R}$ satisfy the same assumption as in the statement of Lemma \ref{lemma:aux_1}. Under conditions \ref{assump:hor}, \ref{assump:coeff}\,$(2K_p)$, \ref{assump:additional_con}, \ref{assump:moments} and \ref{assump:finite_moment} in the main text, if \limit, then
\begin{gather*}
\sup_{\theta \in \Theta}  
\, \Bigl| 
\tfrac{1}{n} \sum_{i=1}^n  
f (\sample{X}{i-1}, \theta) 
\mathbf{m}_{p, i}^{j_1} (\Delta_n, \trueparam) 
\mathbf{m}_{p, i}^{j_2} (\Delta_n, \trueparam) 
- \int f (x, \theta) 
\bigl[ \boldsymbol{\Sigma} (x, \trueparam) \bigr]_{j_1 j_2} \truedist (dx)
\Bigr| \probconv 0;  \label{eq:canonical_conv1} \\
\sup_{\theta \in \Theta} 
\Bigl| 
\tfrac{1}{n  \sqrt{\Delta_n}} 
\sum_{i=1}^n   f (\sample{X}{i-1}, \theta) 
\mathbf{m}_{p, i}^{j_1} (\Delta_n, \trueparam)  
\Bigr| \probconv 0. 
\label{eq:canonical_conv2} 
\end{gather*}
\end{lemma}
\textit{Proof}. 
We define
\begin{align*}
\mathscr{T}_i^{(1)} (\theta) 
= \tfrac{1}{n} f (\sample{X}{i-1}, \theta) 
\mathbf{m}_{p, i}^{j_1} (\Delta_n, \trueparam) 
\mathbf{m}_{p, i}^{j_2} (\Delta_n, \trueparam), 
\quad  
\mathscr{T}_i^{(2)} (\theta) 
 = \tfrac{1}{n  \sqrt{\Delta_n}} f (\sample{X}{i-1}, \theta) \mathbf{m}_{p, i}^{j_1} (\Delta_n, \trueparam). 
\end{align*} 
Due to Lemma \ref{lemma:aux_1}, for any $\theta \in \Theta$, if \limit, then 
\begin{align*}
\sum_{i=1}^n 
\mathbb{E}_{\trueparam}
\bigl[
\mathscr{T}_i^{(1)} (\theta)  | \mathcal{F}_{t_{i-1}}
\bigr] 
& = \tfrac{1}{n} \sum_{i=1}^n 
f (\sample{X}{i-1}, \theta) \bigl[ \boldsymbol{\Sigma}_{i-1} (\trueparam) 
\bigr]_{j_1 j_2}
+ \tfrac{1}{n} \sum_{i=1}^n  
R (\Delta_n, \sample{X}{i-1}, \theta) \\ 
& \quad  \quad\quad\quad\quad\quad\quad\quad\quad\quad\quad
\probconv \int_{\mathbb{R}^N} f (x, \theta) 
\bigl[ \boldsymbol{\Sigma} (x, \trueparam) \bigr]_{j_1 j_2} \truedist (dx); \\[0.2cm] 
\sum_{i=1}^n 
\mathbb{E}_{\trueparam}
\bigl[
\bigl( \mathscr{T}_i^{(1)} (\theta) \bigr)^2 
| \mathcal{F}_{t_{i-1}}
\bigr]  
& = \tfrac{1}{n^2} \sum_{i=1}^n  
\widetilde{R} (1, \sample{X}{i-1}, \theta) \probconv 0, 
\end{align*}
where ${R}, \widetilde{R} \in \mathcal{S}$. Similarly, we have, if \limit, 
\begin{align*}
& \sum_{i = 1}^n \mathbb{E}_{\trueparam}
\bigl[
\mathscr{T}_i^{(2)} (\theta)  | \mathcal{F}_{t_{i-1}}
\bigr]  \probconv 0; \\ 
& \sum_{i = 1}^n \mathbb{E}_{\trueparam}
\bigl[
\bigl( \mathscr{T}_i^{(2)} (\theta)
\bigr)^2 
| \mathcal{F}_{t_{i-1}}
\bigr] = \tfrac{1}{n \Delta_n } \tfrac{1}{n} 
\sum_{i = 1}^n R (1, \sample{X}{i-1}, \theta)
\probconv 0,   
\end{align*}
where $R \in \mathcal{S}$. Thus, due to Lemma 9 in \cite{genon:93}, we have for any $\theta \in \Theta$,   
\begin{gather*}
\tfrac{1}{n} \sum_{i=1}^n  
f (\sample{X}{i-1}, \theta) 
\mathbf{m}_{p, i}^{j_1} (\Delta_n, \trueparam) 
\mathbf{m}_{p, i}^{j_2} (\Delta_n, \trueparam) 
\probconv \int f (x, \theta) 
\bigl[ \boldsymbol{\Sigma} (x, \trueparam) \bigr]_{j_1 j_2} \truedist (dx); \\ 
\tfrac{1}{n  \sqrt{\Delta_n}} 
\sum_{i=1}^n   f (\sample{X}{i-1}, \theta) 
\mathbf{m}_{p, i}^{j_1} (\Delta_n, \trueparam) 
\probconv 0, 
\end{gather*}
as \limit. Uniform convergence w.r.t.~$\theta$ is shown by the same arguments as in the proofs of Lemmas 9, 10 in \cite{kess:97} and now the proof is complete. $\Box$
\subsection{Proof of Lemma \ref{lemma:positive_I}}
\label{append:pf_positive_I}
It is immediate from condition \ref{assump:hor}$(i)$ in the main text that the matrix $a_R(x, \sigma)$ is positive definite for any $(x, \sigma) \in \mathbb{R}^N \times \Theta_\sigma$. 
The condition also implies that 
\begin{align*}
\mrm{span} \Big\{
\mrm{proj}_{N_R + 1, N} \bigl\{ [A_0, A_k] (x, \theta)  \bigr\}, \, 1 \le k \le d \Bigr\}
= 
\mrm{span} \Big\{ 
\mathcal{L}_k \mu_S (x, \theta), 
\, 1 \le k \le d 
\Bigr\}
= 
\mathbb{R}^{N_S}, 
\end{align*}
for all $(x, \theta) \in \mathbb{R}^N \times \Theta$. Thus, the matrix $a_S (x, \theta)$ is also positive definite for any $(x, \theta) \in \mathbb{R}^N \times \Theta$. Finally, the following calculation completes the proof: 
\begin{align*}
\det \boldsymbol{\Sigma}^{(\mrm{I})} (x, \theta) = 
12 \det a_R (x, \sigma) \det a_S (x, \theta) > 0, \quad  (x, \theta) \in \mathbb{R}^N \times \Theta. 
\end{align*}
\section{Proof of Technical Results for Theorem \ref{thm:consistency}} 
\label{appendix:pf_techincal_consistency}
This section is devoted to the proof of Lemmas \ref{lemm:step1}--\ref{lemma:step3_2} stated in Section \ref{sec:pf_consistency}, i.e.~in the context of the proof of consistency of the proposed contrast estimators (Theorem \ref{thm:consistency}). We note that the proof below proceeds with the second class of hypo-elliptic SDEs (\ref{eq:hypo-II}) and makes use of the notation of the contrast function $\ell_{p, n}^{\, (\mrm{II})} (\theta), \, p \ge 2$, defined in the main text with the subscript $(\mrm{II})$ being dropped. Throughout the proofs, we assume $p \ge 3$ and use often the following notation:
\begin{align*}
d (x, \theta) = 
\begin{bmatrix*} 
d_{S_1} (x_S, \beta_{S_1})  \\[0.2cm]
d_{S_2} (x, \beta_{S_2}) \\[0.2cm]  
d_{R} (x, \beta_{R})  
\end{bmatrix*}
= 
\begin{bmatrix} 
\mu_{S_1} (x_S, \beta_{S_1}) - \mu_{S_1} (x_S, \truebeta_{S_1}) \\[0.2cm]
\mu_{S_2} (x, \beta_{S_2}) - \mu_{S_2} (x, \truebeta_{S_2}) \\[0.2cm]  
\mu_{R} (x, \beta_{R}) - \mu_{R} (x, \truebeta_{R}) 
\end{bmatrix},
\qquad (x, \theta) \in \mathbb{R}^N \times \Theta. 
\end{align*}
\subsection{Proof of Lemma \ref{lemm:step1}}  
It holds that 
$$
\tfrac{\Delta_n^3}{n} \ell_{p, n} (\theta) 
= \tfrac{1}{n} \sum_{i = 1}^n \sum_{k = 1}^4 \mathcal{E}^{(k)}_{i-1} (\theta), \quad \theta \in \Theta, 
$$
with 
\begin{align*}
\mathcal{E}^{(1)}_{i-1} (\theta)  
& \equiv 
d_{S_1} ( X_{S_1, t_{i-1}}, \beta_{S_1})^\top 
\boldsymbol{\Lambda}_{S_1 S_1} (\sample{X}{i-1}, \theta) 
d_{S_1} ( X_{S_1, t_{i-1}}, \beta_{S_1}); 
\\[0.3cm] 
\mathcal{E}^{(2)}_{i-1} (\theta)  
& \equiv  \sum_{j_1, j_2 = 1}^N 
R_{j_1 j_2} (\Delta_n^3, \sample{X}{i-1}, \theta)
\, 
\mathbf{m}_{p, i}^{j_1} (\Delta_n, \trueparam) 
\, 
\mathbf{m}_{p, i}^{j_2} (\Delta_n, \trueparam);  
\\[0.1cm] 
\mathcal{E}^{(3)}_{i-1} (\theta)    
& \equiv \sum_{j = 1}^N 
R_j (\sqrt{\Delta_n^3}, \sample{X}{i-1}, \theta)
\, \mathbf{m}_{p, i}^{j} (\Delta, \trueparam),  \quad 
\mathcal{E}^{(4)}_{i-1} (\theta) 
\equiv R (\Delta_n, \sample{X}{i-1}, \theta),\\[-0.4cm]
\end{align*}
for some functions $R_{j_1 j_2}, R_j, R \in \mathcal{S}$. 
From Lemmas  \ref{lemma:aux_1} and \ref{lemma:aux_2}, 
we immediately have that if \limit, then
\begin{align*}
& \sup_{\theta \in \Theta} \, 
\Bigl| 
\tfrac{1}{n} \sum_{i = 1}^n 
\mathcal{E}^{(1)}_{i-1} (\theta) -
\int_{\mathbb{R}^N} 
d_{S_1} (x_S, \beta_{S_1})^\top
\boldsymbol{\Lambda}_{S_1 S_1} (x, \theta) 
d_{S_1} (x_S, \beta_{S_1}) \truedist (dx)
\Bigr| \probconv 0; \\ 
& \sup_{\theta \in \Theta} \, 
\Bigl| 
\tfrac{1}{n} \sum_{i = 1}^n \mathcal{E}^{(k)}_{i-1} (\theta) 
\Bigr| 
\probconv 0, \ \ 2 \le k \le 4.  
\end{align*}
Thus, we obtain the limit (\ref{eq:conv_1}) and the proof is now complete.  
\label{appendix:pf_step1}

\subsection{Proof of Lemma \ref{lemma:step1_B}}  
\label{appendix:pf_step1_B}
\subsubsection{Proof of limit (\ref{eq:conv_A})}
\label{sec:pf_conv_A}
It holds that for $(\beta_{S_2}, \beta_R, \sigma)
 \in \Theta_{\beta_{S_2}} \times \Theta_{\beta_R} \times \Theta_{\sigma}$ and $1 \le k \le N_{\beta_{S_1}}$, 
\begin{align*}
& \mathbf{A}_{p,n}^k (\truebeta_{S_1}, \beta_{S_2}, \beta_R, \sigma) \\ 
& = 
\tfrac{1}{n} \sum_{i = 1}^n  
\Biggl\{ 
\tfrac{1}{\sqrt{\Delta_n}} {F}^k(\sample{X}{i-1}, \theta) 
+ \sqrt{\Delta_n} \cdot {R}^k (1, \sample{X}{i-1}, \theta) 
+ \sum_{j = 1}^N {R}^k_j (1, \sample{X}{i-1}, \theta) \mathbf{m}_{p, i}^j (\Delta_n, \trueparam) \\
& \qquad \qquad 
+ \sum_{j_1, j_2 = 1}^N 
{R}^k_{j_1 j_2} (\sqrt{\Delta_n^3}, \sample{X}{i-1}, \theta) 
\mathbf{m}_{p, i}^{j_1} (\Delta_n, \trueparam)
\mathbf{m}_{p, i}^{j_2} (\Delta_n, \trueparam) 
\Biggr\} |_{\theta = (\truebeta_{S_1}, \beta_{S_2}, \beta_R, \sigma)},      
\end{align*}
where ${R}^k, {R}^k_{j}, {R}^k_{j_1j_2} \in \mathcal{S}$ and 
\begin{align*}
{F}^k (x, \theta)
\equiv
- \bigl( \partial_{\beta_{S_1}, k} \, 
\mu_{S_1} (x, \beta_{S_1}) \bigr)^\top 
\, 
{M}(x, \theta) 
\, 
d_{S_2}  (x, \beta_{S_2}), \qquad (x, \theta) \in 
\mathbb{R}^N \times \Theta,  
\end{align*}
with the $\mathbb{R}^{N_{S_1} \times N_{S_2}}$-valued function $M$ being defined as: 
\begin{align} \label{eq:M}
{M} (x, \theta) 
= \boldsymbol{\Lambda}_{S_1 S_1} (x, \theta) \partial_{x_{S_2}}^\top \mu_{S_1} (x_S, \truebeta_{S_1})  + 
2 \boldsymbol{\Lambda}_{S_1 S_2} (x, \theta) .  
\end{align} 
In the derivation of ${F}_{i-1}^k (\theta)$, we used: \vspace{0.1cm} 
\begin{align} 
& \mathcal{L} \mu_{S_1} (x, \trueparam) 
- \mathcal{L} \mu_{S_1} (x, \theta)|_{\theta = (\truebeta_{S_1}, \beta_{S_2}, \beta_R, \sigma)}  \nonumber \\[0.2cm] 
& \qquad  = \bigl(
\partial_{x_{S_1}}^\top 
\mu_{S_1} (x_S, \truebeta_{S_1})
\bigr) \mu_{S_1} (x_S, \truebeta_{S_1}) 
+ \bigl(
\partial_{x_{S_2}}^\top 
\mu_{S_1} (x_S, \truebeta_{S_1})
\bigr) \mu_{S_2} (x, \truebeta_{S_2}) \nonumber \\ 
& \quad\quad\quad\qquad - \bigl(
\partial_{x_{S_1}}^\top 
\mu_{S_1} (x_S, \truebeta_{S_1})
\bigr) \mu_{S_1} (x_S, \truebeta_{S_1}) 
- \bigl(
\partial_{x_{S_2}}^\top 
\mu_{S_1} (x_S, \truebeta_{S_1})
\bigr) \mu_{S_2} (x, \beta_{S_2}) \nonumber \\[0.2cm] 
& \qquad= 
- \bigl(
\partial_{x_{S_2}}^\top 
\mu_{S_1} (x_S, \truebeta_{S_1})
\bigr) d_{S_2} (x, \beta_{S_2}). 
\label{eq:diff_S1}    
\end{align}
Indeed, it holds that $F^k (\sample{X}{i-1}, \theta) = 0$ due to the following result (Lemma 6 in \cite{igu:23}). 
\begin{lemma} \label{lemm:mat_1}
For any $(x, \theta) \in \mathbb{R}^N \times \Theta$, $M (x, \theta) = \mathbf{0}_{N_{S_1} \times N_{S_2}}$. 
\end{lemma}
Thus, Lemmas \ref{lemma:aux_1} and \ref{lemma:aux_2} yield
\begin{align*}
\sup_{(\beta_{S_2}, \beta_R, \sigma) \in \Theta_{\beta_{S_2}} \times \Theta_{\beta_{R}} \times \Theta_\sigma} 
\Bigl| \mathbf{A}_{p,n}^k (\truebeta_{S_1}, \beta_{S_2}, \beta_R, \sigma) \Bigr| \probconv 0, 
\end{align*}
as \limit. The proof of  convergence (\ref{eq:conv_A}) is now complete. 
\subsubsection{Proof of limit (\ref{eq:conv_B})} 
We define 
\begin{align*}
 \hat{\mathbf{B}}_{p, n} (\theta)
 \equiv \tfrac{\Delta_n^3}{n} \, 
 \partial_{\beta_{S_1}}^2 \ell_{p, n} 
  \left( \theta \right), \quad \theta \in \Theta.    
\end{align*} 
Then, it holds that
\begin{align*}
\hat{\mathbf{B}}_{p, n} (\theta) 
= \tfrac{1}{n} \sum_{i = 1}^n 
\sum_{k = 0}^2  \hat{\mathbf{B}}^{(k)}_{i-1}
(\theta),  
\end{align*}
where we have set:  
for $1 \le j_1, j_2 \le N_{\beta_{S_1}}$, 
\begin{align*}
[\hat{\mathbf{B}}^{(0)}_{i-1} (\theta)]_{j_1 j_2} 
& \equiv  
2 \, \bigl( \partial_{\beta_{S_1, j_1}} \mu_{S_1} 
(X_{S, t_{i-1}}, \beta_{S_1}) \bigr)^\top 
\boldsymbol{\Lambda}_{S_1 S_1} (X_{t_{i-1}}, \theta) 
\,  \partial_{\beta_{S_1}, j_2} 
\mu_{S_1} (X_{S, t_{i-1}}, \beta_{S_1});  \\[0.3cm] 
[\hat{\mathbf{B}}^{(1)}_{i-1} (\theta)]_{j_1 j_2}  
& \equiv  
2 \, d_{S_1} (X_{S, t_{i-1}}, \beta_{S_1})^\top   
\partial_{\beta_{S_1}, j_1} 
\partial_{\beta_{S_1}, j_2}
\bigl\{  
\boldsymbol{\Lambda}_{S_1 S_1} (\sample{X}{i-1}, \theta) 
 d_{S_1} (X_{S, t_{i-1}}, \beta_{S_1}) 
\bigr\} ;  \\[0.3cm]
[\hat{\mathbf{B}}^{(2)}_{i-1} (\theta)]_{j_1 j_2}  
& \equiv  
\sum_{1 \le k_1, k_2 \le N} 
R_{k_1 k_2}^{j_1 j_2} (\Delta_n^3, \sample{X}{i-1}, \theta)
\, 
\mathbf{m}_{p, i}^{k_1} (\Delta_n, \trueparam) 
\, 
\mathbf{m}_{p, i}^{k_2} (\Delta_n, \trueparam) \\ 
& \qquad  + \sum_{1 \le k \le N} 
R_{k}^{j_1 j_2} (\sqrt{\Delta_n}, 
\sample{X}{i-1}, \theta) \, 
\mathbf{m}_{p, i}^{k} (\Delta_n, \trueparam)  
+ R^{j_1 j_2} ( \sqrt{\Delta_n}, \sample{X}{i-1}, \theta),
\end{align*} 
for some $R^{j_1 j_2}_{k_1 k_2}, R^{j_1 j_2}_{k}, R^{j_1 j_2} \in \mathcal{S}$. 
Due to the consistency of $\hat{\beta}_{S_1, p, n}$, we have 
\begin{align*}
\sup_{\lambda \in [0,1]} \sup_{(\beta_{S_2}, \beta_R, \sigma) \in \Theta_{\beta_{S_2}} \times \Theta_{\beta_{R}} \times \Theta_{\sigma}} 
\Bigl|
\tfrac{1}{n} \sum_{i = 1}^n \hat{\mathbf{B}}^{(1)}_{i-1} ( (\truebeta_{S_1} + \lambda (\hat{\beta}_{S_1, p, n} - \truebeta_{S_1}), \beta_{S_2}, \beta_R, \sigma ))
\Bigr| 
\probconv 0,
\end{align*} 
and also obtain from Lemma \ref{lemma:aux_2} that
\begin{align*}
\sup_{\lambda \in [0,1]} \sup_{(\beta_{S_2}, \beta_R, \sigma) \in \Theta_{\beta_{S_2}} \times \Theta_{\beta_{R}} \times \Theta_{\sigma}} 
\Bigl|
\tfrac{1}{n} \sum_{i = 1}^n \hat{\mathbf{B}}^{(2)}_{i-1} ( (\truebeta_{S_1} + \lambda (\hat{\beta}_{S_1, p, n} - \truebeta_{S_1}), \beta_{S_2}, \beta_R, \sigma ))
\Bigr| 
\probconv 0. 
\end{align*} 
Thus, Lemmas \ref{lemma:aux_1} and \ref{lemma:aux_2} yield  
\begin{align*}
& \sup_{(\beta_{S_2}, \beta_R, \sigma) \in \Theta_{\beta_{S_2}} \times \Theta_{\beta_R} \times \Theta_\sigma} 
\Bigl| 
\mathbf{B}_{p,n}
\bigl( 
\hat{\beta}_{S_1, p, n}, \beta_{S_2}, \beta_R, \sigma
\bigr) 
- 2  \mathbf{{B}} 
\bigl( 
\truebeta_{S_1}, \beta_{S_2}, \sigma
\bigr) \Bigr| 
\probconv 0,
\end{align*}
where we used
\begin{align*}
\boldsymbol{\Lambda}_{S_1 S_1} (x, \theta) = 720 \, a_{S_1}^{-1} (x, \theta), \qquad (x, \theta) \in \mathbb{R}^N \times \Theta. 
\end{align*}
The proof of limit (\ref{eq:conv_B}) is now complete. 
\subsection{Proof of Lemma \ref{lemm:step_2}} 
\label{appendix:pf_step_2}
It holds that 
\begin{align*}
\tfrac{\Delta_n}{n} \, \ell_{p, n} 
 ( \theta ) 
=  \tfrac{1}{n} \sum_{i = 1}^n \sum_{k = 1}^5 S_{i-1}^{(k)} (\theta), \qquad \theta \in \Theta,  
\end{align*} 
where we have set:  \vspace{0.1cm}
\begin{align*} 
S_{i-1}^{(1)} (\theta) 
& = \sum_{j_1, j_2 =1}^{N_{S_1}} 
R_{j_1 j_2}^{(1)} (\Delta_n, \sample{X}{i-1}, \theta) 
\cdot 
\tfrac{d_{S_1}^{j_1} (X_{S, t_{i-1}}, \beta_{S_1})}{\sqrt{\Delta_n^3}} 
\cdot 
\tfrac{d_{S_1}^{j_2} (X_{S, t_{i-1}}, \beta_{S_1})}{\sqrt{\Delta_n^3}} ; \\[0.2cm]
S_{i-1}^{(2)} (\theta) 
& = \sum_{j=1}^{N_{S_1}} 
R_{j}^{(1)} (\sqrt{\Delta_n}, \sample{X}{i-1}, \theta) 
\cdot 
\tfrac{d_{S_1}^{j} (X_{S, t_{i-1}}, \beta_{S_1})}{\sqrt{\Delta_n^3}}; \\[0.2cm] 
S_{i-1}^{(3)} (\theta) 
& = \sum_{\substack{1 \le  j_1 \le N_{S_1} \\ 1 \le j_2 \le N}}
R_{j_1, j_2}^{(2)} (1, \sample{X}{i-1}, \theta) 
\cdot 
\tfrac{d_{S_1}^{j_1} (X_{S, t_{i-1}}, \beta_{S_1})}{\sqrt{\Delta_n^3}}
\cdot 
\tfrac{\mathbf{m}_{p, i}^{j_2} (\Delta_n, \trueparam)}{\sqrt{\Delta_n}}; \\[0.2cm]  
S_{i-1}^{(4)} (\theta)  
& = 
\tfrac{1}{4} \bigl( \mathcal{L} \mu_{S_1} (\sample{X}{i-1}, \trueparam) 
- \mathcal{L} \mu_{S_1} (\sample{X}{i-1}, \theta) \bigr)^\top 
\boldsymbol{\Lambda}_{S_1 S_1} (\sample{X}{i-1}, \theta)
\bigl( \mathcal{L} \mu_{S_1} (\sample{X}{i-1}, \trueparam)
- \mathcal{L} \mu_{S_1} (\sample{X}{i-1}, \theta) \bigr) \\ 
& \quad - \tfrac{1}{2} \bigl( \mathcal{L} \mu_{S_1} (\sample{X}{i-1}, \trueparam)
- \mathcal{L} \mu_{S_1} (\sample{X}{i-1}, \theta) \bigr)^\top
\boldsymbol{\Lambda}_{S_1S_2} (\sample{X}{i-1}, \theta)  
d_{S_2} (\sample{X}{i-1}, \beta_{S_2}) \\ 
& \quad - \tfrac{1}{2} d_{S_2} (\sample{X}{i-1}, \beta_{S_2})^\top 
\boldsymbol{\Lambda}_{S_2S_1} (\sample{X}{i-1}, \theta)  
\bigl( \mathcal{L} \mu_{S_1} (\sample{X}{i-1}, \trueparam)
- \mathcal{L} \mu_{S_1} (\sample{X}{i-1}, \theta) \bigr)  \\   
& \quad + d_{S_2} (\sample{X}{i-1}, \beta_{S_2})^\top
\boldsymbol{\Lambda}_{S_2 S_2} (\sample{X}{i-1}, \theta)
d_{S_2} (\sample{X}{i-1}, \beta_{S_2}) \\
& 
= 
d_{S_2} (\sample{X}{i-1}, \beta_{S_2})^\top \, 
\widetilde{S}_{i-1}^{(4)} (x, \theta)
\, 
d_{S_2} (\sample{X}{i-1}, \beta_{S_2}),  \\[0.2cm]
S_{i-1}^{(5)} (\theta)  
& = \sum_{j_1, j_2 = 1}^N 
{R}_{j_1 j_2}^{(3)}
(\Delta_n,  \sample{X}{i-1}, \theta) 
\, \mathbf{m}_{p, i}^{j_1} (\Delta_n, \trueparam) 
\, \mathbf{m}_{p, i}^{j_2} (\Delta_n, \trueparam ) 
+ \sum_{j = 1}^N 
{R}_{j}^{(2)} (\sqrt{\Delta_n}, \sample{X}{i-1}, \theta) \,
\mathbf{m}_{p, i}^{j} (\Delta_n, \trueparam) \\
& \quad + {R} ({\Delta_n},  \sample{X}{i-1}, \theta), \\[-0.3cm] 
\end{align*} 
for some functions $R_j^{(\cdot)}, R_{j_1 j_2}^{(\cdot)}, \, {R} \in \mathcal{S}$ and 
\begin{align} 
\widetilde{S}^{(4)} (x, \theta) 
& = \tfrac{1}{4} \bigl( \partial_{x_{S_2}}^\top
\mu_{S_1} (x_S, \beta_{S_1}) \bigr)^\top 
\, 
\boldsymbol{\Lambda}_{S_1 S_1} (x, \theta) 
\, 
\partial_{x_{S_2}}^\top \mu_{S_1} (x_S, \beta_{S_1})
+ \tfrac{1}{2} 
\bigl( \partial_{x_{S_2}}^\top
\mu_{S_1} (x_S, \beta_{S_1})
\bigr)^\top 
\, 
\boldsymbol{\Lambda}_{S_1 S_2} (x, \theta) \nonumber \\ 
& \quad \quad \quad 
+ \tfrac{1}{2}
\boldsymbol{\Lambda}_{S_2 S_1} (x, \theta)  
\partial_{x_{S_2}}^\top
\mu_{S_1} (x_S, \beta_{S_1}) 
+ \boldsymbol{\Lambda}_{S_2 S_2} (x, \theta) \nonumber  \\[0.2cm] 
& = \tfrac{1}{2}
\boldsymbol{\Lambda}_{S_2 S_1} (x, \theta)  
\partial_{x_{S_2}}^\top
\mu_{S_1} (x_S, \beta_{S_1}) 
+ \boldsymbol{\Lambda}_{S_2 S_2} (x, \theta)  \nonumber \\[0.2cm] 
& = 12 \, a_{S_2}^{-1} (x, \theta). \label{eq:apply_mat_2}
\end{align}  
In the above, we made use of the equation (\ref{eq:diff_S1}) in the first equation, Lemma \ref{lemm:mat_1} in the second equation and the following result (Lemma 13 in \cite{igu:23}) in the last equation. 
\begin{lemma} \label{lemm:mat_2}
For any $(x, \theta) \in \mathbb{R}^N \times \Theta$, it holds that
\begin{align*}
\boldsymbol{\Lambda}_{S_2 S_2} (x, \theta)
= 12 \, a_{S_2}^{-1} (x, \theta) - \tfrac{1}{2}
\boldsymbol{\Lambda}_{S_2 S_1} (x, \theta)  
\partial_{x_{S_2}}^\top
\mu_{S_1} (x_S, \beta_{S_1}).    
\end{align*}
\end{lemma} 
Notice that for any $x_S \in \mathbb{R}^{N_S}$ 
\begin{align*} 
&\tfrac{d_{S_1}^{j}
(x_S, \, \hat{\beta}_{S_1, p, n})}{\sqrt{\Delta_n^3}} \\ 
& = \sum_{\ell = 1}^{N_{\beta_{S_1}}} 
\tfrac{\mu_{S_1}^{j} \bigl(x_S,  (\beta_{S_1}^{\dagger, 1}, \ldots, \beta_{S_1}^{\dagger, \ell}, 
\hat{\beta}_{S_1, p, n }^{\ell+1}, \ldots, \hat{\beta}_{S_1, p, n }^{N_{\beta_{S_1}}} ) \bigr)
- \mu_{S_1}^{j} \bigl(x_S,  (\beta_{S_1}^{\dagger, 1}, \ldots, \beta_{S_1}^{\dagger, \ell-1}, 
\hat{\beta}_{S_1, p, n }^{\ell}, \ldots, \hat{\beta}_{S_1, p, n }^{N_{\beta_{S_1}}} ) \bigr)
}{\beta_{S_1}^{\dagger, \ell}-\hat{\beta}_{S_1,p,n}^{\ell}} \cdot
\tfrac{\beta_{S_1}^{\dagger, \ell}-\hat{\beta}_{S_1,p,n}^{\ell}}{\sqrt{\Delta_n^3}},  
\end{align*}
where we interpret
\begin{align*}
( \beta_{S_1}^{\dagger, 1}, \ldots, \beta_{S_1}^{\dagger, \ell}, 
\hat{\beta}_{S_1, p, n }^{\ell+1}, \ldots, \hat{\beta}_{S_1, p, n }^{N_{\beta_{S_1}}} ) 
= 
\begin{cases}
\hat{\beta}_{S_1, p, n},  & \ell = 0; \\ 
\truebeta_{S_1}, & \ell = N_{\beta_{S_1}}
\end{cases}. 
\end{align*}
Thus, due to condition \ref{assump:additional_con} and the convergence rate $\hat{\beta}_{S_1, p, n} - \truebeta_{S_1} = o_{\mathbb{P}_{\trueparam}} (\sqrt{\Delta_n^3})$ , it follows from Lemmas \ref{lemma:aux_1} and \ref{lemma:aux_2} that
\begin{align*} 
\sup_{(\beta_{S_2}, \beta_{R}, \sigma) \in \Theta_{\beta_{S_1}} \times \Theta_{\beta_R} \times \Theta_\sigma} 
\Bigl| \tfrac{1}{n} \sum_{i = 1}^n  S_{i-1}^{(k)} (\hat{\beta}_{S_1,p,n}, \beta_{S_2}, \beta_{R}, \sigma)
\Bigr| 
\probconv 0,  \qquad  k = 1, 2, 3, 5, 
\end{align*}
as \limit. Finally, for the fourth term, we apply Lemma \ref{lemma:aux_1} to obtain 
\begin{align*} 
\tfrac{1}{n} \sum_{i = 1}^n  S_{i-1}^{(4)} (\hat{\beta}_{S_1,p,n}, \beta_{S_2}, \beta_{R}, \sigma)
\probconv 
\int_{\mathbb{R}^N} 
d_{S_2} (x, \beta_{S_2})^\top 
\widetilde{S}^{(4)} 
\bigl( 
x,  (\truebeta_{S_1}, \beta_{S_2}, \beta_{R}, \sigma)
\bigr) 
d_{S_2} (x, \beta_{S_2}) 
 \truedist (dx), 
\end{align*} 
as \limit \, uniformly in $(\beta_{S_2}, \beta_{R}, \sigma) \in \Theta_{\beta_{S_2}} \times \Theta_{\beta_{R}} \times \Theta_\sigma$. 
The proof of Lemma \ref{lemm:step_2} is now complete. 
\subsection{Proof of Lemma \ref{lemma:step_2}} 
\label{appendix:pf_step_4} 
\subsubsection{Proof of limit (\ref{eq:step2_A})}
We have shown in Section \ref{sec:pf_conv_A} that if \limit, then 
\begin{align*}
\sup_{(\beta_R, \sigma) \in \Theta_{\beta_R} \times \Theta_\sigma} 
\Bigl| 
\tfrac{\sqrt{\Delta_n^3}}{n} \partial_{\beta_{S_1}} \ell_{p, n} (\truebeta_S, \beta_R, \sigma)
\Bigr| 
\probconv 
0.   
\end{align*}
Thus, we here focus on the term
\begin{align*} 
\widetilde{\mathbf{A}}_{p,n}^j (\truebeta_{S}, \beta_R, \sigma) 
= \tfrac{\sqrt{\Delta_n}}{n} \partial_{\beta_{S_2}, j} \ell_{p, n} (\truebeta_S, \beta_R, \sigma), 
\qquad N_{\beta_{S_1}} + 1 \le j \le N_{\beta_S}. 
\end{align*}
Noticing that for any $x \in \mathbb{R}^N, \, (\beta_R, \sigma) \in \Theta_{\beta_R} \times \Theta_\sigma$, 
\begin{align*}
\tfrac{r_{S_1, K_p + 2} (x, \trueparam) - 
r_{S_1, K_p + 2} (x, (\truebeta_S, \beta_R, \sigma))}{\sqrt{\Delta_n^5}} = R (\sqrt{\Delta_n}, x, (\truebeta_S, \beta_R, \sigma)); \\ 
\tfrac{r_{S_2, K_p + 1} (x, \trueparam) - 
r_{S_2, K_p + 1} (x, (\truebeta_S, \beta_R, \sigma))}{\sqrt{\Delta_n^3}} = \widetilde{R} (\sqrt{\Delta_n}, x, (\truebeta_S, \beta_R, \sigma)),  
\end{align*}
for some $R, \widetilde{R} \in \mathcal{S}$, we have 
\begin{align*} 
\widetilde{\mathbf{A}}_{p,n}^j (\truebeta_{S}, \beta_R, \sigma) & = \tfrac{1}{n} \sum_{i = 1}^n 
\biggl\{ 
\sum_{k_1, k_2 = 1}^N  
R^{j}_{k_1 k_2} (\sqrt{\Delta_n}, \sample{X}{i-1}, \theta) 
\mathbf{m}_{p, i}^{k_1} (\Delta_n, \trueparam)
\mathbf{m}_{p, i}^{k_2} (\Delta_n, \trueparam) 
\nonumber \\ 
& \qquad 
+ \sum_{k = 1}^N 
R^{j}_{k} (1, \sample{X}{i-1}, \theta) 
\mathbf{m}_{p, i}^{k_1} (\Delta_n, \trueparam) 
+ R^{j} (\sqrt{\Delta_n}, \sample{X}{i-1}, \theta) 
\biggr\}|_{\theta = (\truebeta_{S}, \beta_R, \sigma)}, 
\end{align*} 
where $R^j_{k_1 k_2}, R^j_{k}, R^j \in \mathcal{S}$. 
From Lemmas \ref{lemma:aux_1} and \ref{lemma:aux_2}, we immediately obtain that 
\begin{align*}
\sup_{(\beta_R, \sigma) \in \Theta_{\beta_R} \times \Theta_\sigma } \Bigl| \widetilde{\mathbf{A}}_{p,n}^{j} (\truebeta_{S}, \beta_R, \sigma)   \Bigr| \probconv 0, 
\end{align*}
as \limit, and the proof of the limit (\ref{eq:step2_A}) is now complete. 
\subsubsection{Proof of limit (\ref{eq:step2_B})}
We define 
\begin{align*}
\mathbf{U}_n (\theta) 
:= \widetilde{\mathbf{M}}_{n} 
\, 
\partial^2_{\beta_S} \ell_{p, n} (\theta) 
\, 
\widetilde{\mathbf{M}}_{n} 
= 
\begin{bmatrix} 
\mathbf{U}_{S_1 S_1} ( \theta ) & \mathbf{U}_{S_1 S_2}  ( \theta )  \\[0.2cm]
\mathbf{U}_{S_2 S_1} ( \theta ) & 
\mathbf{U}_{S_2 S_2} ( \theta )  
\end{bmatrix},  \quad \theta \in \Theta, 
\end{align*} 
where we have set: 
\begin{align*}
\mathbf{U}_{S_1 S_1} ( \theta )  
& = \tfrac{\Delta_n^3}{n} \partial^2_{\beta_{S_1}}  \ell_{p, n} (\theta), \quad  
\mathbf{U}_{S_1 S_2} ( \theta )  
= \tfrac{\Delta_n^2}{n} \partial_{\beta_{S_1}} \partial_{\beta_{S_2}}^\top \ell_{p, n} (\theta); \\[0.1cm]
\mathbf{U}_{S_2 S_1}  ( \theta ) 
& = \mathbf{U}_{S_1 S_2} ( \theta )^\top, \quad 
\mathbf{U}_{S_2 S_2} ( \theta )  
=  \tfrac{\Delta_n}{n} \partial^2_{\beta_{S_2}}  
\ell_{p, n} (\theta),  
\end{align*} 
From the proof of Lemma \ref{lemma:step1_B} in Appendix \ref{appendix:pf_step1_B} and the consistency of the estimator $\hat{\beta}_{S,p ,n}$, we have that if \limit, then 
\begin{align*}
 & \sup_{\lambda \in [0,1]} \sup_{(\beta_R, \sigma) \in \Theta_{\beta_R} \times \Theta_\sigma}
\Bigl| 
\mathbf{U}_{S_1 S_1} 
\bigl( (\truebeta_S + \lambda ( \hat{\beta}_{S, p, n} 
- \truebeta_S), \beta_R, \sigma) 
\bigr)
- 2 \widetilde{\mathbf{B}}^{S_1} 
\bigl( (\truebeta_S, \beta_R, \sigma)  \bigr)  
\Bigr| \probconv 0, 
\end{align*}
where $\widetilde{\mathbf{B}}^{S_1} 
(\truebeta_S, \beta_R, \sigma)$ is defined in (\ref{eq:limit_B_1}). We will next show in detail that 
\begin{align}
 & \sup_{\lambda \in [0,1]} \sup_{(\beta_R, \sigma) \in \Theta_{\beta_R} \times \Theta_\sigma}
\Bigl| 
\mathbf{U}_{S_1 S_2} 
\bigl(
(\truebeta_S + \lambda ( \hat{\beta}_{S, p, n} 
- \truebeta_S), \beta_R, \sigma) 
\bigr)
\Bigr| \probconv 0;  \label{eq:u_12} \\
& \sup_{\lambda \in [0,1]} \sup_{(\beta_R, \sigma) \in \Theta_{\beta_R} \times \Theta_\sigma}
\Bigl| 
\mathbf{U}_{S_2 S_2} 
\bigl(
(\truebeta_S + \lambda ( \hat{\beta}_{S, p, n} 
- \truebeta_S), \beta_R, \sigma) 
\bigr)
- 2 \widetilde{\mathbf{B}}^{S_2} 
\bigl( (\truebeta_S, \beta_R, \sigma) \bigr)  
\Bigr| \probconv 0,  \label{eq:u_22}
\end{align}
where $\widetilde{\mathbf{B}}^{S_2} 
(\truebeta_S, \beta_R, \sigma)$ is defined in (\ref{eq:limit_B_2}). It holds that 
\begin{align*}
\mathbf{U}_{S_1 S_2} (\theta) 
= \tfrac{1}{n} \sum_{i = 1}^n 
\sum_{k = 1}^4 Q^{(k)}_{i-1} (\theta),  
\end{align*}
where we have set for $1 \le j_1 \le N_{\beta_{S_1}}$, $1 \le j_2 \le N_{\beta_{S_2}}$, 
\begin{align*}
\bigl[ Q_{i-1}^{(1)} (\theta) \bigr]_{j_1 j_2} 
& = 
\partial_{\beta_{S_1}, j_1} 
\mu_{S_1}(X_{S, t_{i-1}}, \beta_{S_1})^\top 
M (\sample{X}{i-1}, \theta) 
\partial_{\beta_{S_2}, j_2}
\mu_{S_2} ( \sample{X}{i-1}, \beta_{S_2}); \\[0.3cm] 
%
\bigl[ Q_{i-1}^{(2)} (\theta) \bigr]_{j_1 j_2}  
& = \sum_{k_1, k_2 = 1}^{N_{{S_1}}} 
R_{j_1 j_2}^{k_1 k_2} (1, \sample{X}{i-1}, \theta) 
d_{S_1}^{k_1} (X_{S, t_{i-1}}, \beta_{S_1})
\tfrac{d_{S_1}^{k_2} (X_{S, t_{i-1}}, \beta_{S_1})}{{\Delta_n}}; \\[0.3cm]  
\bigl[ Q_{i-1}^{(3)} (\theta) \bigr]_{j_1 j_2} 
& 
=   
\sum_{\substack{1 \le k_1 \le N_{S_1} \\ 1 \le k_2 \le N_{{S_2}}}} 
{R}_{j_1 j_2}^{k_1 k_2} (1, \sample{X}{i-1}, \theta)
\, d_{S_1}^{k_2} (X_{S, t_{i-1}}, \beta_{S_1}) 
\, d_{S_2}^{k_2} (\sample{X}{i-1}, \beta_{S_2}); \\[0.1cm]  
%
%
\bigl[ Q_{i-1}^{(4)} (\theta) \bigr]_{j_1 j_2} 
& = 
\sum_{k_1, k_2 = 1}^{N} 
\widetilde{R}_{j_1j_2}^{k_1 k_2} (\sqrt{\Delta_n}, \sample{X}{i-1}, \theta) 
\mathbf{m}_{p, i}^{k_1} (\Delta_n, \trueparam) 
\mathbf{m}_{p, i}^{k_2} (\Delta_n, \trueparam) \\ 
& \quad\quad\quad + 
\sum_{k = 1}^{N} \widetilde{R}_{j_1j_2}^{k} (\sqrt{\Delta_n}, \sample{X}{i-1}, \theta) \mathbf{m}_{p, i}^k (\Delta_n, \trueparam) 
+ {R}_{j_1j_2} (\sqrt{\Delta_n}, \sample{X}{i-1}, \theta), 
\end{align*}
for some functions $R_{j_1 j_2}^{k_1 k_2}, \, \widetilde{R}_{j_1 j_2}^{k_1 k_2}, \, {R}_{j_1 j_2}^{k_1 k_2}, \, \widetilde{R}_{j_1 j_2}^{k}, \, {R}_{j_1 j_2} \in \mathcal{S}$ and $M:\mathbb{R}^N \times \Theta \to \mathbb{R}^{N_{S_1} \times N_{S_2}}$ being defined as (\ref{eq:M}).  
From Lemmas \ref{lemma:aux_1}, \ref{lemma:aux_2}, \ref{lemm:mat_1} and the consistency of the estimator $\hat{\beta}_{S, p, n}$ with the rate of convergence  $\hat{\beta}_{S_1, p, n} - \truebeta_{S_1} = o_{\mathbb{P}_{\trueparam}} (\sqrt{\Delta_n^3})$, we have 
that if \limit, then
\begin{align*}
\sup_{\lambda \in [0,1]} 
\sup_{(\beta_R, \sigma) \in \Theta_{\beta_R} \times \Theta_{\sigma}} 
\Bigl|
\tfrac{1}{n} \sum_{i = 1}^n 
Q_{i-1}^{(k)} 
\bigl(
(\truebeta_S + \lambda ( \hat{\beta}_{S, p, n} 
- \truebeta_S), \beta_R, \sigma) 
\bigr) 
\Bigr| \probconv 0, \qquad 1 \le k \le 4, 
\end{align*} 
and now the proof of (\ref{eq:u_12}) is complete. 

We next consider the term $\mathbf{U}_{S_2 S_2} (\theta)$. It holds that 
\begin{align*}
\mathbf{U}_{S_2 S_2} (\theta) 
= \tfrac{1}{n} \sum_{i = 1}^n  
\sum_{k = 1}^4  T_{i-1}^{(k)}(\theta),
\end{align*} 
where we have set, for $1 \le j_1, j_2 \le N_{\beta_{S_2}}$, 
\begin{align*}    
\bigl[T_{i-1}^{(1)} (\theta) \bigr]_{j_1 j_2}  
& = 
\sum_{k_1, k_2 = 1}^{N_{{S_1}}}   
{R}^{j_1 j_2}_{k_1 k_2} (1, \sample{X}{i-1}, \theta) 
\cdot 
\tfrac{d_{S_1}^{k_1} (X_{S, t_{i-1}}, \beta_{S_1})}{\sqrt{\Delta_n}}
\cdot 
\tfrac{d_{S_1}^{k_2} (X_{S, t_{i-1}}, \beta_{S_1})}{\sqrt{\Delta_n}} 
\\[-0.1cm] &\qquad \qquad + \sum_{k = 1}^{N_{{S_1}}} 
R^{j_1 j_2}_{k} (1, \sample{X}{i-1}, \theta) 
\cdot \tfrac{d_{S_1}^{k} 
(X_{S, t_{i-1}}, \beta_{S_1})}{\Delta_n};
\\[0.2cm]
\bigl[T_{i-1}^{(2)} (\theta) \bigr]_{j_1 j_2}   
& = 
2 \partial_{\beta_{S_2}, j_1} \mu_{S_2} (\sample{X}{i-1}, \beta_{S_2})^\top 
\boldsymbol{\Lambda}_{S_2 S_2} (\sample{X}{i-1}, \theta)
\partial_{\beta_{S_2}, j_2} \mu_{S_2} (\sample{X}{i-1}, \beta_{S_2}) \nonumber \\[0.1cm]
& \quad + 
\partial_{\beta_{S_2}, j_1} 
\mathcal{L} \mu_{S_1} (\sample{X}{i-1}, \beta_{S})^\top 
\boldsymbol{\Lambda}_{S_1 S_2} (\sample{X}{i-1}, \theta)
\partial_{\beta_{S_2}, j_2} \mu_{S_2} (\sample{X}{i-1}, \beta_{S_2}) \nonumber \\[0.1cm]  
& \quad +  
\partial_{\beta_{S_2}, j_1} 
 \mu_{S_2} (\sample{X}{i-1}, \beta_{S_2})^\top
\boldsymbol{\Lambda}_{S_2  S_1} (\sample{X}{i-1}, \theta)
\partial_{\beta_{S_2}, j_2} \mathcal{L} \mu_{S_1} (\sample{X}{i-1}, \beta_{S})  \nonumber \\[0.1cm]  
& \quad + \tfrac{1}{2}
\partial_{\beta_{S_2}, j_1} 
\mathcal{L} \mu_{S_1} (\sample{X}{i-1}, \beta_{S})^\top
\boldsymbol{\Lambda}_{S_1 S_1} (\sample{X}{i-1}, \theta)
\partial_{\beta_{S_2}, j_2} \mathcal{L} \mu_{S_1} (\sample{X}{i-1}, \beta_{S}); \\[0.2cm] 
\bigl[T_{i-1}^{(3)} (\theta) \bigr]_{j_1 j_2}  
& = 
\sum_{k = 1}^{N_{{S_2}}} 
\widetilde{R}^{j_1 j_2}_k (1, \sample{X}{i-1}, \theta) 
\, d_{S_2}^{k} (\sample{X}{i-1}, \beta_{S_2})
\\ & \qquad \qquad + \sum_{\substack{1 \le  k_1 \le N_{{S_1}} \\[-0.1cm] 1 \le  k_2  \le N}} 
\bar{R}^{j_1 j_2}_{k_1 k_2} (1, \sample{X}{i-1}, \theta) 
\cdot 
\tfrac{d_{S_1}^{k_1} 
(X_{S, t_{i-1}}, \beta_{S_1})}{\sqrt{\Delta_n}} \cdot 
\mathbf{m}_{p, i}^{k_2} (\Delta_n, \trueparam) 
\\[0.2cm]
\bigl[T_{i-1}^{(4)} (\theta) \bigr]_{j_1 j_2}    
& = 
\sum_{k_1, k_2  = 1}^N  
\widetilde{R}^{j_1 j_2}_{k_1 k_2}  (\Delta_n, \sample{X}{i-1}, \theta) 
\mathbf{m}_{p, i}^{k_1} (\Delta_n, \trueparam)
\mathbf{m}_{p, i}^{k_2} (\Delta_n, \trueparam)   
\\ 
& \qquad \qquad 
+ \sum_{k =1}^N \bar{R}^{j_1 j_2}_{k} (\sqrt{\Delta_n}, \sample{X}{i-1}, \theta) \mathbf{m}_{p, i} (\Delta_n, \trueparam)
+ R (\sqrt{\Delta_n}, \sample{X}{i-1}, \theta),  
\end{align*}
where $R^{j_1 j_2}_{k_1 k_2},  R^{j_1 j_2}_{k}, \widetilde{R}^{j_1 j_2}_{k_1 k_2},  \widetilde{R}^{j_1 j_2}_{k},  \bar{R}^{j_1 j_2}_{k_1 k_2},  
\bar{R}^{j_1 j_2}_{k},  R \in \mathcal{S}$. 
From Lemmas \ref{lemma:aux_1}, \ref{lemma:aux_2}, \ref{lemm:mat_1} and the consistency of $\hat{\beta}_{S, p, n}$ with the convergence rate $\hat{\beta}_{S_1, p, n} - \truebeta_{S_1} = o_{\mathbb{P}_{\trueparam}} (\sqrt{\Delta_n^3})$, we obtain that if \limit, then 
\begin{gather*}
\sup_{\lambda \in [0,1]} 
\sup_{(\beta_R, \sigma) \in \Theta_{\beta_R} \times \Theta_{\sigma}} 
\Bigl|
\tfrac{1}{n} \sum_{i = 1}^n 
T_{i-1}^{(k)} 
\bigl(
(\truebeta_S + \lambda ( \hat{\beta}_{S, p, n} 
- \truebeta_S), \beta_R, \sigma) 
\bigr) 
\Bigr| \probconv 0, \qquad k = 1, 3, 4.  \\ 
\end{gather*} 
Similarly to the computation of $\widetilde{S}^{(4)}_{i-1}$ in (\ref{eq:apply_mat_2}),  we obtain from Lemmas \ref{lemm:mat_1} and \ref{lemm:mat_2} that
\begin{align*} 
[T_{i-1}^{(2)} (\theta)]_{j_1 j_2} 
& = 
2 \partial_{\beta_{S_2}, j_1} \mu_{S_2} (\sample{X}{i-1}, \beta_{S_2})^\top 
\boldsymbol{\Lambda}_{S_2 S_2} (\sample{X}{i-1}, \theta)
\partial_{\beta_{S_2}, j_2} \mu_{S_2} (\sample{X}{i-1}, \beta_{S_2}) \nonumber \\
& \quad 
+ 
\partial_{\beta_{S_2}, j_1}  
\mu_{S_2} (\sample{X}{i-1}, \beta_{S})^\top 
\boldsymbol{\Lambda}_{S_2 S_1} (\sample{X}{i-1}, \theta)
\partial_{\beta_{S_2}, j_2} 
\mathcal{L} \mu_{S_1} 
(\sample{X}{i-1}, \beta_{S_2})
\nonumber \\[0.2cm]
& = 24 \, \partial_{\beta_{S_2}, j_1} \mu_{S_2} (\sample{X}{i-1}, \beta_{S_2})^\top 
\, a_{S_2}^{-1} (\sample{X}{i-1}, \theta) 
\, 
\partial_{\beta_{S_2}, j_2} \mu_{S_2} (\sample{X}{i-1}, \beta_{S_2}).  
\end{align*}
Thus, Lemma \ref{lemma:aux_1} yields 
\begin{align} 
\sup_{\lambda \in [0,1]} 
\sup_{(\beta_R, \sigma) \in \Theta_{\beta_R} \times \Theta_{\sigma}} 
\Bigl| 
\tfrac{1}{n} \sum_{i = 1}^n 
T_{i-1}^{(2)} 
\bigl(
\truebeta_S + \lambda ( \hat{\beta}_{S, p, n} 
- \truebeta_S), \beta_R, \sigma 
\bigr) 
- 2 \widetilde{\mathbf{B}}^{S_2} 
 (\truebeta_S, \beta_R, \sigma) 
\Bigr| 
\probconv 0,  \nonumber 
\end{align}
and the proof of convergence (\ref{eq:u_22}) is now complete.  
\subsection{Proof of Lemma \ref{lemma:step3_1}} 
\label{appendix:pf_step3_1}
We make use of the following notation: 
\begin{align*}
\nu_{S, i-1} (\Delta, \beta_S)
& =
\Bigl[ 
\tfrac{d_{S_1}(X_{S, t_{i-1}}, \beta_{S_1})^\top }{\sqrt{\Delta^3}}, 
\, 
\tfrac{d_{S_2}(\sample{X}{i-1}, \beta_{S_2})^\top}{\sqrt{\Delta}} 
\Bigr]^\top;  \\[0.2cm]
\widetilde{d}_{S_1, i-1} (\Delta, \beta_S)
& = \tfrac{
\mathcal{L} \mu_{S_1} (\sample{X}{i-1}, \truebeta_S) 
- \mathcal{L} \mu_{S_1} (\sample{X}{i-1}, \beta_S)}{\sqrt{\Delta}}, \qquad  
\Delta> 0, \;\,\, \beta_S = (\beta_{S_1}, \beta_{S_2}), 
\end{align*}
for $1\le i \le n$. It holds that 
\begin{align*}
\tfrac{1}{n} \ell_{p, n} (\theta) 
= \tfrac{1}{n} \sum_{i = 1}^n 
\sum_{k = 1}^4 V_{i-1}^{(k)} (\theta),
\quad \theta \in \Theta, 
\end{align*}
with 
\begin{align*}
V_{i-1}^{(1)} (\theta)
& \equiv 
\sum_{j_1, j_2 =1}^{N_{S}} 
R_{j_1 j_2}^{(1)} (1, \sample{X}{i-1}, \theta) 
\, \nu_{S, i-1}^{j_1} (\Delta_n, \beta_S)
\, \nu_{S, i-1}^{j_2} (\Delta_n, \beta_S) \\ 
& \quad 
+ \sum_{j_1, j_2 =1}^{N_{S_1}}    
{R}_{j_1 j_2}^{(2)} (1, \sample{X}{i-1}, \theta)  
\, \widetilde{d}_{S_1, i-1}^{j_1} (\Delta_n, \beta_S) 
\, 
\widetilde{d}_{S_1, i-1}^{j_2} (\Delta_n, \beta_S) \\
& \quad 
+ \sum_{\substack{1 \le j_1 \le N_{S_1} \\ 1 \le j_2 \le N_S}} {R}_{j_1 j_2}^{(3)} (1, \sample{X}{i-1}, \theta) \, 
\widetilde{d}_{S_1, i-1}^{j_1} (\Delta_n, \beta_S) 
\, 
\nu_{S, i-1}^{j_2} (\Delta_n, \beta_S);
\\[0.2cm] 
V_{i-1}^{(2)} (\theta)  
& \equiv    
\sum_{\substack{1 \le j_1 \le N_{S} \\ 1 \le j_2 \le N}}
{R}_{j_1 j_2}^{(4)}  (1, \sample{X}{i-1}, \theta)
\, \nu_{S, i-1}^{j_1} (\Delta_n, \beta_{S}) 
\mathbf{m}_{p, i}^{j_2} (\Delta_n, \trueparam) \\
& \quad +
\sum_{\substack{1 \le j_1 \le N_{S_1} \\ 1 \le j_2 \le N}}
{R}_{j_1 j_2}^{(5)}  (1, \sample{X}{i-1}, \theta)
\, \widetilde{d}_{S, i-1}^{j_1} (\Delta_n, \beta_{S}) 
\mathbf{m}_{p, i}^{j_2} (\Delta_n, \trueparam) \\
& \quad + \sum_{1 \le j \le N_{S}}
{R}_{j}^{(1)} (\sqrt{\Delta_n}, \sample{X}{i-1}, \theta) 
\nu_{S, i-1}^{j} (\Delta_n, \beta_{S}) 
+ \sum_{1 \le j \le N_{S_1}} {R}_{j}^{(2)} (\sqrt{\Delta_n}, \sample{X}{i-1}, \theta) 
\widetilde{d}_{S, i-1}^{j} (\Delta_n, \beta_{S}); \\[0.2cm]
V_{i-1}^{(3)} (\theta)  
& \equiv  
\mathbf{m}_{p, i} (\Delta_n, \trueparam)^\top
\boldsymbol{\Lambda}_{i-1} (\theta)
\mathbf{m}_{p, i} (\Delta_n, \trueparam) 
+ \log \det \boldsymbol{\Sigma}_{i-1} (\theta);
\\[0.3cm]
V_{i-1}^{(4)} (\theta) 
&  \sum_{1 \le j_1, j_2 \le N} 
{R}_{j_1 j_2}^{(6)} ({\Delta_n}, \sample{X}{i-1}, \theta) 
\mathbf{m}_{p, i}^{j_1} (\Delta_n , \trueparam)
\mathbf{m}_{p, i}^{j_2} (\Delta_n , \trueparam) \\
& + \sum_{1 \le j \le N} 
{R}_j^{(3)} (\sqrt{\Delta_n}, \sample{X}{i-1}, \theta) 
\mathbf{m}_{p, i}^{j} (\Delta_n, \trueparam)  
+ R (\Delta_n, \sample{X}{i-1}, \theta), 
\end{align*} 
where $R_{j_1 j_2}^{(\cdot)}, 
R_j^{(\cdot)}, R \in \mathcal{S}$. 
%
Due to condition \ref{assump:additional_con} and the rates of convergence 
$\hat{\beta}_{S_1, p, n} - \truebeta_{S_1} = o_{\mathbb{P}_{\trueparam}} (\sqrt{\Delta_n^3})$ 
and
$\hat{\beta}_{S_2, p, n} - \truebeta_{S_2} = o_{\mathbb{P}_{\trueparam}} (\sqrt{\Delta_n})$,
we have: 
\begin{align*}
\sup_{1 \le i \le n} 
\bigl| 
\nu_{S, i-1} (\Delta_n, \hat{\beta}_{S, p, n}) 
\bigr| \probconv 0, 
\qquad 
\sup_{1 \le i \le n} 
\bigl| 
\widetilde{d}_{S, i-1} (\Delta_n, \hat{\beta}_{S, p, n}) 
\bigr| \probconv 0, 
\end{align*}
as \limit. Thus, Lemmas \ref{lemma:aux_1} and \ref{lemma:aux_2} yield: 
\begin{align*} 
& \sup_{(\beta_R, \sigma) \in \Theta_{\beta_R} \times \Theta_\sigma} 
\Bigl| \tfrac{1}{n} \sum_{i = 1}^n V_{i-1}^{(k)} 
\bigl( (\hat{\beta}_{S, p , n}, \beta_R, \sigma) \bigr)
\Bigr| \probconv 0, \ \ k = 1,2,4; \\[0.1cm]
& \sup_{(\beta_R, \sigma) \in \Theta_{\beta_R} \times \Theta_\sigma} 
\Bigl| \tfrac{1}{n} \sum_{i = 1}^n V_{i-1}^{(3)} 
\bigl( (\hat{\beta}_{S, p , n}, \beta_R, \sigma) \bigr)
-  \mathbf{Y}_3 (\sigma)  \Bigr| \probconv 0, 
\end{align*} 
as \limit. The proof is now complete. 
\subsection{Proof of Lemma \ref{lemma:step3_2}} 
\label{appendix:pf_step3_2}
%
%
We first emphasise that $\boldsymbol{\Sigma}_{i-1}$ and $\boldsymbol{\Lambda}_{i-1}$ indeed depend on parameters $\beta_S = (\beta_{S_1}, \beta_{S_2}) \in \Theta_{\beta_S}$ and $\sigma \in \Theta_\sigma$ but not on $\beta_R \in \Theta_{\beta_R}$. Then, the term $\mathbf{K}(\theta)$,  $\theta \in \Theta$, defined in (\ref{eq:K}) is expressed as: 
\begin{align*}
\mathbf{K} (\theta) 
= \tfrac{1}{n} \sum_{i = 1}^n \sum_{k = 1}^4 Z_{i-1}^{(k)} (\theta), 
\end{align*}
where we have set: 
\begin{align*}
& Z_{i-1}^{(1)} (\theta)   
\equiv \tfrac{1}{\Delta_n}   
\bigl( 
\mathbf{m}_{p, i} (\Delta_n, \theta) 
- \mathbf{m}_{p, i} (\Delta_n, (\beta_S, \truebeta_R, \sigma)) \bigr)^\top 
\boldsymbol{\Lambda}_{i-1} ((\beta_S, \sigma))
\bigl( 
\mathbf{m}_{p, i} (\Delta_n, \theta) 
- \mathbf{m}_{p, i} (\Delta_n, (\beta_S, \truebeta_R, \sigma))
\bigr); \\[0.2cm]
& Z_{i-1}^{(2)} (\theta) 
\equiv  \tfrac{2}{\Delta_n}  
\bigl( \mathbf{m}_{p, i} (\Delta_n, (\beta_S, \truebeta_R, \sigma))) 
- \mathbf{m}_{p, i} (\Delta_n, \trueparam) \bigr)^\top  
\boldsymbol{\Lambda}_{i-1} ((\beta_S, \sigma)) 
\bigl( 
\mathbf{m}_{p, i} (\Delta_n, \theta) 
- \mathbf{m}_{p, i} (\Delta_n, (\beta_S, \truebeta_R, \sigma)) \bigr); \\[0.2cm]
& Z_{i-1}^{(3)} (\theta)  
\equiv  \tfrac{2}{\Delta_n} 
\mathbf{m}_{p, i} (\Delta_n, \trueparam)^\top
\boldsymbol{\Lambda}_{i-1} ((\beta_S, \sigma))
\bigl( \mathbf{m}_{p, i} (\Delta_n, \theta) - \mathbf{m}_{p, i} (\Delta_n, (\beta_S, \truebeta_R, \sigma) ) \bigr); 
\\[0.2cm]
& Z_{i-1}^{(4)} (\theta)    
\equiv \sum_{j = 1}^{K_p} 
\Delta_n^{j-1} \cdot 
\Bigl\{ 
\mathbf{m}_{p, i} (\Delta_n, \theta)^\top 
\, 
\mathbf{G}_{i-1, j} (\theta) 
\, 
\mathbf{m}_{p, i} (\Delta_n, \theta) 
- 
\mathbf{m}_{p, i} (\Delta_n, \widetilde{\theta})^\top 
\, 
\mathbf{G}_{i-1, j} (\widetilde{\theta}) 
\, 
\mathbf{m}_{p, i} (\Delta_n, \widetilde{\theta})  \\ 
& \qquad \qquad 
+ \mathbf{H}_{i-1, j} (\theta)
- \mathbf{H}_{i-1, j} (\widetilde{\theta}) 
\Bigr\}|_{\widetilde{\theta} = (\beta_S, \truebeta_R, \sigma)}. 
\end{align*} 
We will study the terms $Z_{i-1}^{(k)} (\theta), \, 1 \le k \le 4$. Noticing that
\begin{align*} 
& \mathbf{m}_{p, i} (\Delta_n, \theta)
- \mathbf{m}_{p, i} (\Delta_n,  (\beta_S, \truebeta_R, \sigma)) \nonumber \\[0.3cm]
& =  
\sqrt{\Delta_n} 
\begin{bmatrix}
\tfrac{1}{6} \partial_{x_{S_2}}^\top  
\mu_{S_1} (X_{S, t_{i-1}}, \beta_{S_1})  
\partial_{x_R}^\top  \mu_{S_2} (\sample{X}{i-1}, \beta_{S_2}) 
\\[0.3cm]
\tfrac{1}{2} \partial_{x_R}^\top 
\mu_{S_2} (\sample{X}{i-1}, \beta_{S_2}) 
\\[0.3cm] 
I_{N_R \times N_R} 
\end{bmatrix} d_R (\sample{X}{i-1}, \beta_R) 
+ R (\sqrt{\Delta_n^3}, \sample{X}{i-1}, \theta)
\end{align*}
with the residual $R = [R^j]_{1 \le j \le N}$ such that $R^j \in \mathcal{S}$, we have 
\begin{align*}
Z_{i-1}^{(1)} (\theta)  
= d_R (\sample{X}{i-1}, \beta_R)^\top 
a_R^{-1} (\sample{X}{i-1}, \sigma)
d_R (\sample{X}{i-1}, \beta_R)
+ R (\Delta_n, \sample{X}{i-1}, \theta) 
\end{align*}
with $R \in \mathcal{S}$, where we exploited the following result whose proof is postponed to the end of this subsection. 
\begin{lemma} \label{lemma:matrix_3}
For any $\theta = (\beta_S, \beta_R, \sigma) \in \Theta$ and $1 \le i \le n$, it follows under condition \ref{assump:hor} that
\begin{align*}
\boldsymbol{\Lambda}_{i-1} (\beta_S, \sigma)
\begin{bmatrix}
\tfrac{1}{6} \partial_{x_{S_2}}^\top  
\mu_{S_1} (X_{S, t_{i-1}}, \beta_{S_1})  
\partial_{x_R}^\top  \mu_{S_2} (\sample{X}{i-1}, \beta_{S_2}) 
\\[0.3cm]
\tfrac{1}{2} \partial_{x_R}^\top 
\mu_{S_2} (\sample{X}{i-1}, \beta_{S_2}) 
\\[0.3cm]
I_{N_R \times N_R} 
\end{bmatrix}
=  
\begin{bmatrix}
\mathbf{0}_{N_{S_1}\times N_R} \\[0.1cm]
\mathbf{0}_{N_{S_2}\times N_R} \\[0.2cm]
a_R^{-1} (\sample{X}{i-1}, \sigma) 
\end{bmatrix}. 
\end{align*}
\end{lemma}
Thus, we apply Lemma \ref{lemma:aux_1} to obtain: 
\begin{align*}
\sup_{\beta_R \in \Theta_{\beta_R}} \Bigl| 
\tfrac{1}{n} \sum_{i = 1}^n Z_{i-1}^{(1)} (\hat{\beta}_{S, p, n}, \beta_R, \hat{\sigma}_{p, n}) 
- \mathbf{Y}_4 (\beta_R) \Bigr| 
\probconv 0, 
\end{align*}
as \limit. For the terms $Z_{i-1}^{(k)}$, $k = 2, 3$, we again make use of Lemmas \ref{lemma:aux_1}, \ref{lemma:aux_2}, \ref{lemma:matrix_3} with the rates of convergence $\hat{\beta}_{S_1, p, n} - \truebeta_{S_1} = o_{\mathbb{P}_{\trueparam}} (\sqrt{\Delta_n^3})$ and $\hat{\beta}_{S_2, p, n} - \truebeta_{S_2} = o_{\mathbb{P}_{\trueparam}} (\sqrt{\Delta_n})$ to obtain that if \limit, then
\begin{align*}
\sup_{\beta_R \in \Theta_{\beta_R}} \Bigl| 
\tfrac{1}{n} \sum_{i = 1}^n Z_{i-1}^{(k)} (\hat{\beta}_{S, p, n}, \beta_R, \hat{\sigma}_{p, n}) \Bigr| 
\probconv 0, \quad k = 2,3. 
\end{align*} 
Finally, we consider the term $Z_{i-1}^{(4)}$. We have 
\begin{align*}
Z_{i-1}^{(4)} (\theta) = 
\bar{Z}_{i-1}^{(4)} (\theta) - \bar{Z}_{i-1}^{(4)} ((\beta_S, \truebeta_R, \sigma)), \quad \theta = (\beta_S, \beta_R, \sigma) \in \Theta, 
\end{align*}
where we have set: 
\begin{align*}
\bar{Z}_{i-1}^{(4)} (\theta) 
= \sum_{j = 1}^{K_p} 
\Delta_n^{j-1} \cdot 
\Bigl\{ \mathbf{m}_{p, i} (\Delta_n, \theta)^\top 
\, 
\mathbf{G}_{i-1, j} (\theta) 
\, 
\mathbf{m}_{p, i} (\Delta_n, \theta)  
+ \mathbf{H}_{i-1, j} (\theta)
\Bigr\}. 
\end{align*}
From the same argument as used in the proof of Lemma \ref{lemma:step3_1} 
in Section \ref{appendix:pf_step3_1} and Lemma \ref{lemma:aux_1}, we get:
\begin{align*}
\sup_{(\beta_R, \sigma) \in \Theta_{\beta_R} \times \Theta_{\sigma}} 
\Bigl| 
\tfrac{1}{n} \sum_{i = 1}^n 
\bar{Z}^{(4)}_{i-1} (\hat{\beta}_{S, p, n}, \beta_R, \sigma)  - \mathbf{Z} (\beta_R, \sigma)
\Bigr| \probconv 0, 
\end{align*}
as \limit, where we have set for $(\beta_R, \sigma) \in \Theta_{\beta_R} \times \Theta_\sigma$, 
\begin{align*} 
\mathbf{Z} (\beta_R, \sigma) 
= \int
\Bigl\{ 
- \mathrm{Tr} \bigl[ 
\mathbf{G}_1 (x, (\truebeta_S, \beta_R, \sigma))  
\boldsymbol{\Sigma} (x, (\truebeta_S, \truesigma)) \bigr]
+ 
\mathbf{H}_1 (x, (\truebeta_S, \beta_R, \sigma)) 
\Bigr\}
\truedist (dx),
\end{align*}
with
\begin{align*}
\mathbf{G}_1 (x, \theta) = 
\boldsymbol{\Lambda} (x, (\beta_S, \sigma)) 
\boldsymbol{\Sigma}_1 (x, \theta) 
\boldsymbol{\Lambda} (x, (\beta_S, \sigma)), 
\quad 
\mathbf{H}_1 (x, \theta) = \mathrm{Tr} 
\bigl[ \boldsymbol{\Lambda} (x, (\beta_S, \sigma)) 
\boldsymbol{\Sigma}_1 (x, \theta)  \bigr],  
\end{align*} 
for $\quad (x, \theta) \in \mathbb{R}^N \times \Theta$.
In the above definition of $\mathbf{G}_1 (x, \theta)$ and  $\mathbf{H}_1 (x, \theta)$, $\boldsymbol{\Sigma}_1 (x, \theta) \in \mathbf{R}^{N \times N}$ is determined by the It\^o-Taylor expansion (\ref{eq:expansion_cov_2}) with $\sample{X}{i-1}$ replaced by $x \in \mathbb{R}^N$. In particular, we notice from the consistency of $\hat{\sigma}_{p,n}$ that  
\begin{align*}
\sup_{\beta_R \in \Theta_{\beta_R}} 
\Bigl| 
\mathbf{Z} (\beta_R, \hat{\sigma}_{p, n}) 
\Bigr| \probconv 0,
\end{align*} 
thus, 
\begin{align*}
\sup_{\beta_R \in \Theta_{\beta_R}} 
\Bigl| 
\tfrac{1}{n} \sum_{i = 1}^n 
Z_{i-1}^{(4)} (\hat{\beta}_{S, p, n}, \beta_R, \hat{\sigma}_{p, n})  
\Bigr| \probconv 0 
\end{align*}
as \limit. The proof of Lemma \ref{lemma:step3_2} is now complete. 
\\

\noindent 
\textit{Proof of Lemma \ref{lemma:matrix_3}}. 
Making use of the invertibility of the matrix $a_R (x, \theta)$ for any $(x, \theta) \in \mathbb{R}^N \times \Theta$ under condition \ref{assump:hor}, we get
\begin{align}  
\begin{bmatrix}
\tfrac{1}{6} \partial_{x_{S_2}}^\top  
\mu_{S_1} (X_{S, t_{i-1}}, \beta_{S_1})  
\partial_{x_R}^\top  \mu_{S_2} (\sample{X}{i-1}, \beta_{S_2}) 
\\[0.3cm]
\tfrac{1}{2} \partial_{x_R}^\top 
\mu_{S_2} (\sample{X}{i-1}, \beta_{S_2}) 
\\[0.3cm]
I_{N_R \times N_R} 
\end{bmatrix}  
& = 
\begin{bmatrix}
\boldsymbol{\Sigma}_{S_1 R} (\sample{X}{i-1}, (\beta_S, \sigma))
\\[0.3cm]
\boldsymbol{\Sigma}_{S_2 R} (\sample{X}{i-1}, (\beta_S, \sigma)) \\[0.3cm]
a_{R}  (\sample{X}{i-1}, \sigma)  
\end{bmatrix} a_{R}^{-1} (\sample{X}{i-1}, \sigma). 
\label{eq:m_diff}
\end{align}
%
%
Thus, (\ref{eq:m_diff}) leads to the assertion by noticing that 
\begin{align*}
\boldsymbol{\Lambda}_{i-1} (\beta_S, \sigma) 
\begin{bmatrix}
\boldsymbol{\Sigma}_{S_1 R} (\sample{X}{i-1}, (\beta_S, \sigma))
\\[0.3cm]
\boldsymbol{\Sigma}_{S_2 R} (\sample{X}{i-1}, (\beta_S, \sigma)) \\[0.3cm]
a_{R}  (\sample{X}{i-1}, \sigma)  
\end{bmatrix} 
= 
\begin{bmatrix}
\mathbf{0}_{N_{S_1}\times N_R} \\
\mathbf{0}_{N_{S_2}\times N_R} \\
I_{N_{R} \times N_R} 
\end{bmatrix}. 
\end{align*} 
\section{Proof of Technical Results for Theorem \ref{thm:clt}} \label{append:pf_technical_clt}
\subsection{Proof of Lemma \ref{lemm:slln}}
\label{appendix:pf_slln} 
We have already shown in the proof of Lemmas \ref{lemma:step1_B} and \ref{lemma:step_2} that if \limit, then
\begin{align} \label{eq:conv_J}
\sup_{\lambda \in [0,1]} \Bigl|  
\bigl[\mathbf{J}_{p, n} \bigl( \trueparam  + \lambda ( \hat{\theta}_{p, n} - \trueparam)  \bigr)\bigr]_{{k_1 k_2}} 
- 2 \bigl[ \Gamma (\trueparam) \bigr]_{k_1 k_2} \Bigr| 
\probconv 0, 
\end{align}  
for $1 \le k_1, k_2 \le N_{\beta_S}$. We then check the above convergence for other cases of $k_1, k_2$. For simplicity of notation, we introduce: 
\begin{align*}
\boldsymbol{\Psi}_{i-1} (\theta) 
= 
\begin{bmatrix}
\tfrac{1}{6} \partial_{x_{S_2}}^\top  
\mu_{S_1} (X_{S, t_{i-1}}, \beta_{S_1})  
\partial_{x_R}^\top  \mu_{S_2} (\sample{X}{i-1}, \beta_{S_2}) 
\\[0.3cm]
\tfrac{1}{2} \partial_{x_R}^\top 
\mu_{S_2} (\sample{X}{i-1}, \beta_{S_2}) 
\\[0.3cm]
I_{N_R \times N_R} 
\end{bmatrix}, \quad 1 \le i \le n.  
\end{align*}
Since the matrix $\mathbf{J}_{p, n} \bigl( \theta \bigr), \, \theta \in \Theta$ is symmetric, we study the following 4 cases. 
\\ 

\noindent
\textbf{Case (i).} $1 \le k_1 \le N_{\beta_{S}}, N_{\beta_{S}} + 1 \le k_2 \le N_{\beta}$. Due to the rate of convergence $\hat{\beta}_{S, p, n} - \truebeta_S$ and the consistency of $\hat{\theta}_{p, n}$, we have 
\begin{align*}
& \bigl[\mathbf{J}_{p, n} \bigl( \trueparam  + \lambda ( \hat{\theta}_{p, n} - \trueparam)  \bigr)\bigr]_{{k_1 k_2}} 
\nonumber \\ 
& = \tfrac{2}{n} \sum_{i = 1}^n 
\biggl\{ 
\begin{bmatrix}
\partial_{\theta, k_1} \mu_{S} (\sample{X}{i-1}, \beta_S) \\[0.1cm]
\mathbf{0}_{N_R} 
\end{bmatrix}^\top
\boldsymbol{\Lambda}_{i-1} (\theta) \,  
\boldsymbol{\Psi}_{i-1} (\theta) \, 
\partial_{\theta, k_2} \mu_R (\sample{X}{i-1}, \beta_R)
\biggr\}|_{\theta = \trueparam + \lambda ( \hat{\theta}_{p, n} - \trueparam)}
 + o_{\mathbb{P}_{\trueparam}} (1) \\ 
& = o_{\mathbb{P}_{\trueparam}} (1), 
\end{align*}
where $o_{\mathbb{P}_{\trueparam}} (1)$ is uniform w.r.t.~$\lambda \in [0,1]$. Notice that the first term in the right hand side of the first equation becomes $0$ due to Lemma \ref{lemma:matrix_3}. Thus, from Lemma \ref{lemma:aux_1}, the limit (\ref{eq:conv_J}) holds 
for $1 \le k_1 \le N_{\beta_{S}}, N_{\beta_{S}} + 1 \le k_2 \le N_{\beta}$. 
\\

\noindent
\textbf{Case (ii).} $1 \le k_1 \le N_{\beta}, N_{\beta} + 1 \le k_2 \le N_{\theta}$. Again, making use of the rate of convergence $\hat{\beta}_{S, p, n} - \truebeta_S$ and the consistency of $\hat{\theta}_{p, n}$, we obtain 
\begin{align*}
& \bigl[\mathbf{J}_{p, n} \bigl( \trueparam  + \lambda ( \hat{\theta}_{p, n} - \trueparam)  \bigr)\bigr]_{{k_1 k_2}} 
= o_{\mathbb{P}_{\trueparam}} (1),
\end{align*}
thus the limit (\ref{eq:conv_J}) holds for $1 \le k_1 \le N_{\beta}, N_{\beta} + 1 \le k_2 \le N_{\theta}$. 
\\

\noindent
\textbf{Case (iii).} $N_{\beta_S} + 1 \le k_1, k_2 \le N_{\beta}$. Similarly to the early cases, we have that 
\begin{align*}
& \bigl[\mathbf{J}_{p, n} \bigl( \trueparam  + \lambda ( \hat{\theta}_{p, n} - \trueparam)  \bigr)\bigr]_{{k_1 k_2}} 
\nonumber \\ 
& = \tfrac{2}{n} \sum_{i = 1}^n 
\Bigl\{ 
\partial_{\theta, k_1} \mu_R (\sample{X}{i-1}, \beta_R)^\top
\boldsymbol{\Psi}_{i-1} (\theta)^\top 
\boldsymbol{\Lambda}_{i-1} (\theta) \,  
\boldsymbol{\Psi}_{i-1} (\theta) \, 
\partial_{\theta, k_2} \mu_R (\sample{X}{i-1}, \beta_R)
\Bigr\} 
|_{\theta = \trueparam + \lambda ( \hat{\theta}_{p, n} - \trueparam)} \\ 
& \qquad \qquad + o_{\mathbb{P}_{\trueparam}} (1) \\
& = \tfrac{2}{n} \sum_{i = 1}^n 
\Bigl\{ \partial_{\theta, k_1} \mu_R (\sample{X}{i-1}, \beta_R)^\top 
a_R^{-1} (\sample{X}{i-1}, \theta) \, 
\partial_{\theta, k_2} \mu_R (\sample{X}{i-1}, \beta_R)
\Bigr\} 
|_{\theta = \trueparam + \lambda ( \hat{\theta}_{p, n} - \trueparam)} + o_{\mathbb{P}_{\trueparam}} (1),  
\end{align*} 
where we used Lemma \ref{lemma:matrix_3} in the last equation. Thus, from Lemma \ref{lemma:aux_1}, the convergence (\ref{eq:conv_J}) holds for $N_{\beta_S} + 1 \le k_1, k_2 \le N_{\beta}$. 
\\

\noindent
\textbf{Case (iv).} $N_{\beta} + 1 \le k_1, k_2 \le N_{\theta}$. We have that
\begin{align*}
& 
\bigl[\mathbf{J}_{p, n} \bigl( \trueparam  + \lambda ( \hat{\theta}_{p, n} - \trueparam)  \bigr)\bigr]_{{k_1 k_2}} 
\\
& = \tfrac{1}{n} \sum_{i = 1}^n 
\Bigl\{ 
\mathbf{m}_{p,i}(\Delta_n, \trueparam)^\top \partial_{\theta, k_1} \partial_{\theta, k_2} \boldsymbol{\Lambda}_{i-1} (\theta)
\mathbf{m}_{p,i}(\Delta_n, \trueparam) 
+ \partial_{\theta, k_1} \partial_{\theta, k_2} 
\log \det \boldsymbol{\Sigma}_{i-1} (\theta)
\Bigr\} 
|_{\theta = \trueparam + \lambda ( \hat{\theta}_{p, n} - \trueparam)} \\ 
& \qquad \qquad 
+ o_{\mathbb{P}_{\trueparam}} (1).  
\end{align*}
Thus, applying Lemma \ref{lemma:aux_1}, we have that if  \limit, then
\begin{align*}
& \sup_{\lambda \in [0,1]} \bigl[\mathbf{J}_{p, n} \bigl( \trueparam  + \lambda ( \hat{\theta}_{p, n} - \trueparam)  \bigr)\bigr]_{{k_1 k_2}}  \\ 
& \probconv \int \Bigl\{ 
\mathrm{Tr}\bigl[\partial_{\theta, k_1} \partial_{\theta, k_2} \boldsymbol{\Lambda} (x, \trueparam) 
\boldsymbol{\Sigma} (x, \trueparam)
\bigr] + 
\partial_{\theta, k_1} \partial_{\theta, k_2} 
\log \det \boldsymbol{\Sigma} (x, \trueparam) 
\Bigr\} 
\truedist(dx) 
= 2 [\Gamma (\trueparam)]_{k_1 k_2}, 
\end{align*}
where we applied the following formulae in the above equation. For $\theta \in \Theta$, 
\begin{gather*}
\partial_{\theta, k_1} \partial_{\theta, k_2} 
\log \det \boldsymbol{\Sigma} (x, \theta) 
= - \mathrm{Tr} \bigl[
\partial_{\theta, k_1} \partial_{\theta, k_2}  
\boldsymbol{\Lambda} (x, \theta) 
\, \boldsymbol{\Sigma}  (x, \theta ) \bigr] 
- \mathrm{Tr} \bigl[ 
\partial_{\theta, k_1} \boldsymbol{\Sigma} (x, \theta) \partial_{\theta, k_2} \boldsymbol{\Lambda} (x, \theta)  \bigr]; \\[0.2cm]
\mathrm{Tr} \bigl[ 
\partial_{\theta, k_1} \boldsymbol{\Sigma} (x, \theta) \partial_{\theta, k_2} \boldsymbol{\Lambda} (x, \theta) \bigr] 
 = - \mathrm{Tr} \bigl[ 
\partial_{\theta, k_1} 
\boldsymbol{\Sigma} (x, \theta) \, 
\boldsymbol{\Lambda} (x, \theta)
\partial_{\theta, k_2} \boldsymbol{\Sigma}  
(x, \theta) \,  
\boldsymbol{\Lambda} (x, \theta)   
\bigr]. 
\end{gather*} 
The proof of Lemma \ref{lemm:slln} is now complete. 
\subsection{Proof of limits (\ref{eq:hess_E}) \& (\ref{eq:fourth_E})} 
\label{sec:pf_clt_limits} 
\subsubsection{Proof of (\ref{eq:hess_E})}
We write for $(x, \theta) \in \mathbb{R}^N \times \Theta$, 
\begin{align*}
r (x, \theta) 
=
\begin{bmatrix} 
\mu_{S_1} (x_S, \beta_{S_1}) 
+ \tfrac{1}{2} \mathcal{L} \mu_{S_1} (x, \beta_S) + \tfrac{1}{6} \mathcal{L}^2 \mu_{S_1} (x, \theta)  \\[0.2cm]
\mu_{S_2} (x, \beta_{S_2}) 
+ \tfrac{1}{2} \mathcal{L} \mu_{S_2} (x, \theta)
\\[0.2cm] 
\mu_{R} (x, \beta_R)
\end{bmatrix}. 
\end{align*}
From the definition of $\xi_i^k (\trueparam), \, 1 \le i \le n, \, 1 \le k \le N_\theta$, given in (\ref{eq:xi}), we have:
\begin{itemize}
\item For $1 \le k \le N_{\beta}$, 
\begin{align} \label{eq:xi_case_1}
\xi_i^k (\theta)
& = \tfrac{1}{\sqrt{n}} 
\biggl\{ 
- 2 \, \mathbf{m}_{p, i} (\Delta_n, \trueparam)^\top
\boldsymbol{\Lambda}_{i-1} (\trueparam) 
\partial_{\theta, k} r (\sample{X}{i-1}, \trueparam) \nonumber  \\ 
& \qquad + \sum_{j_1, j_2  =1}^N R^{j_1 j_2}_{k} (\sqrt{\Delta_n}, \sample{X}{i-1}, \trueparam)
\mathbf{m}_{p. i}^{j_1}(\Delta_n, \trueparam)
\mathbf{m}_{p. i}^{j_2}(\Delta_n, \trueparam)  \nonumber \\
& \qquad 
+ \sum_{j =1}^N R^{j}_k (\sqrt{\Delta_n}, \sample{X}{i-1}, \trueparam)
\mathbf{m}_{p. i}^{j}(\Delta_n, \trueparam) 
+ R_k (\sqrt{\Delta_n}, \sample{X}{i-1}, \trueparam)
\biggr\}, 
\end{align} 
where $R^{j_1j_2}_k, \, R_k^j, \, R_k \in \mathcal{S}$. 
\item For $N_{\beta} + 1 \le k \le N_{\theta}$, 
\begin{align} \label{eq:xi_case_2}
\xi_i^k (\theta) 
& = \tfrac{1}{\sqrt{n}} 
\biggl\{ 
\mathbf{m}_{p, i} (\Delta_n, \trueparam)^\top 
\partial_{\theta, k} \boldsymbol{\Lambda}_{i-1} (\trueparam)
\mathbf{m}_{p, i} (\Delta_n, \trueparam)  
+ \partial_{\theta, k} \log \det \boldsymbol{\Sigma}_{i-1} (\trueparam) \nonumber \\ 
& \qquad + \sum_{j_1, j_2  =1}^N R^{j_1 j_2}_{k} (\sqrt{\Delta_n}, \sample{X}{i-1}, \trueparam)
\mathbf{m}_{p. i}^{j_1}(\Delta_n, \trueparam)
\mathbf{m}_{p. i}^{j_2}(\Delta_n, \trueparam)  \nonumber \\
& \qquad 
+ \sum_{j =1}^N R^{j}_k (\sqrt{\Delta_n}, \sample{X}{i-1}, \trueparam)
\mathbf{m}_{p. i}^{j}(\Delta_n, \trueparam) 
+ R_k (\sqrt{\Delta_n}, \sample{X}{i-1}, \trueparam) 
\biggr\}, 
\end{align}
where $R^{j_1j_2}_k, \, R_k^j, \, R_k \in \mathcal{S}$. 
\end{itemize}  
To simplify the notation, we write  
\begin{align*}
\mathbf{P}_{k_1 k_2} (\trueparam) 
\equiv 
\sum_{i = 1}^n \mathbb{E}_{\trueparam}
\bigl[  
 \xi_{i}^{k_1} (\trueparam) 
 \xi_{i}^{k_2} (\trueparam) 
| \mathcal{F}_{t_{i-1}} 
\bigr],  \qquad 1 \le k_1, k_2 \le N_{\theta},  
\end{align*} 
and then check its limit in the high-frequency observation regime. We note that from a similar argument used in the proof of Lemma \ref{lemma:aux_2}, it holds that  
\begin{align}
& \tfrac{1}{n} \sum_{i = 1}^n \sum_{j_1, j_2, j_3 = 1}^N R^{j_1j_2j_3} (1, \sample{X}{i-1}, \trueparam) 
\mathbf{m}_{p, i}^{j_1} (\Delta_n, \trueparam) 
\mathbf{m}_{p, i}^{j_2} (\Delta_n, \trueparam)  
\mathbf{m}_{p, i}^{j_3} (\Delta_n, \trueparam)  
\probconv 0; \label{eq:m_three}\\ 
& \tfrac{1}{n} \sum_{i = 1}^n \sum_{j_1, j_2, j_3, j_4 = 1}^N R^{j_1j_2j_3j_4} (1, \sample{X}{i-1}, \trueparam) 
\mathbf{m}_{p, i}^{j_1} (\Delta_n, \trueparam) 
\mathbf{m}_{p, i}^{j_2} (\Delta_n, \trueparam)  
\mathbf{m}_{p, i}^{j_3} (\Delta_n, \trueparam)  
\mathbf{m}_{p, i}^{j_4} (\Delta_n, \trueparam)  \nonumber  \\ 
& \probconv 
\sum_{j_1, j_2, j_3, j_4 = 1}^N
\int R^{j_1j_2j_3j_4} (1, x, \trueparam) 
\Bigl\{ 
[\boldsymbol{\Sigma} (x, \trueparam)]_{j_1 j_2} 
[\boldsymbol{\Sigma} (x, \trueparam)]_{j_3 j_4}  
+ 
[\boldsymbol{\Sigma} (x, \trueparam)]_{j_1 j_3} 
[\boldsymbol{\Sigma} (x, \trueparam)]_{j_2 j_4} \nonumber  \\
& \qquad \qquad \qquad \qquad + 
[\boldsymbol{\Sigma} (x, \trueparam)]_{j_1 j_4} 
[\boldsymbol{\Sigma} (x, \trueparam)]_{j_2 j_3}  
\Bigr\} 
\truedist (dx), \label{eq:m_four}
\end{align} 
as \limit, where $R^{j_1j_2j_3}, R^{j_1j_2j_3j_4} \in \mathcal{S}$.    
For $1 \le k_1, k_2 \le N_{\beta}$, we apply Lemmas \ref{lemma:aux_1}, \ref{lemma:aux_2} and the limits (\ref{eq:m_three}) and (\ref{eq:m_four}) to obtain: 
\begin{align*}
& \mathbf{P}_{k_1 k_2} (\trueparam) 
= \tfrac{4}{n} \sum_{i = 1}^n \sum_{j_1, j_2, j_3, j_4 = 1}^N 
\biggl\{ 
\mathbf{m}_{p, i}^{j_1} (\trueparam)[\boldsymbol{\Lambda}_{i-1} (\trueparam)]_{j_1 j_2} \partial_{\theta, k_1} r^{j_2} (\sample{X}{i-1}, \trueparam) \\ 
& \qquad \qquad \qquad \qquad \qquad  
\times \mathbf{m}_{p, i}^{j_3} (\trueparam)[\boldsymbol{\Lambda}_{i-1}(\trueparam)]_{j_3 j_4} \partial_{\theta, k_2} 
r^{j_4} (\sample{X}{i-1}, \trueparam) \biggr\}
+ o_{\mathbb{P}_{\trueparam}} (1)  \nonumber \\ 
& \probconv 4  \sum_{j_1, j_2, j_3, j_4 =1}^N 
\int [\boldsymbol{\Lambda}(x, \trueparam)]_{j_1 j_2} \partial_{\theta, k_1} r^{j_2} (x, \trueparam)  
[\boldsymbol{\Lambda}(x, \trueparam)]_{j_3 j_4} 
[\boldsymbol{\Sigma} (x, \trueparam)]_{j_1 j_3}
\partial_{\theta, k_2} r^{j_4} (x, \trueparam) 
\, \truedist(dx)  \nonumber \\
& = 4 
\int 
\partial_{\theta, k_1} r (x, \trueparam)^\top
[\boldsymbol{\Lambda}(x, \trueparam)]  
\partial_{\theta, k_2} r (x, \trueparam)  
\, \truedist(dx)  = 4 [\Gamma (\trueparam)]_{k_1 k_2},  
\end{align*}
as \limit. In particular in the last equation above, we used Lemmas \ref{lemm:mat_1}, \ref{lemm:mat_2} and \ref{lemma:matrix_3}. For $N_{\beta} + 1 \le k_1, k_2 \le N_\theta$, we have that if \limit, then  
\begin{align*}
& \mathbf{P}_{k_1 k_2} (\trueparam)  
= \tfrac{1}{n} \sum_{i = 1}^n \biggl\{
\sum_{j_1, j_2, j_3, j_4 = 1}^N 
[\partial_{\theta, k_1} \boldsymbol{\Lambda}_{i-1} (\trueparam)]_{j_1 j_2}
[\partial_{\theta, k_2} \boldsymbol{\Lambda}_{i-1} (\trueparam)]_{j_3 j_4} \prod_{q = 1}^4 
\mathbf{m}_{p, i}^{j_{q}}(\Delta_n, \trueparam) 
\\ 
& \qquad + \sum_{j_1, j_2 = 1}^N 
[\partial_{\theta, k_1} \boldsymbol{\Lambda}_{i-1} (\trueparam)]_{j_1 j_2} 
\partial_{\theta, k_2} \log \det \boldsymbol{\Sigma}_{i-1} 
(\trueparam) 
\mathbf{m}_{p, i}^{j_1}(\Delta_n, \trueparam) 
\mathbf{m}_{p, i}^{j_2} (\Delta_n, \trueparam) \nonumber \\ 
& \qquad + \sum_{j_1, j_2 = 1}^N 
\partial_{\theta, k_1} \log \det \boldsymbol{\Sigma}_{i-1} (\trueparam)
[\partial_{\theta, k_2} \boldsymbol{\Lambda}_{i-1} (\trueparam)]_{j_1 j_2}
\mathbf{m}_{p, i}^{j_1}(\Delta_n, \trueparam) 
\mathbf{m}_{p, i}^{j_2}(\Delta_n, \trueparam) 
\nonumber \\ 
& \qquad 
+ 
\partial_{\theta, k_1} \log \det \boldsymbol{\Sigma}_{i-1} (\trueparam)
\partial_{\theta, k_2} \log \det \boldsymbol{\Sigma}_{i-1} (\trueparam) 
\biggr\}  + o_{\mathbb{P}_{\trueparam}} (1) \nonumber \\ 
& \probconv \sum_{j_1, j_2, j_3, j_4 = 1}^N  
\int \Bigl\{ 
[\partial_{\theta, k_1}
\boldsymbol{\Lambda} (x, \trueparam)]_{j_1 j_2} 
[\partial_{\theta, k_2} \boldsymbol{\Lambda}
(x, \trueparam)]_{j_3 j_4} 
\times \bigl( 
[\boldsymbol{\Sigma} (x, \trueparam)]_{j_1 j_2} 
[\boldsymbol{\Sigma} (x, \trueparam)]_{j_3 j_4} 
\nonumber  \\
& \qquad \qquad \qquad \qquad + 
[\boldsymbol{\Sigma} (x, \trueparam)]_{j_1 j_3} 
[\boldsymbol{\Sigma} (x, \trueparam)]_{j_2 j_4}  
+ 
[\boldsymbol{\Sigma} (x, \trueparam)]_{j_1 j_4} 
[\boldsymbol{\Sigma} (x, \trueparam)]_{j_2 j_3}  
\bigr) \Bigr\}  \truedist(dx) \nonumber \\
& \qquad + \sum_{j_1, j_2 = 1}^N 
\int \Bigl\{ [\partial_{\theta, k_1}
\boldsymbol{\Lambda} (x, \trueparam)]_{j_1 j_2}
\partial_{\theta, k_2} \log \det \boldsymbol{\Sigma} (x,\trueparam) 
[ \boldsymbol{\Sigma} (x, \trueparam)]_{j_1 j_2}  \Bigr\}  \truedist(dx) \nonumber \\ 
& \qquad + \sum_{j_1, j_2 = 1}^N 
\int \Bigl\{ 
\partial_{\theta, k_1} \log \det \boldsymbol{\Sigma} (x, \trueparam) [\partial_{\theta, k_2}
\boldsymbol{\Lambda} (x, \trueparam)]_{j_1 j_2}
[ \boldsymbol{\Sigma} (x, \trueparam)]_{j_1 j_2} \Bigr\} \truedist(dx) \nonumber \\ 
& \qquad + \int \Bigl\{ 
\partial_{\theta, k_1} \log \det \boldsymbol{\Sigma} (x, \trueparam)  
\partial_{\theta, k_2} \log \det \boldsymbol{\Sigma} (x, \trueparam)  
\Bigr\} \truedist (dx)  \nonumber \\
& = \sum_{j_1, j_2, j_3, j_4 = 1}^N
\int \Bigl\{[\partial_{\theta, k_1}
\boldsymbol{\Lambda} (x, \trueparam)]_{j_1 j_2} 
[\partial_{\theta, k_2} \boldsymbol{\Lambda}
(x, \trueparam)]_{j_3 j_4}  
\bigl( 
[\boldsymbol{\Sigma} (x, \trueparam)]_{j_1 j_3} 
[\boldsymbol{\Sigma} (x, \trueparam)]_{j_2 j_4} \nonumber \\
& \qquad \qquad \qquad \qquad + 
[\boldsymbol{\Sigma} (x, \trueparam)]_{j_1 j_4} 
[\boldsymbol{\Sigma} (x, \trueparam)]_{j_2 j_3}  
\bigr)\Bigr\} \truedist (dx) \nonumber  \\ 
& = 4 [\Gamma (\trueparam)]_{k_1 k_2}, 
\end{align*}
where in the second and third equality, we made use of the following formulae. For $x \in \mathbb{R}^N$,  
\begin{gather*}
\partial_{\theta, k} 
\log \det \boldsymbol{\Sigma} (x, \trueparam) 
= - \sum_{j_1, j_2 = 1}^N 
[\partial_{\theta, k} \boldsymbol{\Lambda} (x, \trueparam)]_{j_1 j_2} 
[\boldsymbol{\Sigma} (x, \trueparam)]_{j_2 j_1}; \\ 
[\partial_{\theta, k} 
\boldsymbol{\Lambda} (x, \trueparam)]_{j_1 j_2}
= - \sum_{j_3, j_4 = 1}^N  
[\boldsymbol{\Lambda} (x, \trueparam) ]_{j_1 j_3}
[\partial_{\theta, k} \boldsymbol{\Sigma}
(x, \trueparam)]_{j_3 j_4} 
[\boldsymbol{\Lambda} (x, \trueparam)]_{j_4 j_2}, 
\end{gather*} 
for $N_{\beta} +  1 \le k \le N_{\theta}$ and $1 \le j_1, j_2 \le N$.  Finally, for the cases 
$1 \le k_1 \le N_{\beta}, \, 
N_{\beta} + 1 \le k_2 \le N_{\theta}$ or $1 \le k_2 \le N_{\beta}, \, 
N_{\beta} + 1 \le k_1 \le N_{\theta}$, we have from (\ref{eq:xi_case_1}), (\ref{eq:xi_case_2}) Lemmas \ref{lemma:aux_1}-\ref{lemma:aux_2} and the limit (\ref{eq:m_three}) that 
\begin{align*}
\mathbf{P}_{k_1 k_2} (\trueparam)  
\probconv 0, 
\end{align*}
as \limit. The proof of limit (\ref{eq:hess_E}) is
now complete. 
\subsubsection{Proof of limit (\ref{eq:fourth_E})}
From the similar argument in the proof of (\ref{eq:hess_E}), it is shown that if \limit, then
\begin{align*}
\sum_{i = 1}^n \mathbb{E}_{\trueparam}
\bigl[  
\bigl(  \xi_{i}^{k_1} (\trueparam) 
 \xi_{i}^{k_2} (\trueparam) \bigr)^2 
| \mathcal{F}_{t_{i-1}} 
\bigr] \probconv 0,  \qquad 1 \le k_1, k_2 \le N_\theta, 
\end{align*}
by noticing that the left hand side contains $1/n^2$. We omit the detailed proof. 
\subsection{Proof of Lemma \ref{lemma:F_2_conv}} 
\label{append:pf_F2_conv}
\noindent Let $1 \le i \le n$ and $p \ge 2$. Making use of 
\begin{align*} 
& \mathbb{E}_{\trueparam} 
\bigl[ 
\mathbf{m}_{p, i}^{l_1} (\Delta_n, \trueparam) 
\mathbf{m}_{p, i}^{l_2} (\Delta_n, \trueparam) | \mathcal{F}_{t_{i-1}} 
\bigr]  
= [\boldsymbol{\Xi}_{K_p, i-1} (\Delta_n, \trueparam)]_{l_1 l_2} + 
[\mathbf{C} (\Delta_n^{K_p + 1}, \sample{X}{i-1}, \trueparam)]_{l_1 l_2}, \quad 1 \le l_1, l_2 \le N, 
\end{align*}
for some $[\mathbf{C}]_{l_1 l_2} \in \mathcal{S}$, we have $F^{(2), k}_{i-1} (\trueparam) = F^{(2, \mrm{I}), k}_{i-1} (\trueparam)  + F^{(2, \mrm{II}), k}_{i-1} (\trueparam)$ with 
\begin{align*}
F^{(2, \mrm{I}), k}_{i-1} (\trueparam) 
& =  
[\mathbf{M}_n^{-1}]_{kk} 
\sum_{j = 0}^{K_p} \Delta_n^j
\Bigl\{ 
\mathrm{Tr}
\Bigl[ 
\bigl( \partial_{\theta, k} \mathbf{G}_{i-1, j} 
(\trueparam) \bigr) \, 
\boldsymbol{\Xi}_{K_p, i-1} (\trueparam)
\Bigr] 
+ \partial_{\theta, k} \mathbf{H}_{i-1, j} (\trueparam) \Bigr\}; \\
F^{(2, \mrm{II}), k}_{i-1} (\trueparam)  
& =
\tfrac{1}{n} R^{{(2, \mrm{II}), k}} \bigl(\sqrt{n \Delta_n^{2 K_p + 2}}, \sample{X}{i-1}, \trueparam \bigr),  
\end{align*} 
where $R^{{(2, \mrm{II}), k}} \in \mathcal{S}$. We will show that under the event $D_{\Delta_n, i-1} = \{\omega \in \Omega \, |  \, \det \boldsymbol{\Xi}_{K_p, i-1} (\Delta_n, \trueparam)  > 0 \}$, $1 \le i \le n$, it follows that  %
\begin{align} \label{eq:key}
\sum_{j = 0}^{K_p} \Delta_n^j
\Bigl\{ 
\mathrm{Tr} 
\Bigl[ 
\bigl( \partial_{\theta, k} \mathbf{G}_{i-1, j} 
(\trueparam) \bigr) \boldsymbol{\Xi}_{K_p, i-1} (\trueparam) 
\Bigr]
+ 
\partial_{\theta, k} \mathbf{H}_{i-1, j} (\trueparam) \Bigr\} 
= R
(\Delta_n^{K_p + 1}, \sample{X}{i-1}, \trueparam),  
\end{align}
where $R \in \mathcal{S}$. Due to the invertibility of the matrix $\boldsymbol{\Xi}_{K_p, i-1} (\Delta_n, \trueparam)$ under the event $D_{\Delta_n, i-1}$, we have 
\begin{align}
\partial_{\theta, k}
\log \det \boldsymbol{\Xi}_{K_p, i-1} (\Delta_n, \trueparam) 
& = \mrm{Tr} \bigl[ \bigl( \boldsymbol{\Xi}_{K_p, i-1} (\Delta_n, \trueparam) \bigr)^{-1} 
\partial_{\theta, k}
\boldsymbol{\Xi}_{K_p, i-1} (\Delta_n, \trueparam) \bigr]
\nonumber \\
& = - \mrm{Tr} \bigl[ \partial_{\theta, k}
\bigl( \boldsymbol{\Xi}_{K_p, i-1} (\Delta_n, \trueparam) \bigr)^{-1}  
\boldsymbol{\Xi}_{K_p, i-1} (\Delta_n, \trueparam) \bigr],  \quad 1 \le k \le N_\theta, \label{eq:deriv_logdet}
\end{align}
where in the last equality we have used $N = \mrm{Tr} [\bigl( \boldsymbol{\Xi}_{K_p, i-1} (\Delta_n, \trueparam) \bigr)^{-1}  
\boldsymbol{\Xi}_{K_p, i-1} (\Delta_n, \trueparam) ]$. Taylor expansion of $\bigl( \boldsymbol{\Xi}_{K_p, i-1} (\Delta_n, \trueparam) \bigr)^{-1}$ and $\log \det \boldsymbol{\Xi}_{K_p, i-1} (\Delta_n, \trueparam)$ at $\Delta_n = 0$ yield: 
\begin{align}
\Bigl[ \bigl( \boldsymbol{\Xi}_{K_p, i-1} (\Delta_n, \trueparam) \bigr)^{-1} \Bigr]_{l_1 l_2} 
& = \sum_{j  = 0}^{K_p} \Delta_n^j \,  
\bigl[ \mathbf{G}_{i-1, j} (\trueparam) \bigr]_{l_1 l_2} 
+ R^1_{l_1 l_2} (\Delta_n^{K_p + 1}, \sample{X}{i-1}, \trueparam), \quad  1 \le l_1, l_2 \le N; \label{eq:taylor_inv} \\  
\log \det \boldsymbol{\Xi}_{K_p, i-1} (\Delta_n, \trueparam) 
& = \sum_{j  = 0}^{K_p} \Delta_n^j \,  \mathbf{H}_{i-1, j} (\trueparam) + R^2 (\Delta_n^{K_p + 1}, \sample{X}{i-1}, \trueparam), \label{eq:taylor_det} 
\end{align} 
where $R^1_{l_1 l_2}, \, R^2 \in \mathcal{S}$ so that they are continuously differentiable w.r.t. $\theta \in \Theta$ and $\partial_{\theta, k} R^1_{l_1 l_2}, \, \partial_{\theta, k} R^2 \in \mathcal{S}$ for $1 \le k \le N_\theta$. Hence, (\ref{eq:key}) immediately follows from (\ref{eq:deriv_logdet}), (\ref{eq:taylor_inv}) and (\ref{eq:taylor_det}). 
\\

Thus, (\ref{eq:key}) gives
\begin{align*} 
F^{(2, \mrm{I}), k}_{i-1} (\trueparam)  
=
\begin{cases} 
\tfrac{1}{n} R^{(2, \mrm{I}), k}
\bigl( 
\sqrt{n \Delta_n^{2K_p + 2}}, \sample{X}{i-1}, \trueparam \bigr),  &  1 \le k \le N_{\beta_S}, \, N_{\beta} + 1 \le k \le N_\theta; \\
\tfrac{1}{n} R^{(2, \mrm{I}), k} ( \sqrt{n \Delta_n^{2 K_p + 1}}, \sample{X}{i-1}, \trueparam),  &  N_{\beta_S} + 1 \le k \le N_\beta,    
\end{cases} 
\end{align*}
where $R^{(2, \mrm{I}), k} \in \mathcal{S}$. 
Thus, under the event $\textstyle{ D_{\Delta_n} = \bigcap_{i = 1}^n D_{\Delta_n, i-1}}$, if \limit \, with $\Delta_n = o (n^{-1/p})$, then 
\begin{align*}
\sum_{i = 1}^n F^{(2), k}_{i-1} (\trueparam)
= \sum_{i = 1}^n \bigl\{ 
F^{(2, \mrm{I}), k}_{i-1} (\trueparam) 
+ F^{(2, \mrm{II}), k}_{i-1} (\trueparam)
\bigr\} \probconv 0,
\end{align*} 
and now the proof of Lemma \ref{lemma:F_2_conv} is complete.  
\section{
Supporting Material for Numerical Experiments}
\label{appendix:KF}
We provide the details to construct the marginal likelihood used in the numerical experiment of parameter estimation of FHN model (\ref{eq:fhn_model}) under a partial observation regime in Section \ref{sec:fhn_partial} in the main text. 
\\ 

To construct the marginal likelihood used in the computation of the MLE (\ref{eq:mle_partial}), we introduce (locally) Gaussian approximations associated with the contrast function $\ell_{p, n}^{\, (\mrm{I})} (\theta)$, $p=2, 3$, defined in (\ref{eq:contrast_I_p2}) and (\ref{eq:contrast_I}). We note that when $p=2$, the scheme corresponds to the local Gaussian (LG) scheme developed in \cite{glot:20, glot:21}. The one-step Gaussian approximations for $p=2, 3$ are written in the following form: 

\begin{align} \label{eq:condtionally_Gaussian}
{Z}_{t_{i+1}} 
= 
\begin{bmatrix}
   {X}_{t_{i+1}}  \\[0.1cm]
   {Y}_{t_{i+1}} 
\end{bmatrix} 
= 
a (\Delta_n, {X}_{t_{i}}, \theta) 
+ b (\Delta_n, {X}_{t_{i}}, \theta) Y_{t_{i}} 
+ c_p (\Delta_n, {X}_{t_{i}}, \theta) w_{i+1},  \qquad p = 2, 3, 
\end{align}
where $a, b, c_p : (0, \infty) \times \mathbb{R} \times \Theta \to \mathbb{R}^2$ are explicitly determined and $\{w_{i+1}\}_{i = 0, \ldots, n-1}$ is i.i.d. sequence of standard normal random variables. In particular, we have: 
\begin{align*} 
a (\Delta, x, \theta) 
& = 
\begin{bmatrix}
    &x + \tfrac{\Delta}{\varepsilon} (x - x^3 - s)  
    + \tfrac{\Delta^2 }{2 \varepsilon^2} (1 - 3 x^2) (x - x^3 - s)
    - \tfrac{\Delta^2}{2 \varepsilon} (\gamma x + \beta) \\[0.2cm] 
    & (\gamma x + \beta) \Delta
\end{bmatrix};  \\[0.3cm]
b (\Delta, x, \theta) 
& = 
\begin{bmatrix}
- \tfrac{\Delta}{\varepsilon} 
+  
\bigl\{-  (1 - 3 x^2 ) + \varepsilon \bigr\} \tfrac{\Delta^2}{2 \varepsilon^2} \\[0.2cm]
1 - \Delta  
\end{bmatrix},  
\end{align*} 
for $\Delta > 0$, $x \in \mathbb{R}$ and $\theta = (\varepsilon, \gamma, \alpha, \sigma)$. Also, the covariance matrix $\Sigma_p \equiv c_p c_p^\top, \, p =2, 3$, is defined as follows: 
%
\begin{align*}
\Sigma_2 (\Delta, x, \theta) 
& =
\begin{bmatrix}
\tfrac{\Delta^3}{3} \tfrac{\sigma^2}{\varepsilon^2} & - \tfrac{\Delta^2}{2} \tfrac{\sigma^2}{\varepsilon}  \\[0.2cm]
- \tfrac{\Delta^2}{2} \tfrac{\sigma^2}{\varepsilon} 
& \Delta \, \sigma^2  
\end{bmatrix};  \\[0.2cm]
\Sigma_3 (\Delta, x, \theta)  
& = \Sigma_2 (\Delta, x, \theta)  \\ 
&  
\quad + 
\begin{bmatrix}
\tfrac{\Delta^4}{4} L_1 (x, \theta) L_2 (x, \theta)
& \Delta^3 \bigl( \tfrac{1}{6} 
\sigma L_2 (x, \theta) 
+ \tfrac{1}{3} L_1 (x, \theta) L_3 (x, \theta) \bigr)  \\[0.2cm]
\Delta^3 \bigl( \tfrac{1}{6} 
\sigma L_2 (x, \theta) 
+ \tfrac{1}{3} L_1 (x, \theta) L_3 (x, \theta) \bigr) 
& 
\Delta^2 \sigma L_3 (x, \theta) 
\end{bmatrix}, 
\end{align*}
where we have set: 
\begin{align*}
L_1 (x, \theta) = - \tfrac{\sigma}{\varepsilon}, 
\quad 
L_2 (x, \theta) = \tfrac{\sigma}{\varepsilon^2} \times \bigl\{- 
(1 - 3 x^2) + \varepsilon
\bigr\}, 
\quad L_3 (x, \theta) = - \sigma.   
\end{align*}
Thus, given the observation $X_{t_i}$, the scheme (\ref{eq:condtionally_Gaussian}) is interpreted as a linear Gaussian model w.r.t.~the hidden component $Y_{t_{i}}$. For the linear Gaussian state space model (\ref{eq:condtionally_Gaussian}), the Kalman Filter (KF) recursion formula can be obtained. 
We recall $\Delta_n$ is the step-size of the observation and $n$ is the number of data. 
For the simplicity of notation, we write $Z_k = (X_k, Y_k), \, 0 \le k \le n$ instead of $Z_{t_k} = (X_{t_k}, Y_{t_k})$ and $Z_{0:n} \equiv \{ (X_k, Y_k) \}_{k = 0, \ldots, n}$. 
Then, it holds that 
$$
Y_{k} | X_{0:k} \sim \mathscr{N} (m_{p, k}, Q_{p, k}), 
\quad  0 \le k \le n, \ \  p = 2, 3,
$$ 
where the filtering mean $m_{p, k}$ and variance $Q_{p, k}$ is defined as follows (the derivation can be found in Appendix F in \cite{igu:23}): 
\begin{align*}
m_{p, k}  = \mu^Y_{p, k-1} 
+ \frac{\Lambda_{p, k-1}^{YX}}{\Lambda_{p, k-1}^{XX}} (X_{k} - \mu^X_{p, k-1}), \qquad 
Q_{p, k}  = \Lambda_{p, k-1}^{YY}
- \frac{\Lambda_{p, k-1}^{YX} 
\,  \Lambda_{p, k-1}^{XY}}{\Lambda_{p, k-1}^{XX}}, 
\end{align*}
where  
\begin{align*}
\begin{bmatrix}
    \mu^X_{p, k-1} \\[0.2cm] 
    \mu^Y_{p, k-1} 
\end{bmatrix}
&  = a (\Delta_n, X_{k-1}, \theta)
 + b (\Delta_n, X_{k-1}, \theta) m_{p, k-1};  \\[0.3cm]
\begin{bmatrix}
    \Lambda_{p, k-1}^{XX} & \Lambda_{p, k-1}^{XY} \\[0.2cm] 
    \Lambda_{p, k-1}^{XY}  & \Lambda_{p, k-1}^{YY} 
\end{bmatrix}
& = \Sigma_p (\Delta_n, X_{k-1}, \theta) + 
b (\Delta_n, X_{k-1}, \theta) \,  Q_{p, k-1} \, b (\Delta_n, X_{k-1}, \theta)^\top.  
\end{align*}  
Since it holds that  $X_k | X_{0:k-1} \sim \mathscr{N} (\mu^X_{p, k-1}, \Lambda_{p, k-1}^{XX})$, the marginal likelihood $f_{p, n} (X_{0:n} ; \theta), \, p = 2, 3$, is obtained as:
\begin{align*} 
f_{p, n} (X_{0:n} ; \theta) 
= \varphi_0 (X_0) \times \prod_{k = 1}^n \varphi  \bigl( 
 X_{k} \, ;  \, 
 \mu^X_{p, k-1}, \, \Lambda_{p, k-1}^{XX} 
 \bigr), 
\end{align*}
where $\varphi_0$ is the density of $X_0$ and $\varphi (\cdot \, ;  \, \mu, v)$ is the density of (scalar) Gaussian random variable with the mean $\mu$ and the variance $v$.   

\bibliographystyle{rss} 
\bibliography{Hypo}      

\end{document}